\documentclass[onefignum,onetabnum]{siamonline250211}
\usepackage{amssymb, amsfonts,leftindex}
\allowdisplaybreaks[1]

\usepackage{lipsum}
\usepackage{amsfonts}
\usepackage{graphicx}
\usepackage{epstopdf}
\usepackage{algorithmic}
\ifpdf
  \DeclareGraphicsExtensions{.eps,.pdf,.png,.jpg}
\else
  \DeclareGraphicsExtensions{.eps}
\fi

\usepackage{enumitem}
\setlist[enumerate]{leftmargin=.5in}
\setlist[itemize]{leftmargin=.5in}

\newsiamremark{remark}{Remark}
\newsiamremark{hypothesis}{Hypothesis}
\crefname{hypothesis}{Hypothesis}{Hypotheses}
\newsiamthm{claim}{Claim}
\newsiamremark{fact}{Fact}
\crefname{fact}{Fact}{Facts}

\headers{Variational posterior convergence}{ZU, S., Jia, J. and Meng, D.}

\title{Consistency of Variational Inference for Nonlinear Inverse Problems of Partial Differential Equations\thanks{Submitted to the editors May 29,2025.
\funding{
This research was partially funded by the National Key Research and Development Program of China (Grant No. 2022YFA1004100), the National Natural Science Foundation of China (Grant Nos. 12322116, 12271428, 12326606, and 42474139), the Fundamental and Interdisciplinary Disciplines Breakthrough Plan of the Ministry of Education of China (Grant No. JYB2025XDXM101), and the Major Projects of the National Natural Science Foundation of China (Grant Nos. 12090021 and 12090020).
}}}

\author{
	Shaokang Zu\thanks{School of Mathematics and Statistics,
	Xi'an Jiaotong University,
	Xi'an
        710049, China
	(\email{incredit1@stu.xjtu.edu.cn}).}
\and 
	Junxiong Jia\thanks{Corresponding author. 
	School of Mathematics and Statistics,
	Xi'an Jiaotong University,
	Xi'an
	710049, China
	(\email{jjx323@xjtu.edu.cn}).}
\and 
	Deyu Meng\thanks{
	School of Mathematics and Statistics and Ministry of  Education Key Lab of Intelligent Networks and Network Security, 
	Xi’an Jiaotong University, Pazhou Laboratory (Huangpu),  
	Xi’an, 710049, China
	(\email{dymeng@mail.xjtu.edu.cn}).}
}

\usepackage{amsopn}

\ifpdf
\hypersetup{
  pdftitle={Variational posterior convergence for inverse problems},
  pdfauthor={Zu, S., Jia, J. and Meng, D.}
}
\fi

\newsiamthm{condition}{Condition}
\crefname{condition}{Condition}{Conditions}

\newsiamthm{example}{Example}
\crefname{example}{Example}{Examples}
\newcommand{\abs}[1]{\lvert {#1} \rvert}
\newcommand{\norm}[1]{\lVert {#1} \rVert}
\newcommand{\pdt}[2]{\langle {#1},{#2} \rangle}
\newcommand{\im}{\mathrm{i}}
\newcommand{\mexp}[1]{\exp\Big\{{#1}\Big\}}
\newcommand{\bbra}[1]{\bigg({#1}\bigg)}
\newcommand{\set}[1]{\Big\{{#1}\Big\}}

\begin{document}
\maketitle
\begin{abstract}
We investigate the convergence rates of variational posterior distributions for statistical inverse problems involving nonlinear partial differential equations (PDEs). Departing from exact Bayesian inference, variational inference transforms the inference problem into an optimization problem by introducing variational sets. Based on a modified ``prior mass and testing'' framework, we propose general conditions for three categories of inverse problems: mildly ill-posed, severely ill-posed, and those with unknown model parameters. Concentrating on the variational sets comprising the restricted Gaussian or widely utilized Gaussian mean-field families, we demonstrate that for all three categories, the convergence rate can be decomposed into a true distribution term and a variational approximation term. Moreover, we illustrate that the true distribution term dominates the convergence rates, thereby substantiating the effectiveness of the variational inference method for inverse problems of PDEs. As specific examples, we examine a collection of non-linear inverse problems, including the Darcy flow problem, the inverse potential problem for a subdiffusion equation, and the inverse medium scattering problem. Besides, we show that our convergence rates are minimax optimal for these inverse problems.
\end{abstract}



\section{Introduction} 
Motivated by significant applications in radar imaging, seismic explorations, and many other domains, inverse problems of partial differential equations (PDEs) have undergone enormous development over the past few decades \cite{IntroPDE_haber2003learning}. As computational power continues to increase, researchers are not satisfied with obtaining just an estimated solution but pursue performing some statistical analysis based on uncertain information. The Bayesian inverse approach draws the attention of researchers because it transforms inverse problems into statistical inference problems and has provided a framework for analysing the uncertainties of the parameters interested in inverse problems of PDEs \cite{IntroPDE_stuart2010inverse,IntroPDE_Dashti2017}.
\par
To extract information from the posterior probability distribution, sampling methods such as Markov chain Monte Carlo (MCMC) are frequently utilized. Although MCMC is highly efficient and theoretically sound as a sampling method, its computational cost can become excessive for many applications. This is due to the need to calculate a computationally intensive likelihood function for inverse problems of PDEs \cite{Fichtner2011Book}. Consequently, variational Bayesian inference methods have emerged as a popular alternative. Given a variational set $\mathcal{Q}$ and a posterior distribution $\Pi(\cdot\vert D_N)$, variational Bayesian inference seeks to identify $\hat{Q}$, the closest approximation to $\Pi(\cdot\vert D_N)$ under the Kullback-Leibler divergence within $\mathcal{Q}$. The structure of these variational sets $\mathcal{Q}$ often allows variational Bayesian inference to operate much faster than MCMC while achieving similar approximation accuracy. Moreover, variational Bayesian inference generally scales better to cases involving large datasets and computationally intensive likelihood functions due to its optimization-based nature \cite{IntroVI_yang2020alpha}. This approach is increasingly favored in Bayesian inverse problems, as illustrated by recent studies \cite{IntroVI_blei2017variational,IntroVI_meng2023sparse,IntroVI_povala2022variational,IntroVI_jia2021variational} and their references. Recently, theoretical results for variational Bayesian inference have begun to emerge \cite{IntroVI_alquier2020concentration,IntroVI_yang2020alpha,zhang2020convergence,IntroVI_wang2019frequentist,IntroVI_ray2022variational}, which provide theoretical guarantees for variational Bayesian inference and its variants.
However, there remains limited theoretical understanding of variational Bayesian inference in inverse problems. In particular, to our knowledge, \textcolor{black}{few researches} have investigated the convergence rate of the variational posterior $\hat{Q}$ for nonlinear inverse problems. Compared to the classical statistical inference problems investigated in \cite{IntroVI_blei2017variational,IntroVI_zhang2018advances,zhang2020convergence}, statistical inverse problems of PDEs involve a complex nonlinear forward operator, necessitating nontrivial modifications of both inference and PDE theories. Due to the distinct features of the variational Bayesian approach, the methodologies developed for statistical inverse problems, as illustrated in \cite{IntroNonLinear_nickl2023bayesian,IntroNonLinear_giordano2020consistency}, can hardly be directly employed.
\par
Here, we approach the variational posterior $\hat{Q}$ as the solution to the optimization problem 
\[\mathop{\min}_{Q\in \mathcal{Q}} D(Q\Vert\Pi(\cdot|D_{N})),\] where $D(\cdot \Vert \cdot)$ represents the Kullback-Leibler (KL) divergence, and $\Pi(\cdot|D_{N})$ denotes the posterior distribution derived from the data $D_N$ within a natural statistical observation model of the forward map $\mathcal{G}$ (refer to \cref{GeneralSetting}). The primary goal of this study is to determine the convergence rate of the variational posterior $\hat{Q}$ towards the truth. We propose general conditions on the forward map $\mathcal{G}$, the prior and the variational class $\mathcal{Q}$ to describe this contraction. Assuming that the forward map $\mathcal{G}$ satisfies our regularity and conditional stability conditions, we demonstrate that, for a broad class of Gaussian process priors, the variational posterior $\hat{Q}$ converges to the true parameter at the rate specified by \[\varepsilon_N^2 + \frac{1}{N}\mathop{\inf}_{Q \in \mathcal{Q}}P_{0}^{(N)}D(Q\Vert \Pi_N(\cdot|D_N)).\]
The first term, $\varepsilon_N^2$, represents the convergence rate of the posterior $\Pi_N(\cdot|D_N)$ and is ordered as $N^{-a}$ for some $a > 0$. The second term represents the variational approximation error arising from the data-generating process $P_{0}^{(N)}$, which is induced by the true parameter. When the variational set $\mathcal{Q}$ comprises {\color{black} restricted} Gaussian measures or a mean-field variational class, we demonstrate that the variational approximation error is dominated by $\varepsilon_N^2$ (up to a logarithmic factor). Therefore, this implies that the convergence rates of the variational posterior distributions for nonlinear inverse problems can be same as the posterior distributions (up to a logarithmic factor). In \cref{ApplicationSection}, we derive convergence rates of variational posteriors for the Darcy flow problem and the inverse potential problem for a subdiffusion equation. We also prove that these rates are minimax optimal (up to a logarithmic factor).
\par
In addition to the general conditions established for ensuring the convergence of variational posteriors in nonlinear inverse problems, this paper also explores several additional aspects. The initial conditions on the forward map $\mathcal{G}$ require polynomial growth in the norm of the parameters for the constants involved (see \cref{VariationalConsistency}). Some inverse problems may not meet these requirements, such as the inverse medium scattering problem, which has been shown to have at most log-type conditional stability \cite{furuya2024consistency}. For this kind of problem, we devise a specific prior to allow the variational posterior to still contract to the true parameter under weaker conditions on the forward map $\mathcal{G}$. The inverse medium scattering problem fits into these settings, with its variational posterior demonstrated to converge to the true parameter at a minimax optimal rate. Besides, we also address inverse problems that include extra unknown parameters. Compared to the conditions on the forward map $\mathcal{G}$ proposed in \cref{VariationalConsistency}, we introduce conditions that include additional requirements related to these unknown parameters. Inspired by model selection methods, we determine the unknown parameters and simultaneously obtain the convergence rate of the variational posterior $\hat{Q}$ towards the true target parameter by optimizing the evidence lower bound. Our theorems are applied to the inverse potential problem for a subdiffusion equation with an unknown fractional order. 
\subsection*{Related work}
The theory behind Bayesian methods for linear inverse problems is now well-founded. Initial research into asymptotic behavior focused on conjugate priors \cite{IntroLinear_knapik2011bayesian}, and later extended to non-conjugate priors \cite{IntroLinear_ray2013bayesian}. See \cite{IntroLinear_knapik2013bayesian,IntroLinear_agapiou2014bayesian,IntroLinear_agapiou2013posterior,IntroLinear_jia2021posterior} for more references. In recent years, the theory of Bayesian nonlinear inverse problems has seen significant advancements. 
Nickl et al. \cite{IntroNonLinear_nickl2020convergence} present the theory for the convergence rate of maximum a posterior estimates with Gaussian process priors, providing examples involving the Darcy flow problem and the Schr\"{o}dinger equation. For
X-ray transforms, Monard et al. \cite{IntroNonLinear_monard2019efficient,IntroNonLinear_monard2021consistent} prove Bernstein–von Mises theorems for a large family of one-dimensional linear functionals of the target parameter, and show the convergence rate of the statistical error in the recovery of a matrix field with Gaussian process priors. Subsequently, Giordano and Nickl \cite{IntroNonLinear_giordano2020consistency} demonstrated the convergence rate of the posterior with Gaussian process priors for an elliptic inverse problem. In 2020, Abraham and Nickl \cite{IntroNonLinear_abraham2020statistical} showed that a statistical algorithm constructed in the paper recovers the target parameter in supremum-norm loss at a statistical optimal convergence rate of the logarithmic order. For the Schr\"{o}dinger equation, Nickl \cite{IntroNonLinear_nickl2020bernstein} established the Bernstein–von Mises theorem for the posterior distribution using a prior based on a specific basis. For additional references, see \cite{IntroNonLinear_nickl2023some,IntroNonLinear_bohr2022bernstein,IntroNonLinear_monard2021statistical} and the monograph \cite{IntroNonLinear_nickl2023bayesian}.
\par
Variational Bayes inference has seen extensive use across various fields. Recently, theoretical results start to surface. In \cite{IntroVI_wang2019frequentist}, Wang and Blei established Bernstein–von Mises type of results for parametric models. For other related references on theories on parametric variational Bayes inference, refer to \cite{IntroVI_blei2017variational}. Regarding nonparametric and high-dimensional
models that are more relevant to non-linear inverse problems, recent researches \cite{IntroVI_yang2020alpha,IntroVI_alquier2020concentration} investigated variational approximation to tempered posteriors.
In the context of variational Bayesian inference for usual posterior distributions, studies \cite{zhang2020convergence,IntroVI_pati2018statistical} achieved results comparable to those obtained for tempered posteriors. When it comes to inverse problems, some algorithmic developments can be found in \cite{IntroVI_jia2021variational,IntroVI_jin2012variational,IntroVI_guha2015variational}. For theoretical results of linear inverse problems, Randrianarisoa and Szabo \cite{IntroLinear_randrianarisoa2023variational} employed the inducing variable method to derive the contraction rates of variational posterior with Gaussian process priors around the true function. Concerning non-linear inverse problems, to our knowledge, we are the first in extending these theoretical results of variational posterior to non-linear inverse problem settings.
\subsection*{Organization}
The structure of this paper is as follows. In \cref{GeneralSetting}, we begin by introducing the notations and statistical settings for the inverse regression model. We also define some information distances and connect these distances to the $L^2$ norm. In \cref{VariationalConsistency}, we establish a theory that guarantees convergence rates of variational Bayesian inference with Gaussian priors for nonlinear inverse problems. In \cref{ApplicationSection}, we demonstrate the application of our theory to the Darcy flow problem and the inverse potential problem for a subdiffusion equation, showing that the convergence rates for these problems are minimax optimal. In \cref{SectionContractionRateUnstable}, we extend our theory to certain inverse problems that lack conditional stability, including an application to the inverse medium scattering problem. In \cref{SectionContractionExtra}, we present a theory based on model selection methods to address some inverse problems with extra unknown parameters. This theory ensures the determination of unknown parameters and the simultaneous convergence rate of the variational posterior $\hat{Q}$ to the true target parameter, with an application to the inverse potential problem for a subdiffusion equation with an unknown fractional order. All proofs are provided in the Supplementary Material \cite{Zu2024AoS}.

\section{General settings}\label{GeneralSetting}
\subsection{Basic notations}
For $\mathcal{X} \subset \mathbb{R}^d$ an open set, $L^2(\mathcal{X})$ denotes the standard Hilbert space of square integrable functions on $\mathcal{X}$ with respect to Lebesgue measure, with inner product 
$\pdt{\cdot}{\cdot}_{L^2}$ and $\norm{\ \cdot\ }_{L^2}$ norm where we omit $\mathcal{X}$ when no confusion may arise. $L^2_{\lambda}(\mathcal{X})$ denotes the
space of $\lambda$-square integrable function on $\mathcal{X}$, where $\lambda$ is a measure on $\mathcal{X}$.
\par
For a multi-index $i = (i_1, ..., i_d)$, $i_j = \mathbb{N}\cup\{0\}$. Let $D^i$ denote the $i$-th (weak) partial differential operator of order $\abs{i} = \sum_j i_j$. Then the Sobolev spaces are defined as
\begin{align*}
H^{\alpha}(\mathcal{X}) = \set{{f \in L^2(\mathcal{X}) : D^i f \in L^2(\mathcal{X}) \ \forall |i| \leq \alpha}},\quad \alpha \in \mathbb{N}
\end{align*}
normed by $\norm{f}_{H^{\alpha}(\mathcal{X})} = \sum_{\abs{i}\leq \alpha} \norm{D^i f}_{L^2(\mathcal{X})}.$
For a non-integer real number $\alpha >0$, one defines $H^{\alpha}$ by interpolation.
\par
The space of bounded and continuous functions on $\mathcal{X}$ is denoted by $C^0(\mathcal{X})$,
equipped with the supremum norm $\norm{\ \cdot\ }_{\infty}$. For $\alpha \in \mathbb{N}\cup\{0\}$, the space of $\alpha$-times differentiable functions on $\mathcal{X}$ can be similarly defined as
\begin{align*}
C^{\alpha}(\mathcal{X}) = \set{{f \in C^0(\mathcal{X}) : D^i f \in C^0(\mathcal{X}) \ \forall |i| \leq \alpha}},\quad \alpha \in \mathbb{N}
\end{align*}
normed by
$\norm{f}_{C^{\alpha}(\mathcal{X})} = \sum_{\abs{i}\leq \alpha} \norm{D^i f}_{\infty}.$ 
The symbol $C^{\infty}(\mathcal{X})$ denotes the set of all infinitely differentiable functions on $\mathcal{X}$.
For a non-integer real number $\alpha >0$, we say $f \in C^{\alpha}(\mathcal{X})$ if for all multi-indices $i$ with $\abs{i}\leq \lfloor \alpha \rfloor$, $D^if$ exists and is $\alpha - \lfloor \alpha \rfloor$-H$\ddot{\mathrm{o}}$lder continuous. The norm on the space $C^{\alpha}(\mathcal{X})$ for such $\alpha$ is given by
\begin{align*}
\norm{f}_{C^{\alpha}(\mathcal{X})} = \norm{f}_{C^{\lfloor \alpha \rfloor}(\mathcal{X})} + \sum_{\abs{i} = \lfloor \alpha \rfloor} \sup_{x,y\in\mathcal{X}, x\neq y}\frac{\abs{D^i f(x)-D^i f(y)}}{\abs{x-y}^{\alpha - \lfloor \alpha \rfloor}}. 
\end{align*}
\par
For any space $S(\mathcal{X})$ with norm $\norm{\ \cdot \ }_{S(\mathcal{X})}$, we will sometimes omit $\mathcal{X}$ in the
notation and denote the norm by $\norm{\ \cdot \ }_{S}$ when no confusion may arise. The notation $S_0(\mathcal{X})$ denotes the subspace $(S_0(\mathcal{X}),\norm{\ \cdot \ }_{S(\mathcal{X})})$, consisting of elements of $S(\mathcal{X})$ that vanish at $\partial\mathcal{X}$. The notation $S_c(\mathcal{X})$ denotes the subspace $(S_c(\mathcal{X}),\norm{\ \cdot \ }_{S(\mathcal{X})})$, consisting of elements of $S(\mathcal{X})$ that are compactly supported in $\mathcal{X}$. Similarly, $S_K(\mathcal{X})$ denotes the subspace $(S_K(\mathcal{X}),\norm{\ \cdot \ }_{S(\mathcal{X})})$, consisting of elements of $S(\mathcal{X})$ that are supported in a subset $K\subset \mathcal{X}$. We use $S(\mathcal{X})^*$ to denote the dual space of $S(\mathcal{X})$. $B_{S(\mathcal{X})}(r)$ denotes $\{x \in S(\mathcal{X}) : \norm{x}_{S(\mathcal{X})} < r\}$ for $r\geq0$.
\par
All preceding spaces and norms can be defined for vector fields $f : \mathcal{X} \rightarrow \mathbb{R}^{d'}$ with
standard modification of the norms, by requiring each of the coordinate functions $f_i(\cdot), i = 1,\dots,d'$, to belong to the corresponding space of real-valued maps. We denote preceding spaces $S(\mathcal{X})$ defined for vector fields by $S(\mathcal{X},\mathbb{R}^{d'})$.
\par
Throughout the paper, $C, c$ and their variants denote generic constants that are either universal or ``fixed'' depending on the context.
For $a, b \in\mathbb{R}$, let $a\vee b = \max(a,b)$ and $a\wedge b = \min(a,b)$. The relation $a\lesssim b$ denotes an inequality $a\leq Cb$, and the corresponding convention is used for $\gtrsim$. The relation $a \simeq b$ holds if both $a\lesssim b$ and $a\gtrsim b$ hold. We define $\mathbb{R}^+=\{x\in\mathbb{R}: x\geq 0\}$. For $x\in \mathbb{R}^+$, $\lceil x \rceil$ is the smallest integer no smaller than $x$ and $\lfloor x \rfloor$ is the largest integer no larger than $x$. Given a set $S$, $\abs{S}$ denotes its cardinality, and $\textbf{1}_{S}$ is  the associated indicator function. The notations $\mathbb{P}$ and $\mathbb{E}$ are used to denote generic probability and expectation respectively, whose distribution is determined from the context. Additionally, the notation $\mathbb{P}f$ also means expectation of $f$ under $\mathbb{P}$, that is $\mathbb{P}f = \int f d\mathbb{P}$.
\par
For any $\alpha > d/2$, $0< \beta < \alpha - d/2$ and $\mathcal{X}$ a bounded domain with smooth boundary, the Sobolev imbedding implies that $H^{\alpha}(\mathcal{X})$ embeds continuously into $C^{\beta}(\mathcal{X})$, with norm estimates
\begin{align}
\norm{f}_{\infty} \lesssim \norm{f}_{C^{\beta}} \lesssim \norm{f}_{H^{\alpha}}, \quad \forall f\in H^{\alpha}(\mathcal{X}).
\end{align}
We repeatedly use the inequalities
\begin{align}
\norm{fg}_{H^{\alpha}} & \lesssim \norm{f}_{H^{\alpha}}\norm{g}_{H^{\alpha}}, \qquad \alpha >d/2, \label{Sobolevinter1} \\
\norm{fg}_{H^{\alpha}} & \lesssim \norm{f}_{C^{\alpha}}\norm{g}_{H^{\alpha}}, \qquad \alpha \geq 0. \label{Sobolevinter2}
\end{align}
We also need the following interpolation inequality. For all $\alpha_1, \alpha_2 \geq 0$ and $\theta \in [0,1]$,
\begin{align}
\norm{f}_{H^{\theta\alpha_1 + (1-\theta)\alpha_2}} \lesssim \norm{f}_{H^{\alpha_1}}^{\theta}\norm{g}_{H^{\alpha_2}}^{1-\theta}, \qquad \forall f \in H^{\alpha_1}\cap H^{\alpha_2}.
\end{align} 

\subsection{Statistical settings} 
Let $(\mathcal{X},\mathcal{A})$ and $(\mathcal{Z},\mathcal{B})$ be measurable spaces equipped with probability measure $\lambda$ and Lebesgue measure respectively. Let $\mathcal{Z}$ be a bounded smooth domain in $\mathbb{R}^d$. Moreover, let $V$ be a vector space of fixed finite dimension $p_V \in \mathbb{N}$, with inner product $\pdt{\cdot}{\cdot}_V$ and norm $\norm{\ \cdot\ }_V$. For parameter $f$, we assume that the parameter spaces $\mathcal{F}$ is a Borel-measurable subspace of $L^2(\mathcal{Z},\mathbb{R})$. We define $G:\mathcal{F} \rightarrow L^2_{\lambda}(\mathcal{X},V)$ to be a measurable forward map. However, the parameter spaces $\mathcal{F}$ may not be linear spaces. In order to use Gaussian process priors that are naturally supported in linear spaces such as Sobolev spaces, we now consider a bijective reparametrization of $\mathcal{F}$ through a regular link function $\Phi$ as in \cite{IntroNonLinear_nickl2020convergence,IntroNonLinear_giordano2020consistency}. The parameter spaces can be reparametrized as \textcolor{black}{$\mathcal{F} = \{{f_{\theta} = \Phi(\theta)}\, : \, \theta \in \Theta\}$} where $\Theta$ is a subspace of $L^2(\mathcal{Z},\mathbb{R})$. For the forward map $G:\mathcal{F} \rightarrow L^2_{\lambda}(\mathcal{X},V)$, we define the reparametrized forward map $\mathcal{G}$ by 
\[\mathcal{G}(\theta) = G(\Phi(\theta)), \qquad \forall \ \theta \in \Theta,\] 
and consider independent and identically distributed (i.i.d.) random variables $(Y_i,X_i)_{i=1}^N$ of the following random design regression model
\begin{align}\label{model}
    Y_i = \mathcal{G}({\theta})(X_i) + \varepsilon_i, \quad \varepsilon_i \mathop{\sim}^{iid} N(0,I_V), \quad i = 1, 2, \dots, N,
\end{align}
with the diagonal covariance matrix $I_V$.
The joint law of the random variables $(Y_i,X_i)_{i=1}^N$ defines a product probability measure on $(V\times\mathcal{X})^N$ and we denote it by 
$P_{\theta}^{(N)}=\bigotimes_{i=1}^NP_{\theta}^i$ where $P_{\theta}^i=P_{\theta}^1$ for all $i$. We write $P_{\theta}$ for the law of a copy $(Y,X)$ which has probability density 
\begin{align}\label{modeldensity}
    \frac{dP_{\theta}}{d\mu}(y,x) \equiv p_{\theta}(y,x) = \frac{1}{(2\pi)^{p_V/2}}\mexp{-\frac{1}{2}\abs{y-\mathcal{G}_{\theta}(x)}_V^2}
\end{align}
for dominating measure $d\mu = dy \times d\lambda$ where $dy$ is Lebesgue measure on $V$. We also represent the full data with the notation
\[D_N := \set{(Y_i,X_i) : i = 1, 2, \dots, N}, \qquad N\in \mathbb{N}.\]
\par
In order to give the result about variational posterior, we introduce some information distances and relate theses distances on the laws $\{P_{\theta} : \theta \in \Theta\}$ to the forward map $\mathcal{G}({\theta})$.
With $\rho>0$ and $\rho \neq 1$, the $\rho$-Rényi divergence between two probability measures $P_1$ and $P_2$ is defined as
\begin{align*}
    D_{\rho}(P_1\Vert P_2) = \left\{ \begin{aligned}
                 &\frac{1}{\rho - 1} \log \int \bbra{\frac{dP_1}{dP_2}}^{\rho-1} dP_1  &\text{if}\ P_1 \ll P_2,\\
                 &+ \infty  &\text{otherwise}.
            \end{aligned} \right.
\end{align*}
When $\rho \rightarrow 1$, the Rényi divergence converges to the KL divergence, defined as
\begin{align*}
    D(P_1\Vert P_2) = \left\{ \begin{aligned}
                &\int \log\bbra{\frac{dP_1}{dP_2}} dP_1 & \text{if}\ P_1 \ll P_2,\\
                 &+ \infty & \text{otherwise}.
            \end{aligned} \right.
\end{align*}
Moreover, the Rényi divergence $D_{\rho}(P_1||P_2)$ is a non-decreasing function of $\rho$, which gives
\[D(P_1\Vert P_2) \leq D_{2}(P_1\Vert P_2). \]
The Hellinger distance $h$ is defined as 
\begin{align*}
    h^2(P_1,P_2) = \frac{1}{2}\int\bbra{\sqrt{dP_1}-\sqrt{dP_2}}^2.
\end{align*}
\par
The following proposition relates these information distances on the laws $\{P_{\theta} : \theta \in \Theta\}$ to the $L^2_{\lambda}(\mathcal{X},V)$ norm, assuming $\mathcal{G}(\theta)$ is uniformly bounded.
\begin{proposition}\label{le2.1}
    Suppose that for a subset $\Theta \subset L^2(\mathcal{Z},\mathbb{R})$ and some finite constant $U = U_{\mathcal{G},\Theta} > 0$, we have
    \[\mathop{\sup}_{\theta \in \Theta}\Vert\mathcal{G}(\theta)\Vert_{\infty} \leq U.\]
    For the model density from \cref{modeldensity}, we have for every $\theta_1,\theta_2 \in \Theta$,
    \begin{gather}
         D_2(P_{\theta_1}\Vert P_{\theta_2}) \leq e^{4U^2} \norm{\mathcal{G}(\theta_1)-\mathcal{G}(\theta_2)}_{L^2_{\lambda}(\mathcal{X},V)}^2,\label{D2} \\
        D(P_{\theta_1}\Vert P_{\theta_2}) = \frac{1}{2} \norm{\mathcal{G}(\theta_1)-\mathcal{G}(\theta_2)}_{L^2_{\lambda}(\mathcal{X},V)}^2, \label{KL}
    \end{gather}
    and for $\textcolor{black}{C_{\mathcal{G}}} = \frac{1-e^{-U^2/2}}{2U^2}$,
    \begin{align}\label{hellinger}
        \textcolor{black}{C_{\mathcal{G}}}\Vert\mathcal{G}(\theta_1)-\mathcal{G}(\theta_2)\Vert_{L^2_{\lambda}(\mathcal{X},V)}^2 \leq h^2(p_{\theta_1},p_{\theta_2}) \leq \frac{1}{4}\Vert\mathcal{G}(\theta_1)-\mathcal{G}(\theta_2)\Vert_{L^2_{\lambda}(\mathcal{X},V)}^2.
    \end{align}
\end{proposition}
Now we introduce the variational posterior distribution. Let prior $\Pi$ be a Borel probability measure on $\Theta$. The posterior distribution  $\Pi_N(\cdot|D_N) = \Pi_N(\cdot|(Y_i,X_i)_{i=1}^N)$ of $\theta|(Y_i,X_i)_{i=1}^N$ on $\Theta$ arising from the data in model \cref{model} is given by
\begin{align}\label{Post}
    d\Pi(\theta|D_N) = \frac{\prod^{N}_{i=1}p_{\theta}(Y_i,X_i)d\Pi(\theta)}{\int_{\Theta}\prod^{N}_{i=1}p_{\theta}(Y_i,X_i)d\Pi(\theta)}.
\end{align}
To address computational difficulty of forward problems in posterior distributions, variational inference aim to find the closest element to the posterior distribution in a variational set $\mathcal{Q}$ consisting of probability measures. The most popular definition of variational inference is given through the KL-divergence. The variational posterior is defined as 
\begin{align}\label{variationalposterior}
    \hat{Q} = \mathop{\mathrm{argmin}}_{Q\in \mathcal{Q}} D(Q\Vert \Pi(\theta|D_{N})).
\end{align}
The choice of variational set $\mathcal{Q}$ usually determines the effect of variational posterior. The variational posterior from the variational set $\mathcal{Q}$ can be regarded as the projection of the true posterior onto $\mathcal{Q}$ under KL-divergence. When $\mathcal{Q}$ is large enough (even consisting of all probability measures), $\hat{Q}$ is exactly the true posterior $\Pi(\theta|D^{N})$. However, a larger $\mathcal{Q}$ often leads to a higher computational cost of the optimization problem \cref{variationalposterior}. Thus, it is of significance to choose a proper variational set $\mathcal{Q}$ in order to reduce the computational difficulty and give a good approximation of posterior simultaneously. \textcolor{black}{We note that the posterior under the Gaussian prior is not Gaussian due to the nonlinearity of the forward map $\mathcal{G}$. In \cref{VariationalConsistency}, we will give the convergence rate of variational posteriors when the non-Gaussian posterior is approximated by the variational set $\mathcal{Q}$ containing Gaussian measures.} 

\section{Variational posterior consistency theorem}\label{VariationalConsistency}
In this section, assuming that the data $D_N$ are generated through the model \cref{model} of law $P^{(N)}_{\theta_0}$ and the forward map $\mathcal{G}$ to satisfy certain conditions, we will show that the variational posterior distribution arising from certain Gaussian priors concentrates to any sufficiently regular ground truth $\theta_0$ (or, equivalently, $f_0 = \Phi(\theta_0)$), and show the rate of this contraction. \textcolor{black}{The proof of the main theorems (\cref{mainthm,mainthmsv}) uses Theorem 2.1 in \cite{zhang2020convergence} combined with the technique from \cite{IntroNonLinear_nickl2023bayesian}.}
\subsection{Conditions for the forward map}
In this part, we give the regularity and conditional stability conditions which resemble the conditions in \cite{IntroNonLinear_nickl2023bayesian}. The following regularity condition requires the forward map $\mathcal{G}$ to be uniformly bounded and Lipschitz continuous. Compared to the conditions in \cite{IntroNonLinear_nickl2023bayesian},  we further require polynomial growth in parameter's norm of uniform upper bound and the Lipschitz constants.
\begin{condition}\label{condreg}
     Consider a parameter space $\Theta \subseteq L^2(\mathcal{Z},\mathbb{R})$. 
  The forward map $\mathcal{G}:\Theta \rightarrow L^2_{\lambda}(\mathcal{X},V)$ is measurable. 
  Suppose for some normed linear subspace $(\mathcal{R},\Vert \cdot\Vert _{\mathcal{R}})$ of $\Theta$ and all $M>1$, there exist finite constants $C_U > 0$, $C_L > 0$ , ${\kappa}\geq0$, $p\geq 0$ and $l \geq 0$ such that 
  \begin{gather}
          \mathop{\sup}_{\theta \in \Theta \cap B_{\mathcal{R}}(M)}\mathop{\sup}_{x \in \mathcal{X}}\abs{\mathcal{G}(\theta)(x)}_V\leq C_{U}M^p,\label{bound}\\
          \norm{\mathcal{G}(\theta_1)-\mathcal{G}(\theta_2)}_{L^{2}_{\lambda}(\mathcal{X},V)} \leq C_L(1+\norm{\theta_1}^l_{\mathcal{R}}\vee \norm{\theta_2}^l_{\mathcal{R}})\norm{\theta_1-\theta_2} _{(H^{{\kappa}})^*}, \quad \theta_1,\theta_2 \in \mathcal{R}.\label{lip}
      \end{gather}
\end{condition}
Similarly, the conditional stability condition is given below.
\begin{condition}\label{condstab}
    Consider a parameter space $\Theta \subseteq L^2(\mathcal{Z},\mathbb{R})$. 
  The forward map $\mathcal{G}:\Theta \rightarrow L^2_{\lambda}(\mathcal{X},V)$ is measurable. 
  Suppose for some normed linear subspace $(\mathcal{R},\Vert \cdot\Vert _{\mathcal{R}})$ of $\Theta$ and all $M>1$, there exist a \textcolor{black}{non-constant function} $F: \mathbb{R}^{+} \rightarrow \mathbb{R}^{+}$ and finite constants $C_T > 0$, $q\geq 0$ such that
        \begin{align}\label{stab}
              F(\Vert f_{\theta} - f_{\theta_0}\Vert_{\mathcal{F}} ) \leq C_TM^q\Vert \mathcal{G}(\theta)-\mathcal{G}(\theta_0)\Vert _{L^{2}_{\lambda}(\mathcal{X},V)}, \forall \theta \in \Theta \cap B_{\mathcal{R}}(M).
          \end{align}
\end{condition}
These regularity and stability conditions on the forward map $\mathcal{G}$ require explicit estimates on growth rates of coefficients. We will verify these conditions for the inverse problems related to the Darcy flow equation and a subdiffusion equation in \cref{ApplicationSection}.
\subsection{Rescaled Gaussian prior}
Consider a centred Gaussian Borel probability measure $\Pi'$. To deal with the non-linearity of the inverse problem, we assume $\Pi'$ {\color{black} satisfies} some conditions.
\begin{condition}\label{prior}
    Let $\Pi'$ be a centred Gaussian Borel probability measure on the linear space $\Theta \subset L^2(\mathcal{Z})$ with reproducing
kernel Hilbert spaces (RKHS) $\mathcal{H}$. Suppose further that $\Pi^{'}(\mathcal{R}) = 1$ for some separable normed linear subspace $(\mathcal{R},\Vert \cdot\Vert _{\mathcal{R}})$ of $\Theta$.
\end{condition}
We build the priors $\Pi_N$ through a $N$-dependent rescaling step to $\Pi'$ as in \cite{IntroNonLinear_giordano2020consistency}.
The prior $\Pi = \Pi_N$ arises from the base prior $\Pi'$ defined as the law of
\begin{align}\label{rescaledprior}
    \theta = N^{-d/(4\alpha+4{\kappa}+2d)}\theta',
\end{align}
where $\theta' \sim \Pi'$,  $\alpha \geq 0$ represents the ``regularity'' of the ground truth $\theta_0$, and ${\kappa}$ represents the ``forward smoothing degree'' of $\mathcal{G}$ in \cref{condreg}.
Again, $\Pi_N$ defines a centered Gaussian prior on $\Theta$ and its RKHS $\mathcal{H}_N$ is still $\mathcal{H}$ but with the norm
\begin{align*}
    \norm{\theta}_{\mathcal{H}_N} = N^{d/(4\alpha+4{\kappa}+2d)}\norm{\theta}_{\mathcal{H}}, \quad \forall \theta \in \mathcal{H}.
\end{align*}
We first show the variational posterior converges to the ground truth $\theta_0$ at an explicit rate.
\begin{theorem} \label{mainthm}
  Suppose \cref{condreg} holds for the forward map $\mathcal{G}$, the separable normed linear subspace $(\mathcal{R},\Vert \cdot\Vert _{\mathcal{R}})$ and the finite constants $C_U > 0$, $C_L > 0$ , ${\kappa}\geq0$, $p\geq 0$, $l \geq 0$.
  Let $\Pi^{'}$ satisfy \cref{prior} for this $\mathcal{R}$ with $RKHS$ $\mathcal{H}$. Suppose that for some integer $\alpha \geq 0$, $\mathcal{H}$ satisfies the continuous embedding
  \[ \mathcal{H}\subset H^{\alpha}_c(\mathcal{Z})\text{ if }{\kappa}\geq 1/2\quad \text{ or }\quad \mathcal{H}\subset H^{\alpha}(\mathcal{Z})\text{ if }{\kappa} < 1/2.\]
  Denote by $\Pi_N$ the rescaled prior as in \cref{rescaledprior}.
  $\Pi_N(\cdot|(Y_i,X_i)_{i=1}^N) = \Pi_N(\cdot|D_N)$ is the corresponding posterior distribution in \cref{Post} arising from the data in the model \cref{model}.
  Assume that $\alpha + {\kappa} \geq \frac{d(l+1)}{2}$ and $\theta_0\in\mathcal{H} \cap \mathcal{R}$. Then, for $\varepsilon_N=N^{-\frac{\alpha+{\kappa}}{2\alpha+2{\kappa}+d}}$, 
  \[\gamma_N^2 = \frac{1}{N}\mathop{\inf}_{Q \in \mathcal{Q}}P_{\theta_0}^{(N)}D(Q\Vert \Pi_N(\cdot|D_N)),\]
  and variational posterior $\hat{Q}$ defined in \cref{variationalposterior}, we have
  \begin{align}
      P_{\theta_0}^{(N)}\hat{Q}\Vert \mathcal{G}(\theta)-\mathcal{G}(\theta_0)\Vert _{L^{2}_{\lambda}}^{\frac{2}{p+1}}\lesssim \varepsilon_N^{\frac{2}{p+1}} +\gamma^2_N \cdot \varepsilon_N^{-\frac{2p}{p+1}}.
  \end{align}
  Moreover, assume that \cref{condstab} also holds for $\mathcal{G}$, $\mathcal{R}$, function $F$ and finite constants $C_T > 0$, $q\geq 0$. Then, we further have
  \begin{align}
      P_{\theta_0}^{(N)}\hat{Q}[F(\Vert f_{\theta} - f_{\theta_0}\Vert_{\mathcal{F}} )]^{\frac{2}{p+q+1}}\lesssim \varepsilon_N^{\frac{2}{p+q+1}} +\gamma^2_N \cdot \varepsilon_N^{-\frac{2p+2q}{p+q+1}}.
  \end{align}
\end{theorem}
The convergence rate can be represented as $(\varepsilon_N^2+\gamma_N^2) \cdot \varepsilon_N^{-\eta}$ for some nonnegative $\eta < 2$. We see that the convergence rate is controlled by two terms $\varepsilon_N^2$ and $\gamma_N^2$. The first term $\varepsilon_N^2$ is the convergence rate of the true posterior $\Pi(\theta|D_N)$ and it can reach the same rate in \cite{IntroNonLinear_nickl2023bayesian} which is optimal in minimax sense for several inverse problems. The second term $\gamma_N^2$ characterizes the approximation error given by the variational set $\mathcal{Q}$. A larger $\mathcal{Q}$ leads to a faster rate $\gamma_N^2$. To use \cref{mainthm} in specific problems, we would like to choose a proper $\mathcal{Q}$ such that the approximation error reaches the posterior convergence rate, that is $\gamma_N^2\lesssim \varepsilon_N^2$. Similar to the illustration shown in \cite{zhang2020convergence}, 
we could derive a general upper bound of $\gamma_N^2$. For {\color{black}the} reader's convenience, we exhibit the simple derivation as follows:
\begin{align*}
N\gamma_N^2 & \leq P_{\theta_0}^{(N)}D(Q\Vert \Pi_N(\cdot|D_N)) = D(Q\Vert \Pi_N) + Q\Big[\int\log\Big(\frac{dP_{\Pi}^{(N)}}{dP^{(N)}_{\theta}}\Big)dP^{(N)}_{\theta_0}\Big] \\
& = D(Q\Vert \Pi_N) + Q[D(P^{(N)}_{\theta_0}\Vert P^{(N)}_{\theta})-D(P^{(N)}_{\theta_0}\Vert P^{(N)}_{\Pi})] \\
& \leq D(Q\Vert \Pi_N) + QD(P^{(N)}_{\theta_0}\Vert P^{(N)}_{\theta}),
\end{align*}
where $Q\in \mathcal{Q}$ and $P^{(N)}_{\theta} = \int P^{(N)}_{\theta} d\Pi_N(\theta)$.
Then we obtain the following upper bound
\[\gamma_N^2 \leq \mathop{\inf}_{Q \in \mathcal{Q}}R(Q),\quad R(Q): = \frac{1}{N}(D(Q\Vert \Pi) + Q[D(P^{(N)}_{\theta_0}\Vert P^{(N)}_{\theta})]).\]
In the following theorem, we provide a possible estimate for $\mathop{\inf}_{Q \in \mathcal{Q}}R(Q)$.
\begin{theorem} \label{boundgam}
    Suppose that the forward map $\mathcal{G}$ satisfies \cref{condreg} and the Gaussian process prior $\Pi_N$ defined as in \cref{mainthm}. Then, for $\theta_0\in\mathcal{H} \cap \mathcal{R}$ and $\varepsilon_N=N^{-\frac{\alpha+{\kappa}}{2\alpha+2{\kappa}+d}}$, there exists a {\color{black} restricted} Gaussian measure $Q_N$ such that
    \begin{align}
        R(Q_N) \lesssim \varepsilon_N^2.
    \end{align}
\end{theorem}
From \cref{boundgam}, if the variational set $\mathcal{Q}$ contains the {\color{black} restricted} Gaussian measure $Q_N$, we can give a bound
$\gamma_N^2 \leq \mathop{\inf}_{Q \in \mathcal{Q}}R(Q) \lesssim \varepsilon_N^2.$  Thus, the approximation error contracts at the same rate as the posterior distribution, leading to the subsequent theorem.
\begin{theorem}\label{finalthm}
    Suppose \cref{condreg} holds for the forward map $\mathcal{G}$, the separable normed linear subspace $(\mathcal{R},\Vert \cdot\Vert _{\mathcal{R}})$ and the finite constants $C_U > 0$, $C_L > 0$ , ${\kappa}\geq0$, $p\geq 0$ and $l \geq 0$.
Let $\Pi^{'}$ satisfy \cref{prior} for this $\mathcal{R}$ with $RKHS$ $\mathcal{H}$. Suppose that for some integer $\alpha \geq 0$, $\mathcal{H}$ satisfies the continuous embedding
\[ \mathcal{H}\subset H^{\alpha}_c(\mathcal{Z})\text{ if }{\kappa}\geq 1/2\quad \text{ or }\quad \mathcal{H}\subset H^{\alpha}(\mathcal{Z})\text{ if }{\kappa} < 1/2.\]
Denote by $\Pi_N$ the rescaled prior as in \cref{rescaledprior}.
$\Pi_N(\cdot|(Y_i,X_i)_{i=1}^N) = \Pi_N(\cdot|D_N)$ is the corresponding posterior distribution in \cref{Post} arising from the data in the model \cref{model}.
Assume that $\alpha + {\kappa} \geq \frac{d(l+1)}{2}$ and $\theta_0\in\mathcal{H} \cap \mathcal{R}$. Then, for $\varepsilon_N=N^{-\frac{\alpha+{\kappa}}{2\alpha+2{\kappa}+d}}$ and the variational posterior $\hat{Q}$ defined in \cref{variationalposterior} with any variational set $\mathcal{Q}$ containing the {\color{black} restricted} Gaussian measure $Q_N$ introduced in \cref{boundgam}, we have
\begin{align}
  P_{\theta_0}^{(N)}\hat{Q}\Vert \mathcal{G}(\theta)-\mathcal{G}(\theta_0)\Vert _{L^{2}_{\lambda}}^{\frac{2}{p+1}}\lesssim \varepsilon_N^{\frac{2}{p+1}}.
\end{align}
Moreover, assume that \cref{condstab} also holds for $\mathcal{G}$, $\mathcal{R}$, function $F$ and finite constants $C_T > 0$, $q\geq 0$. Then, we further have
\begin{align}
  P_{\theta_0}^{(N)}\hat{Q}[F(\Vert f_{\theta} - f_{\theta_0}\Vert_{\mathcal{F}} )]^{\frac{2}{p+q+1}}\lesssim \varepsilon_N^{\frac{2}{p+q+1}}.
\end{align}
\end{theorem}
     Our results suggest that the convergence rates of variational posterior distributions for non-linear inverse problems can attain the same convergence rates as those of posterior distributions, which are demonstrated to be optimal in a minimax sense for various inverse problems in \cite{IntroNonLinear_nickl2023bayesian}. We highlight that the probability measure $Q_N$ presented in \cref{boundgam} is a {\color{black} restricted} Gaussian measure. Actually, the construction of $Q_N$ in the proof of \cref{boundgam} (see the Supplementary Material \cite{Zu2024AoS}) essentially requires the variational family to contain the prior restricted to a Rényi ball centered at the true parameter. Without prior knowledge of the true parameter or the radius, such a family is not feasible for practical variational inference. Hence, \cref{boundgam} should be regarded as an existence result rather than a practically implementable setting. To solve the optimization \cref{variationalposterior}, it may be more practical to employ a variational set $\mathcal{Q}$ composed of Gaussian measures rather than such {\color{black} restricted} Gaussian measures $Q_N$. Consequently, we introduce high-dimensional Gaussian sieve priors to provide contraction theorems with a variational set $\mathcal{Q}$ of Gaussian mean-field families in \cref{High-dimensional Gaussian sieve priors}.
     \begin{remark}
    {\color{black} For the general results stated in \cref{finalthm} (also \cref{finalthmsv}), the presence of the power terms in the loss function, e.g. $\norm{\, \cdot\, }^
{\frac{2}{p+1}}$, is to the best of our current knowledge, required in order to relate the Hellinger distance 
$h(p_{\theta}, p_{\theta_0})$ to either 
$\lVert \mathcal{G}(\theta) - \mathcal{G}(\theta_0) \rVert_{L^2_{\lambda}}$ 
or 
$F\bigl(\lVert f_{\theta} - f_{\theta_0} \rVert_{\mathcal{F}}\bigr)$ (see the proof of \cref{mainthm} for detail). 
Through Markov’s inequality, we can convert our power bounds into
‘standard’ risk with squared distances (e.g., $L^2$ norm):
\begin{align*}
    &P_{\theta_0}^{(N)}\hat{Q}\bbra{\Vert \mathcal{G}(\theta)-\mathcal{G}(\theta_0)\Vert _{L^{2}_{\lambda}}\geq M_N\varepsilon_N} \rightarrow 0,\\
    &P_{\theta_0}^{(N)}\hat{Q}\bbra{F(\Vert f_{\theta} - f_{\theta_0}\Vert_{\mathcal{F}} )\geq M_N\varepsilon_N} \rightarrow 0,
\end{align*}
for any diverging sequence $M_N\rightarrow \infty$, which aligns with the posterior contraction in terms of standard squared distances (such as the results in \cite{knapik_general_2018}).
}
\end{remark}
    \par
\subsection{High-dimensional Gaussian sieve priors} \label{High-dimensional Gaussian sieve priors}
The settings of high-dimensional Gaussian sieve priors in this section follow those in \cite[Section 2.2.2]{IntroNonLinear_giordano2020consistency}, and we record them here for convenience.
First, we introduce an orthonormal wavelet basis of the Hilbert space $L^2(\mathbb{R}^d)$ as
\begin{align} \label{Dbase}
\left\{\psi_{lr}: r \in \mathbb{Z}^d, l \in \mathbb{N}\cup\{-1,0\}\right\}
\end{align}
composed of sufficiently regular, compactly supported Daubechies wavelets (for details see Chapter 4 in \cite{gin2015mathematical}). 
For a bounded smooth domain $\mathcal{Z} \subset \mathbb{R}^d$, assume $\theta_0 \in H^{\alpha}_K(\mathcal{Z})$ for some compact set $K\subset \mathcal{Z}$,
and denote by $R_l$ the set of indices $r$ for which the support of $\psi_{lr}$ intersects $K$.
We can find a compact set $K'$ such that $K \subsetneq K' \subset \mathcal{Z}$, and a cut-off function $\chi \in C_c^{\infty}(\mathcal{Z})$
such that $\chi = 1$ on $K'$.
Note that $\theta_0$ compactly support on $K$, it can be represented as
\begin{align}
\theta_0 = \chi\theta_0= \sum_{l=-1}^{\infty}\sum_{r\in R_l}\pdt{\theta_0}{\psi_{lr}}_{L^2(\mathcal{Z})}\chi\psi_{lr}.
\end{align}
We define finite dimensional approximation of $\theta_0$ to be
\begin{align}
P_J(\theta_0) = \sum_{l=-1}^{J}\sum_{r\in R_l}\pdt{\theta_0}{\psi_{lr}}_{L^2(\mathcal{Z})}\chi\psi_{lr}, \quad J \in \mathbb{N}.
\end{align}
For integer $\alpha > d/2$, we consider the prior $\Pi'_J$ is defined by the law of $\chi \upsilon$, where
\begin{align}\label{seiveprior1}
\upsilon = \sum_{l=-1}^{J}\sum_{r\in R_l}2^{-l\alpha}\xi_{lr}\psi_{lr}, \quad \xi \mathop{\sim}^{iid} N(0,1).
\end{align}
The truncation point $J = J_N$ diverges when $N\rightarrow\infty$.
Then $\Pi'_J$ is a centred Gaussian probability measure with RKHS
\begin{align*}
\mathcal{H}_J=\set{h : h = \sum_{l=-1}^J\sum_{r\in R_l}2^{-l\alpha}h_{lr}\chi\psi_{lr}},
\end{align*} whose norm is defined by
$\norm{h}_{\mathcal{H}_J}^2 = \sum_{l=-1}^J\sum_{r\in R_l}h^2_{lr}.$
The RKHS $\mathcal{H}_J$ embeds into $H^{\alpha}_c(\mathcal{Z})$ with 
\begin{align}
\norm{h}_{H^{\alpha}(\mathcal{Z})} \leq c\norm{h}_{\mathcal{H}_J}, \quad \forall h \in \mathcal{H}_J,
\end{align}
for some constant $c>0$ (see \cite[Appendix B]{IntroNonLinear_giordano2020consistency}).
\begin{theorem} \label{mainthmsv}
Suppose \cref{condreg} holds for the forward map $\mathcal{G}$, the separable normed linear subspace $(\mathcal{R},\Vert \cdot\Vert _{\mathcal{R}}) = (H^{\alpha}(\mathcal{Z}),\Vert \cdot \Vert _{H^{\alpha}(\mathcal{Z})})$ with integer $\alpha \geq 0$, and the finite constants $C_U > 0$, $C_L > 0$ , ${\kappa}\geq0$, $p\geq 0$, $l \geq 0$.
Let the probability measure $\Pi'_J$ as in \cref{seiveprior1}, and $J = J_N \in \mathbb{N}$ is such that $2^{J} \simeq N^{\frac{1}{2\alpha +2{\kappa} +d}}$. 
Denote by $\Pi_N$ the rescaled prior as in \cref{rescaledprior} with $\theta' \sim \Pi'_J$, and $\Pi_N(\cdot|(Y_i,X_i)_{i=1}^N) = \Pi_N(\cdot|D_N)$ is the corresponding posterior distribution in \cref{Post} arising from data in model \cref{model}. Assume that $\alpha + {\kappa} \geq \frac{d(l+1)}{2}$ and $\theta_0\in H_K^\alpha(\mathcal{Z})$. Then, for $\varepsilon_N=N^{-\frac{\alpha+{\kappa}}{2\alpha+2{\kappa}+d}}$, $\gamma_N^2 = \frac{1}{N}\mathop{\inf}_{Q \in \mathcal{Q}}P_{\theta_0}^{(N)}D(Q\Vert \Pi_N(\cdot|D_N))$, and variational posterior $\hat{Q}$ defined in \cref{variationalposterior}, we have      
  \begin{align}
      P_{\theta_0}^{(N)}\hat{Q}\Vert \mathcal{G}(\theta)-\mathcal{G}(\theta_0)\Vert _{L^{2}_{\lambda}}^{\frac{2}{p+1}}\lesssim \varepsilon_N^{\frac{2}{p+1}} +\gamma^2_N \cdot \varepsilon_N^{-\frac{2p}{p+1}}.
  \end{align}
  Moreover, assume that \cref{condstab} also holds for $\mathcal{G}$, $\mathcal{R}$, the function $F$, and the finite constants $C_T > 0$, $q\geq 0$. Then, we further have
  \begin{align}
      P_{\theta_0}^{(N)}\hat{Q}[F(\Vert f_{\theta} - f_{\theta_0}\Vert_{\mathcal{F}} )]^{\frac{2}{p+q+1}}\lesssim \varepsilon_N^{\frac{2}{p+q+1}} +\gamma^2_N \cdot \varepsilon_N^{-\frac{2p+2q}{p+q+1}}.
  \end{align}
\end{theorem}

We note that the sample of prior $\Pi_N$ is parameterized by the coefficients of the basis:
\[\left\{\Tilde{\theta} = (\Tilde{\theta}_{lr})\in \mathbb{R}^{d_J}: r \in R_l, l \in \{-1,0,\dots,J\}\right\},\]
where $d_J:= \sum_{l=-1}^{J}\abs{R_l}\simeq 2^{Jd}$.
Thus, it is possible to find a variational posterior among finite-dimensional probability measures.
Define a Gaussian mean-field family $\Tilde{\mathcal{Q}}_G^J$ as
\begin{align}\label{OGMF}
\Bigg\{Q = \mathop{\bigotimes}_{l=-1}^J\mathop{\bigotimes}_{r\in R_l}N(\mu_{lr},\sigma_{lr}^2):\mu_{lr} \in \mathbb{R}, \sigma_{lr}^2 \geq 0\Bigg\}.
\end{align}
The map $\Psi_J: \mathbb{R}^{d_J} \mapsto L^2(\mathcal{Z})$ is defined by
\begin{align}\label{pushforwardmap}
\Psi_J(\Tilde{\theta}) = \sum_{l=-1}^J\sum_{r\in R_l}\Tilde{\theta}_{lr}\chi\psi_{lr},\quad \forall \Tilde{\theta} = (\Tilde{\theta}_{lr})\in \mathbb{R}^{d_J}.
\end{align}
Thus, our variational set $\mathcal{Q}_G^J$ is obtained by the push-forward of $\Tilde{\mathcal{Q}}_G^J$ via $\Psi_J$, i.e.,
\begin{align}\label{GMF}
\mathcal{Q}_G^J = \Bigg\{Q = \Tilde{Q}\circ\Psi_J^{-1} : \Tilde{Q} \in \Tilde{\mathcal{Q}}_G^J\Bigg\}.
\end{align}
We can find a probability measure $Q\in \mathcal{Q}_G^J$ such that $R(Q) \lesssim \varepsilon_N^2\log N$, which gives the convergence rate in the following theorem.
\begin{theorem}\label{boundgamfinite}
Suppose the forward map $\mathcal{G}$ satisfies \cref{condreg} and the Gaussian process prior $\Pi_N$ defined as in \cref{mainthmsv}. Then, for $\theta_0\in H_K^\alpha(\mathcal{Z})$ and $\varepsilon_N=N^{-\frac{\alpha+{\kappa}}{2\alpha+2{\kappa}+d}}$, there exists a probability measure $Q_N \in \mathcal{Q}_G^J$ such that
        \begin{align*}
            R(Q_N) \lesssim \varepsilon_N^2\log N.
        \end{align*}
\end{theorem}
With \cref{boundgamfinite}, it is easy to bound $\gamma_N^2$ by
$\gamma_N^2\leq \inf_{Q\in \mathcal{Q}_G^J}R(Q) \leq R(Q_N).$
Then, \cref{mainthmsv,boundgamfinite} together show the convergence rate of variational posterior from $\mathcal{Q}_G^J$.
\begin{theorem}\label{finalthmsv}
      Suppose \cref{condreg} holds for the forward map $\mathcal{G}$, the separable normed linear subspace $(\mathcal{R},\Vert \cdot\Vert _{\mathcal{R}}) = (H^{\alpha}(\mathcal{Z}),\Vert \cdot \Vert _{H^{\alpha}(\mathcal{Z})})$ with integer $\alpha \geq 0$ and the finite constants $C_U > 0$, $C_L > 0$, ${\kappa}\geq0$, $p\geq 0$, $l \geq 0$.
Let probability measure $\Pi'_J$ as in \cref{seiveprior1}, and $J = J_N \in \mathbb{N}$ is such that $2^{J} \simeq N^{\frac{1}{2\alpha +2{\kappa} +d}}$. Denote by $\Pi_N$ the rescaled prior as in \cref{rescaledprior} with $\theta' \sim \Pi'_J$, and $\Pi_N(\cdot|(Y_i,X_i)_{i=1}^N) = \Pi_N(\cdot|D_N)$ is the corresponding posterior distribution in \cref{Post} arising from the data in the model \cref{model}. Assume that $\alpha + {\kappa} \geq \frac{d(l+1)}{2}$ and $\theta_0\in H_K^\alpha(\mathcal{Z})$. Then, for $\varepsilon_N=N^{-\frac{\alpha+{\kappa}}{2\alpha+2{\kappa}+d}}$, the variational set $\mathcal{Q}_G^J$ defined by \cref{GMF} and variational posterior $\hat{Q}$ defined in \cref{variationalposterior}, we have      
  \begin{align*}
      P_{\theta_0}^{(N)}\hat{Q}\Vert \mathcal{G}(\theta)-\mathcal{G}(\theta_0)\Vert _{L^{2}_{\lambda}}^{\frac{2}{p+1}}\lesssim \varepsilon_N^{\frac{2}{p+1}}\log N.
  \end{align*}
  Moreover, assume that \cref{condstab} also holds for $\mathcal{G}$, $\mathcal{R}$, the function $F$ and the finite constants $C_T > 0$, $q\geq 0$. Then, we further have
  \begin{align*}
      P_{\theta_0}^{(N)}\hat{Q}[F(\Vert f_{\theta} - f_{\theta_0}\Vert_{\mathcal{F}} )]^{\frac{2}{p+q+1}}\lesssim \varepsilon_N^{\frac{2}{p+q+1}}\log N.
  \end{align*}
\end{theorem}
\begin{remark}
\textcolor{black}{The convergence rates achieved in this theorem include an extra logarithmic factor compared to the convergence rates from \cref{finalthm}. This difference arises because our variational set $\mathcal{Q}_G^J$ has been limited to the push-forward of a Gaussian mean-field family. 
We note that \cref{finalthmsv} is not restricted to the Gaussian mean-field family $\mathcal{Q}_G^J$. In fact, $\mathcal{Q}_G^J$ serves only as a sufficiently rich variational class for which the approximation error $\gamma_N^2$ can be controlled by the convergence rate $\varepsilon_N^2$. \cref{finalthmsv} is applicable to any variational class $\mathcal{Q}$ that contains $\mathcal{Q}_G^J$. In the context of inverse problems, many applications of variational inference methods fall within the scope of our theorem \cite{pinski2015kullback,pinski2015algorithms,jia2022stein,zhao2025functionalnormalizingflowstatistical}. In \cref{sec:computation} of the Supplementary Material \cite{Zu2024AoS}, we explicitly illustrate how \cref{finalthmsv} applies to these methods, and we also discuss the corresponding computational procedures when the specific variational family $\mathcal{Q}_G^J$ is employed.}
\end{remark}

\section{Contraction rate for two typical inverse problems}\label{ApplicationSection}
In this section, we apply \cref{finalthmsv} to the Darcy flow problem and the inverse potential problem for a subdiffusion equation. In order to verify \cref{condreg,condstab} for these problems, we use link functions satisfying specific properties as in \cite{IntroNonLinear_nickl2020convergence}. \textcolor{black}{For regularity and conditional stability estimates of the Darcy flow problem, we follow the results in section 5 of \cite{IntroNonLinear_nickl2020convergence}. For estimates of the subdiffusion equation, we use the technique in \cite{Frac_jing2022simultaneous}.}
\subsection{Darcy flow problem}\label{SubsectionDarcyFlow}
For a bounded smooth domain $\mathcal{X} \subset \mathbb{R}^d$ ($d\in \mathbb{N}$) and a given source function $g \in C^{\infty}(\mathcal{X})$, we consider solutions $u=u_f$ to the Dirichlet boundary problem
\begin{align}\label{Darcy}
    \left\{\begin{aligned}
    &\nabla\cdot(f \nabla u) = g \quad \mbox{on } \mathcal{X},\\
    &u = 0 \quad \mbox{on } \partial\mathcal{X}.
    \end{aligned}\right.  
\end{align}
In this subsection, we will identify $f$ from the observation of $u_f$.
Assume the parameter $f \in \mathcal{F}_{\alpha,K_{\min}}$ for some integer $\alpha > 1 + d/2$, $K_{\min} \in [0,1)$, where  $\mathcal{F}_{\alpha,K_{\min}}$ is defined as
\begin{align}
\begin{aligned}
    \bigg\{ &f\in H^{\alpha}(\mathcal{X}): f > K_{\min} \text{\ on\ } \mathcal{X}, f= 1 \text{\ on\ } \partial \mathcal{X},  
      \frac{\partial^j f}{\partial n^j} = 0 \text{\ on\ } \partial \mathcal{X} \text{\ for\ } j= 1,\dots,\alpha -1 \bigg\}.
\end{aligned}  
\end{align}
Here, the forward map $G$ is defined by 
\begin{align*}
G : \mathcal{F}_{\alpha,K_{\min}} \rightarrow L^2_{\lambda}(\mathcal{X}),\qquad f \mapsto u_f,
\end{align*}
where the probability measure $\lambda$ is chosen as the uniform distribution on $\mathcal{X}$. 
To build re-parametrisation of $\mathcal{F}_{\alpha,K_{\min}}$, we introduce the approach of using regular link functions $\Phi$ as in \cite{IntroNonLinear_nickl2020convergence}.
Define a function $\Phi$ that satisfies the following properties:
\begin{description}
    \item[{\color{white}a}(i)] For given $K_{\min}> 0$, $\Phi: \mathbb{R}\rightarrow (K_{\min},\infty)$ is a smooth, strictly increasing bijective function such that $\Phi(0)=1$ and $\Phi'>0$ on $\mathbb{R}$;
    \item[{\color{white}a}(ii)] All derivatives of $\Phi$ are bounded, i.e.,
        $\sup_{x\in\mathbb{R}}\abs{\Phi^{(k)}(x)} < \infty$ for all $k\geq 1.$
\end{description}
An example of such a link function is given in \cref{Darcylinkex} of the Supplementary Material \cite{Zu2024AoS}.
Then, the reparametrized forward map $\mathcal{G}$ is defined as
\begin{align}\label{forwardmapDarcyflow}
\mathcal{G} : \Theta_{\alpha,K_{\min}} \rightarrow L^2_{\lambda}(\mathcal{X}), \qquad \theta \mapsto \mathcal{G}(\theta):=G(\Phi(\theta)),
\end{align}
where $\Theta_{\alpha,K_{\min}} : = \set{\theta = \Phi^{-1}\circ f: f\in \mathcal{F}_{\alpha,K_{\min}} }$. It is verified by the properties of $\Phi$ that
\begin{align*}
\Theta_{\alpha,K_{\min}} = \bigg\{\theta \in H^{\alpha}: \frac{\partial^j \theta}{\partial n^j} = 0 \text{\ on\ } \partial \mathcal{X} \text{\ for\ } j=0,\dots,\alpha-1\bigg\} = H^{\alpha}_c(\mathcal{X}).
\end{align*}
The reason why we use the link functions $\Phi$ here instead of the common choice $\Phi = \exp$ is that \cref{condreg,condstab} require the polynomial growth in $\norm{\theta}_{\mathcal{R}}$ of those coefficients. If we use $\Phi = \exp$ as the link function, the polynomial growth is not satisfied.
\begin{theorem}\label{mainthmDarcy}
Let integer $\alpha > (2 + d/2)\vee(2d-1)$ , $d \in \mathbb{N}$ and ${\kappa}=1$. Consider the forward map $\mathcal{G}$ as in \cref{forwardmapDarcyflow}.
Let $\Pi_N$, $\Pi_N(\cdot|D_N)$, $\mathcal{Q}_G^J$ and $\hat{Q}$ be as defined in \cref{finalthmsv}.
Assume that $\theta_0\in H_K^\alpha(\mathcal{X})$. Then, for $\varepsilon_N=N^{-\frac{\alpha+{\kappa}}{2\alpha+2{\kappa}+d}}$,
we have
\begin{gather}
    P_{\theta_0}^{(N)}\hat{Q} \Vert u_{f_{\theta}} - u_{f_0}\Vert_{L^2}^{\frac{2}{p+1}}\lesssim \varepsilon_N^{\frac{2}{p+1}}\log N,\label{DracyGeneralerror}\\
    P_{\theta_0}^{(N)}\hat{Q} \Vert f_{\theta} - f_{0}\Vert_{L^2}^{\frac{\alpha+1}{\alpha-1}\cdot\frac{2}{p+q+1}}\lesssim \varepsilon_N^{\frac{2}{p+q+1}} \log N. \label{Dracyerror}
\end{gather}
with $t = \lceil d/2 \rceil+2$, $p = t^3 + t^2$, $q = \frac{(2\alpha^2+1)(\alpha+1)}{\alpha-1}$. 
\end{theorem}
We note that for the ``PDE-constrained regression'' problem of recovering $u_{f_0}$ in ``prediction'' loss, the convergence rate
obtained in \cref{DracyGeneralerror} can be shown to be minimax optimal (up to a logarithmic factor) \cite[Section 2.3.2]{IntroNonLinear_giordano2020consistency}. For a smooth truth $f_0$, both of the rates obtained in \cref{DracyGeneralerror} and \cref{Dracyerror} approach the optimal rate $N^{-1/2}$ of finite-dimensional models when we let $\alpha \rightarrow +\infty$. However, the optimal reconstruction rate for the Darcy flow problem with general Sobolev regularity of the truth $f_0$ remains to be studied for future research.
\begin{remark}
For the inverse potential problem of the Schr\"{o}dinger equations discussed in \cite{IntroNonLinear_nickl2020bernstein}, it is worth noting that the convergence rate of the variational posterior can be obtained using the link function detailed in \cref{SubsectionSuddiffuion}. The convergence rate towards the truth $f_0$ reaches the same rate $N^{-\frac{\alpha}{2\alpha + 4 + d}}$ (up to a logarithmic factor) as that proved to be minimax optimal in \cite{IntroNonLinear_nickl2020bernstein}. Since our results for Schr\"{o}dinger equations can be obtained directly through a process similar to that used for the Darcy flow problem, using regularity and conditional stability estimates from \cite{IntroNonLinear_nickl2023bayesian}, we will not provide theorems and proofs here. Instead, we present our results on the inverse potential problem for a subdiffusion equation in \cref{SubsectionSuddiffuion}.
\par
\end{remark}
\subsection{Inverse potential problem for a subdiffusion equation}\label{SubsectionSuddiffuion}
Let domain $\Omega = (0,1)$ and we consider solutions $u(t)=u_{\beta,q}(t)$ to a subdiffusion equation with a non-zero Dirichlet boundary condition:
\begin{align}\label{Fractional}
    \left\{\begin{aligned}
    &\partial_t^{\beta}u - \partial_{xx}u + qu = f \quad \mbox{in } \Omega \times (0,T],\\
    &u(0,t) = a_0, u(1,t) = a_1 \quad \mbox{on } (0,T],\\
    &u(0) = u_0 \quad \mbox{in } \Omega,
    \end{aligned}\right.  
\end{align}
where $\beta \in (0,1)$ represents the fractional order, $T > 0$ stands for a fixed final time, $f>0$ is a specified source term, $u_0>0$ denotes given initial data, the non-negative function $q\in L^{\infty}(\Omega)$ refers to a spatially dependent potential, and $a_0$ and $a_1$ are positive constants. The notation $\partial^\beta_tu(t)$ denotes the Djrbashian--Caputo fractional derivative in time $t$ of order $\beta\in (0,1)$, 
\begin{align}
\partial_t^{\beta}u(t) = \frac{1}{\Gamma(1-\beta)}\int^t_0(t-s)^{-\beta}u'(s)ds,
\end{align}
where $\Gamma(x)$ is the Gamma function. 
For in-depth analysis of fractional differential equations and the Djrbashian-Caputo fractional derivative, please refer to references \cite{jin2021fractional,Jia2017JDE}. \textcolor{black}{There are also posterior consistency results for subdiffusion equations \cite{Frac_kow2025consistency}, where the subdiffusion equation is governed by the Riemann--Liouville derivative different from the Djrbashian--Caputo derivative we consider here.}
\par 
In this section we consider the identification of the potential $q$ from the observation of $u(T)$.
For $\alpha \in \mathbb{N}$, we define the parameter space
\begin{align}\label{FracParameterspcae}
\mathcal{F}_{\alpha,M_0} =\left\{ q\in H^{\alpha} \cap \mathcal{I}: q\vert_{\partial\Omega}=1, \frac{\partial^jq}{\partial n^j}\Big|_{\partial\Omega} =0 \text{\ for\ } j=1,\dots,\alpha -1 \right\},
\end{align}
where $\mathcal{I} = \{q \in L^{\infty}: 0< q < M_0\}$ for $M_0 >1$, and its subclasses
\[\mathcal{F}_{\alpha,M_0}(R) =\left\{ q\in \mathcal{F}_{\alpha,M_0}: \norm{q}_{H^{\alpha}}\leq R\right\}, \quad R > 0.\]
We assume $u_0 \in H^{\alpha}(\Omega)$, $f \in H^{\alpha}(\Omega)$ with $u_0, f \geq L_0$ a.e. and $a_0, a_1 \geq L_0$ for $L_0>0$.
Here the forward map $G$ is defined by 
\begin{align*}
G : \mathcal{F}_{\alpha,M_0} \rightarrow L^2_{\lambda}(\Omega),\qquad q \mapsto u_q(T),
\end{align*}
where probability measure $\lambda$ is chosen as the uniform distribution on $\Omega$.
We use a link function $\Phi$ to construct a reparametrization of $\mathcal{F}_{\alpha,M_0}$.
Define $\Phi$ that satisfies the following properties:
\begin{description}
    \item[{\color{white}a}(i)] For given $M_0> 1$, $\Phi: \mathbb{R}\rightarrow (0,M_0)$ is a smooth, strictly increasing bijective function such that $\Phi(0)=1$ and $\Phi'>0$ on $\mathbb{R}$;
    \item[{\color{white}a}(ii)] All derivatives of $\Phi$ are bounded, i.e.,
	$\sup_{x\in\mathbb{R}}\abs{\Phi^{(k)}(x)} < \infty$ for all $k\geq 1.$
\end{description}
One example to satisfy (i) and (ii) is the logistic function \cite{furuya2024consistency}:
\[\Phi(t) = \frac{M_0}{M_0+(M_0-1)(e^{-t}-1)}.\]
Then, the reparametrized forward map $\mathcal{G}$ is then defined as
\begin{align}\label{forwardmapFrac}
\mathcal{G} : \Theta_{\alpha,M_0} \rightarrow L^2_{\lambda}(\Omega), \qquad \theta \mapsto \mathcal{G}(\theta):=G(\Phi(\theta)).
\end{align}
with $\Theta_{\alpha,M_0} : = \left\{\theta = \Phi^{-1}\circ q: q\in \mathcal{F}_{\alpha,M_0} \right\}$. 
It is verified through the properties of $\Phi$ that
\begin{align*}
\Theta_{\alpha,M_0} = \left\{\theta \in H^{\alpha}: \frac{\partial^j \theta}{\partial n^j} = 0 \text{\ on\ } \partial \Omega \text{\ for\ } j=0,\dots,\alpha-1\right\} = H^{\alpha}_c(\Omega).
\end{align*}
\par
\begin{theorem}\label{mainthmFrac}
Let $\alpha \in \mathbb{N}$, $\alpha> 2+d/2$, $d = 1$ and ${\kappa}= 2$.
Consider the forward map $\mathcal{G}$ as in \cref{forwardmapFrac} with terminal time $T\geq T_0$ where $T_0$ is large enough.
Let $\Pi_N$, $\Pi_N(\cdot|D_N)$, $\mathcal{Q}_G^J$ and $\hat{Q}$ be as defined in \cref{finalthmsv}.
Assume that $\theta_0\in H_K^\alpha(\Omega)$. Then, for $\varepsilon_N=N^{-\frac{\alpha+{\kappa}}{2\alpha+2{\kappa}+d}}$,
we have
\begin{gather}
   P_{\theta_0}^{(N)}\hat{Q} \Vert u_{q_{\theta}}(T) - u_{q_{0}}(T)\Vert_{L^2}^{2}\lesssim \varepsilon_N^{2}\log N, \label{FracGeneralerror}\\
    P_{\theta_0}^{(N)}\hat{Q} \Vert q_{\theta} - q_{0}\Vert_{L^2}^{\frac{2+\alpha}{\alpha}\cdot\frac{2}{q+1}}\lesssim \varepsilon_N^{\frac{2}{q+1}}\log N, \label{Fracerror}
\end{gather}
with $q = 2+4\alpha$.
\end{theorem}
\par
For the ``PDE-constrained regression'' problem of recovering $u_{q_{0}}(T)$ in ``prediction'' loss, the convergence rate found in \cref{FracGeneralerror} will be demonstrated to be minimax optimal (up to a logarithmic factor), as evidenced by \cref{Fracminmax} given below. The convergence rate for recovering $q_{0}$ in $L^2$ norm as presented in \cref{Fracerror} will similarly be proven minimax optimal (up to a logarithmic factor) in \cref{Fracminmax}. For a smooth truth $q_0$, the rates obtained in \cref{FracGeneralerror} and \cref{Fracerror} both approach the optimal rate $N^{-1/2}$ of finite-dimensional models as $\alpha \rightarrow +\infty$.
\begin{theorem}\label{Fracminmax}
For $M_0>1$, $\alpha \in \mathbb{N}$, $q\in\mathcal{F}_{\alpha,M_0}$, consider the solution $u_q(t)$ of the problem \cref{Fractional}. Then there exist fixed $T_0>0$ and a finite constant $C>0$ such that for $N$ large enough, the terminal time $T\geq T_0$, variational posterior $\hat{Q}$ defined in \cref{mainthmFrac} and any $\eta_1,\eta_2>0$,
    \begin{gather*}
         \inf_{\Tilde{u}_N}\sup_{q_0\in\mathcal{F}_{\alpha,M_0}(R)}P^{(N)}_{\theta_0}\hat{Q}\norm{\Tilde{u}_N-u_{q_0}(T)}^{\eta_1}_{L^2(\Omega)}\geq C N^{-\frac{\alpha+2}{2\alpha+4+1}\cdot\eta_1},\\
         \inf_{\Tilde{q}_N}\sup_{q_0\in\mathcal{F}_{\alpha,M_0}(R)}P^{(N)}_{\theta_0}\hat{Q}\norm{\Tilde{q}_N-q_0}^{\frac{\alpha+2}{\alpha}\cdot \eta_2}_{L^2(\Omega)}\geq C N^{-\frac{\alpha+2}{2\alpha+4+1}\cdot \eta_2},
     \end{gather*}
      where $\theta_0 = \Phi^{-1}(q_0)$ and the infimums range over all measurable functions $\Tilde{u}_N = \Tilde{u}_N(\theta)$, $\Tilde{q}_N = \Tilde{q}_N(\theta)$ that take value in $L^2(\Omega)$ with $\theta$ from $\hat{Q}$ respectively. 
\end{theorem}
\section{Contraction rate for severely ill-posed problem}\label{SectionContractionRateUnstable}
There are some inverse problems that have not been shown to satisfy our conditional stability condition \cref{stab}. For example, inverse problems with log-type stability may not satisfy the polynomial growth of the coefficients in \cref{condstab}. For this kind of problem, we construct a particular prior $\Tilde{\Pi}_N$ and prove that its variational posterior can still contract to the true parameter. 
\subsection{Construction of the prior distribution}
First we introduce a kind of orthonormal wavelet basis of $L^2(\mathcal{Z})$ used in \cite{IntroNonLinear_nickl2020bernstein}, which is given by
$\left\{\psi^{\mathcal{Z}}_{lr} : r \leq N_l, l \in \mathbb{N}\cup\{-1,0\}\right\}$ for $r, N_l \in \mathbb{N}.$
This wavelet basis consists of $\psi_{lr}^{\mathcal{Z}} = \psi_{lr}$ that are compactly supported in $\mathcal{Z}$ and also consists of boundary corrected wavelets $\psi_{lr}^{\mathcal{Z}} = \psi_{lr}^{bc}$, which are supported near the boundary of $\mathcal{Z}$ and are constructed as linear combinations
\[\psi_{lr}^{bc}(z) = \sum_{\abs{m-m'}\leq K}d_{l}(m,m')\psi_{lm'}(z), \quad d_l(m,m') \in \mathbb{R}, m = m(l,r), K \in \mathbb{N},\]
where $K$ is independent from $l,r$ and $\psi_{lr}$ are the $S$-regular Daubechies wavelets given in \cref{Dbase}. Here we record some properties that will be used later from \cite{IntroNonLinear_nickl2020bernstein}:
\begin{align}\label{baseproperty}
\begin{aligned}
&\qquad\qquad N_l \leq c_02^{ld}, \quad \sum_{\abs{m-m'}\leq K}d_{l}(m,m') \leq D,\\
&\text{supp }\psi_{lr}^{bc} \subset \set{z\in \mathcal{Z}: \text{dist}(z,\partial\mathcal{Z})\leq c_1/2^l}, \quad \sum_r\abs{D^i\psi_{0r}} \in C(\mathbb{R}^d).
\end{aligned}
\end{align}
for $\abs{i} \leq S$. Define H$\Ddot{\text{o}}$lder-Zygmund type spaces for this wavelet basis by
\begin{align}\label{HZspace}
\mathcal{C}^{\alpha,W}(\mathcal{Z}) = \set{f\in L^2(\mathcal{Z}):\norm{f}_{\mathcal{C}^{\alpha,W}(\mathcal{Z})}\equiv \sup_{l,r} 2^{l(\alpha+d/2)}\abs{\pdt{f}{\psi_{lr}^{\mathcal{Z}}}_{L^2(\mathcal{Z})}}< \infty}.
\end{align}
\par
For integer $\alpha \geq 0$, we construct the prior $\Pi_J'$  (use the same notation as the prior in \cref{High-dimensional Gaussian sieve priors} with basis $\psi_{lr}$ in slight abuse of notations) by the law of $\upsilon$, where
\begin{align}\label{priorsvZ}
\upsilon = \sum_{l=-1}^J\sum_{r=1}^{N_l}2^{-l(\alpha+d/2)}\bar{l}^{-2}\xi_{lr}\psi_{lr}^{\mathcal{Z}}, \quad \xi_{lr} \mathop{\sim}^{iid} N(0,1), \quad \bar{l} = \max(l,1).
\end{align}
Its RKHS is $\mathcal{H}_J^{\mathcal{Z}}$ defined by
\begin{align}\label{RKHSsvZ}
\mathcal{H}^{\mathcal{Z}}_J = \bigg\{f = \sum_{l=-1}^{J}\sum_{r=1}^{N_l}2^{-l(\alpha+d/2)}\bar{l}^{-2}f_{lr}\psi_{lr}^{\mathcal{Z}} : f_{lr} \in \mathbb{R}\bigg\} 
\end{align} 
with norm
$\norm{f}_{\mathcal{H}_J^{\mathcal{Z}}} = \sqrt{\sum_{l=1}^J\sum_{r=1}^{N_l}f_{lr}^2}.$
Then, we have that $\mathcal{H}_J^{\mathcal{Z}}$ is continuously embedded in $C^{\alpha}(\mathcal{Z})$ (see \cref{embeddinglemma} in the Supplementary Material \cite{Zu2024AoS}), that is
$\norm{f}_{C^{\alpha}} \lesssim \norm{f}_{\mathcal{H}_J^{\mathcal{Z}}}.$
We also define the rescaled prior $\Pi_N$ as in \cref{rescaledprior} with $\theta' \sim \Pi'_J$ where $\Pi'_J$ is defined by \cref{priorsvZ}.
\par
In order to construct a prior supported in a ball, let us define a function $h$ {\color{black}that} satisfies the following properties:
\begin{description}
    \item[{\color{white}a}(i)] $h : \mathbb{R} \rightarrow (-B,B)$ is continuous with derivative $h' > 0$;
    \item[{\color{white}a}(ii)] $h(0) = 0$ , $h'$ is uniformly bounded with continuous derivative $(h^{-1})'$.
\end{description}
It is obvious that there exists a constant $c$ such that
$\abs{h(x_1)-h(x_2)} \leq c\abs{x_1-x_2}$ and $\abs{h(x)}\leq c\abs{x}.$
Using $h$, we define the operator $\Psi: L^2({\mathcal{Z}}) \rightarrow L^{2}({\mathcal{Z}})$ through
\begin{align}\label{linksvZ}
\Psi(\theta) = \sum_{l=-1}^{\infty}\sum_{r=1}^{N_l}2^{-l(\alpha+d/2)}\bar{l}^{-2}h(\theta_{lr})\psi_{lr}^{\mathcal{Z}}, \quad \text{for\ } \theta_{lr} = 2^{l(\alpha+d/2)}\bar{l}^{\, 2}\pdt{\theta}{\psi_{lr}^{\mathcal{Z}}}_{L^2}.
\end{align}
Note that $\Psi(\theta) \in \mathcal{F}_J$ for any $\theta \in \mathcal{H}_J^{\mathcal{Z}}$
where
\[\mathcal{F}_{J} = \left\{\Psi(\theta) : \theta \in \mathcal{H}_J^{\mathcal{Z}}\right\} = \bigg\{\theta = \sum_{l=-1}^{J}\sum_{r=1}^{N_l}2^{-l(\alpha+d/2)}\bar{l}^{-2}\theta_{lr}\psi_{lr}^{\mathcal{Z}} : \abs{\theta_{lr}} < B \bigg\}.\]
Then, the ``push forward'' prior $\Tilde{\Pi}_N$ induced by $\Pi_N$ and $\Psi$ is defined as the law of
\begin{align}\label{pushforwardsvZ}
\Psi(N^{-\frac{d}{4\alpha+4{\kappa}+2d}} F) = \sum_{l=1}^J\sum_{r=1}^{N_l}2^{-l(\alpha+d/2)}\bar{l}^{-2}h(N^{-\frac{d}{4\alpha+4{\kappa}+2d}}\xi_{lr})\psi_{lr}^{\mathcal{Z}}, \quad \xi_{lr} \mathop{\sim}^{iid} N(0,1).
\end{align}
The construction of the prior $\Tilde{\Pi}_N$ is inspired by the prior used in \cite{IntroNonLinear_nickl2020bernstein} which is defined as the law of
$\sum_{l=1}^J\sum_{r=1}^{N_l}b_{lr}\psi_{lr}^{\mathcal{Z}}$, where
$b_{lr} \mathop{\sim}^{iid}U(-B2^{-l(\alpha+d/2)}\bar{l}^{-2},B2^{-l(\alpha+d/2)}\bar{l}^{-2})$ for $\bar{l} = \max(l,1)$.
In our prior, we replace $b_{lr}$ by $2^{-l(\alpha+d/2)}\bar{l}^{-2}h(N^{-\frac{d}{4\alpha+4{\kappa}+2d}}\xi_{lr})$, which is also supported in the interval 
$(-B2^{-l(\alpha+d/2)}\bar{l}^{-2},B2^{-l(\alpha+d/2)}\bar{l}^{-2}).$ This approach aims to establish a connection to the theorems proposed with Gaussian priors in \cref{VariationalConsistency}.
\subsection{Contraction theorem with weak conditions}
We see that $\Psi(\theta)$ is in a ball of $C^{\gamma}(\mathcal{Z})$ for any $\theta \in L^2(\mathcal{Z})$ and any integer $\gamma$ such that $ 0 \leq \gamma \leq \alpha $. By \cref{baseproperty}, we have
\begin{align*}
\Big\vert\sum_{l=-1}^{\infty}\sum_{r=1}^{N_l}2^{-l(\alpha +d/2)}\bar{l}^{-2}h(\theta_{lr})D^i\psi_{lr}^{\mathcal{Z}}\Big\vert
    \leq cB\sum_{l=-1}^{\infty}2^{-l(\alpha - \abs{i})}{\bar{l}}^{-2}\sum_{r}|D^i\psi_{0r}|\leq C(B)
\end{align*}
for any $0\leq \abs{i}\leq \gamma$.
Thus, the ``push forward'' prior $\Tilde{\Pi}_N$ is supported in a ball of $C^{\gamma}(\mathcal{Z})$ with radius $C(B)$. For forward maps defined on a ball of $C^{\gamma}(\mathcal{Z})$ with radius $C(B)$, the growth rates in \cref{condreg,condstab} are naturally satisfied with fixed constants depending on $B$. We are able to relax the conditions on the forward map and give convergence rates for inverse problems with log-type stability under our framework. Here we propose the weak regularity and conditional stability conditions on the forward map $\mathcal{G}$.
\begin{condition}\label{loosecondreg}
     Consider a parameter space $\Theta \subseteq L^2(\mathcal{Z},\mathbb{R})$. 
  The forward map $\mathcal{G}:\Theta \rightarrow L^2_{\lambda}(\mathcal{X},V)$ is measurable. Assume that $(\mathcal{R},\Vert \cdot\Vert _{\mathcal{R}}) = (C^{\gamma},\Vert \cdot\Vert _{C^{\gamma}})$ is a subspace of $\Theta$ for some $\gamma \geq 0$.
  Suppose for all $M>1$, there exist finite constants $U_{\mathcal{G}} > 0$, $L_{\mathcal{G}} > 0$ that may depend on M, such that 
  \begin{gather}
          \mathop{\sup}_{\theta \in \Theta \cap B_{\mathcal{R}}(M)} \mathop{\sup}_{x \in \mathcal{X}}\abs{\mathcal{G}(\theta)(x)}_V \leq U_{\mathcal{G}}, \label{loosebound} \\
          \norm{\mathcal{G}(\theta_1)-\mathcal{G}(\theta_2)}_{L^{2}_{\lambda}(\mathcal{X},V)} \leq L_{\mathcal{G}}\norm{\theta_1-\theta_2} _{L^{2}(\mathcal{Z})}, \quad \theta_1,\theta_2 \in \Theta \cap B_{\mathcal{R}}(M).  \label{looselip}
   \end{gather}
\end{condition}
\begin{condition}\label{loosecondstab}
    Consider a parameter space $\Theta \subseteq L^2(\mathcal{Z},\mathbb{R})$. 
  The forward map $\mathcal{G}:\Theta \rightarrow L^2_{\lambda}(\mathcal{X},V)$ is measurable. Assume that $(\mathcal{R},\Vert \cdot\Vert _{\mathcal{R}}) = (C^{\gamma},\Vert \cdot\Vert _{C^{\gamma}})$ is a subspace of $\Theta$ for some $\gamma \geq 0$.
  Suppose for all $M>1$, there exist a function $F: \mathbb{R}^{+} \rightarrow \mathbb{R}^{+}$ and finite constants $T_{\mathcal{G}} > 0$ that both may depend on $M$ such that
        \begin{align}\label{loosestab}
              F(\Vert f_{\theta} - f_{\theta_0}\Vert_{\mathcal{F}} ) \leq T_{\mathcal{G}}\Vert \mathcal{G}(\theta)-\mathcal{G}(\theta_0)\Vert _{L^{2}_{\lambda}(\mathcal{X},V)}, \,\,\forall\, \theta \in \Theta \cap B_{\mathcal{R}}(M).
        \end{align}
\end{condition}
\par
In order to control the bias $\norm{\theta_0 -P_J(\theta_0)}$, we let the ground truth $\theta_0$ have a stronger regularity than that of the prior. Assume the parameter $\theta_0 \in C_c^{\beta}(\mathcal{Z})$ for $\beta$ such that $\alpha + d/2 < \beta < S$. The truth $\theta_0$ can be represented as
$\theta_0 = \sum_{l=-1}^{\infty}\sum_{r=1}^{N_l}\pdt{\theta_0}{\psi_{lr}}_{L^2(\mathcal{Z})}\psi^{\mathcal{Z}}_{lr}.$
Proposition 26 in \cite{IntroNonLinear_nickl2020bernstein} implies the continuous embedding $C_c^{\beta}(\mathcal{Z}) \subset \mathcal{C}^{\beta,W}(\mathcal{Z})$.
Thus, we have 
\[ \abs{\pdt{\theta_0}{\psi_{lr}^{\mathcal{Z}}}_{L^2(\mathcal{Z})}} \leq 2^{-l(\beta+d/2)}\sup_{l,r} 2^{l(\beta+d/2)}\abs{\pdt{\theta_0}{\psi_{lr}^{\mathcal{Z}}}_{L^2(\mathcal{Z})}} \leq c\norm{\theta_0}_{C^{\beta}(\mathcal{Z})}2^{-l(\beta+d/2)},\]
which leads to the decay assumption on the wavelet coefficients of $\theta_0$:
\begin{align}\label{coefdecay}
\sup_{l,r} 2^{l(\beta+d/2)}\abs{\pdt{\theta_0}{\psi_{lr}^{\mathcal{Z}}}_{L^2(\mathcal{Z})}} =B_0 < B.
\end{align}
Note that in this context, our variational set $\mathcal{Q}_{MF}^J$ is defined by a modification of \cref{GMF}, where we replace the basis $\chi\psi_{lr}$ by the basis $\psi_{lr}^{\mathcal{Z}}$ introduced in this section.
\begin{theorem}\label{finalthmsvZ}
      Suppose \cref{loosecondreg} holds for the forward map $\mathcal{G}$, the normed linear subspace $(\mathcal{R},\Vert \cdot\Vert _{\mathcal{R}}) = (C^{\gamma},\Vert \cdot\Vert _{C^{\gamma}})$ with $U_{\mathcal{G}} > 0$, $L_{\mathcal{G}} > 0$ and integer $\gamma \geq 0$. Denote by $\Tilde{\Pi}_N$ the prior as in \cref{pushforwardsvZ} with $J = J_N \in \mathbb{N}$ such that $2^{J} \simeq N^{\frac{1}{2\alpha +d}}$. $\Tilde{\Pi}_N(\cdot|(Y_i,X_i)_{i=1}^N) = \Tilde{\Pi}_N(\cdot|D_N)$ is the corresponding posterior distribution in \cref{Post} arising from the data in the model \cref{model}. Assume that integer $\alpha \geq \frac{d}{2}\vee\gamma$. Denote by $\hat{Q}$ the variational posterior  defined in \cref{variationalposterior} with variational set $\mathcal{Q}_{MF}^J$. If $\theta_0\in C_c^{\beta}(\mathcal{Z})$ satisfies \cref{coefdecay}, then for $\varepsilon_N=N^{-\frac{\alpha}{2\alpha+d}}$, we have      
  \begin{align}
      P_{\theta_0}^{(N)}\hat{Q}\Vert \mathcal{G}(\theta)-\mathcal{G}(\theta_0)\Vert _{L^{2}_{\lambda}}^{2}\lesssim \varepsilon_N^{2}\log N.
  \end{align}
  Moreover, assume that \cref{loosecondstab} also holds for $\mathcal{G}$, $\mathcal{R}$, the function $F$ and the finite constant $T_{\mathcal{G}} > 0$. Then, we further have
  \begin{align}
      P_{\theta_0}^{(N)}\hat{Q}[F(\Vert f_{\theta} - f_{\theta_0}\Vert_{\mathcal{F}} )]^{2}\lesssim \varepsilon_N^{2}\log N.
  \end{align}
\end{theorem}
\par
\subsection{Inverse medium scattering problem}
Let the real-valued positive refractive index $n \in L^{\infty}(\mathbb{R}^3)$, with supp$(1-n) \subset D$, where $D$ is a bounded domain with smooth boundary. Let $u = u^i + u_n^s$ satisfy 
\begin{align}\label{Helm}
\Delta u + \omega^2 n u = 0\quad \text{in }\mathbb{R}^3,
\end{align}
and the Sommerfeld radiation condition
\begin{align*}
\mathop{\lim}_{\vert x \vert \rightarrow \infty} \abs{x} \left(\frac{\partial u_n^s}{\partial\abs{x}}-\im \omega u_n^s\right) = 0.
\end{align*}
Assume that $u^i$ is the plane incident field, i.e. $u^i = e^{\im\omega x\cdot\vartheta}$ with $\vartheta \in \mathcal{S}^2$, where $\mathcal{S}^2$ is the unit sphere in $\mathbb{R}^3$.
Then the far field pattern $u_n^{\infty}(\hat{x},\vartheta)$ is defined by
\begin{align}
u_n^s(x,\vartheta) = \frac{e^{\im\omega\abs{x}}}{\abs{x}}u_n^{\infty}(\hat{x},\vartheta) + O\left(\frac{1}{\abs{x}^2}\right),
\end{align}
where $\hat{x} = x/\abs{x}$. Our aim is to determine the medium $n$ from the knowledge of the far field pattern $u_n^{\infty}(\hat{x},\vartheta)$ for all $\hat{x},\vartheta \in \mathcal{S}^2$ at a fixed wave number $\omega >0$.
For the statistical model, we set $\lambda$ be the uniform distribution on $\mathcal{S}^2\times \mathcal{S}^2$. The forward map $G$ is defined as
\begin{align}\label{inversemediumforward}
G(n)(X_i) = u_n^{\infty}(\hat{x},\vartheta),
\end{align}
where $(\hat{x},\vartheta)$ is the realization of $X_i$. Since $G(n)(X_i)$ is a complex-valued function, we treat it as a $\mathbb{R}^2$-valued function, denoted as
\[G(n)(X_i) = \left(\begin{array}{c}
                 \text{Re}\{u_n^{\infty}(\hat{x},\vartheta)\} \\
                 \text{Im}\{u_n^{\infty}(\hat{x},\vartheta)\}
            \end{array}\right).\]
We define the parameter space for $M_0>1$ and integer $\alpha > 3/2$
\begin{align*}
\mathcal{F}_{\alpha,M_0} =\left\{ n\in C^{\alpha}(D): 0<n<M_0, n\vert_{\partial D}=1, \frac{\partial^jn}{\partial n^j}
\Big\vert_{\partial D} =0 \text{\ for\ } j=1,\dots,\alpha -1 \right\}.
\end{align*}
Each $n\in\mathcal{F}_{\alpha,M_0}$ can be seen as functions defined on $\mathbb{R}^3$ by extending $n\equiv1 \in \mathbb{R}^3\setminus D$.
We reparametrize $\mathcal{F}_{\alpha,M_0}$ by the same link function $\Phi$ used in \cref{SubsectionSuddiffuion}.
We set $\Theta_{\alpha,M_0} : = \left\{\theta = \Phi^{-1}\circ n: n\in \mathcal{F}_{\alpha,M_0} \right\}$. The reparametrized forward map $\mathcal{G}$ is then defined as
\begin{align}\label{forwardmapScra}
\mathcal{G} : \Theta_{\alpha,M_0} \rightarrow L^2_{\lambda}(\mathcal{S}^2\times \mathcal{S}^2,\mathbb{R}^2), \qquad \theta \mapsto \mathcal{G}(\theta):=G(\Phi(\theta)).
\end{align}
It is verified through the properties of $\Phi$ that
\begin{align*}
\Theta_{\alpha,M_0} = \left\{\theta \in C^{\alpha}(D): \frac{\partial^j \theta}{\partial n^j} = 0 \text{\ on\ } \partial D \text{\ for\ } j=0,\dots,\alpha-1\right\} = C^{\alpha}_c(D).
\end{align*}
\begin{theorem}\label{mainScra}
Consider the forward map $\mathcal{G}$ in \cref{forwardmapScra} with integer $\alpha >3/2$. Let $\Tilde{\Pi}_N$, $\Tilde{\Pi}_N(\cdot|D_N)$ and $\hat{Q}$ be as defined in \cref{finalthmsvZ}. If $\theta_0\in C_c^{\beta}(\mathcal{Z})$ satisfies \cref{coefdecay} with $\beta > \alpha + 3/2$ and $B>0$, then for $\varepsilon_N=N^{-\frac{\alpha}{2\alpha+3}}$, we have
  \begin{align*}
    P_{\theta_0}^{(N)}\hat{Q}\mexp{-C{\norm{n_{\theta}-n_0}_{L^{\infty}(D)}^{\frac{1}{-a(\alpha)+\delta}}}}\lesssim \varepsilon_N^{2}\log N.
\end{align*}
with any $0<\delta<a(\alpha): =\frac{2\alpha-3}{2\alpha+3}$ and $C = C(D,\omega,\delta,B)$.
\end{theorem}
Note that the convergence rate shown in \cref{mainScra} is equivalent to the so-called logarithmic order rate $(\log N)^{-a}$, $a>0$. It is optimal in the sense of \cite{furuya2024consistency}. By a simple application of Markov's inequality, we have the following corollary.
\begin{corollary}\label{CoroScra}
Under the settings of \cref{mainScra}, we have
\begin{align*}
P_{\theta_0}^{(N)}\hat{Q}\bbra{\norm{n_{\theta}-n_0}_{L^{\infty}(D)}>C(\log N)^{-\frac{2\alpha-3}{2\alpha+3}+\delta}} \rightarrow 0.
\end{align*}
for any $0<\delta<\frac{2\alpha-3}{2\alpha+3}$ and $C=C(\alpha,D,\omega,\delta,B).$
\end{corollary}
\section{Contraction with extra unknown parameters}\label{SectionContractionExtra}
Sometimes, besides the target parameter $\theta$, our forward map $\mathcal{G}$ of PDE also contains another unknown parameter $\beta$ in some finite dimensional vector space with norm $\norm{\ \cdot\ }_{\mathcal{U}}$. The forward map can be denoted as 
\begin{align}
\mathcal{G}_{\beta}(\theta) = G_{\beta}(\Phi(\theta)).
\end{align}
A leading example is the inverse potential problem for a subdiffusion equation with an unknown fractional order $\beta$, which will be discussed later.
To infer $\theta$ while the $\beta$ is unknown, we apply the method of variational inference with model selection as demonstrated in section 4 of \cite{zhang2020convergence}.
\subsection{General theorem}
Assume the true parameter $\beta_0$ belongs to a set $\mathcal{U}$. Then the possible probability models for sampled data are
\begin{align*}
\mathcal{M} = \left\{ P^{(N)}_{\beta, \theta}:\beta \in \mathcal{U}, \theta \in \Theta\right\},
\end{align*}
where the product probability measure is defined as in \cref{modeldensity} with density $p^{(N)}_{\beta,\theta}$. Let $\pi(\beta)$ denote the probabality density for the prior of $\beta$ and $\theta$ follow the prior $\Pi_N$ as in \cref{mainthmsv}. 
For variational inference, we use the Gaussian mean-field variational set $\mathcal{Q}_G^J$ defined by \cref{GMF}.

To select the best $\beta \in \mathcal{U}$ from data, we consider optimizing the evidence lower bound. One of model selection procedures is to maximize $\log (p^{(N)}_{\beta}(D_{N})\pi(\beta))$ over $\beta \in \mathcal{U}$, where 
\begin{align}\label{integral}
p^{(N)}_{\beta}(D_{N}) = \int p^{(N)}_{\beta,\theta}(D_{N}) d\Pi(\theta).
\end{align}
Thus the optimization problem is 
\begin{align}\label{optimizationup}
\mathop{\max}_{\beta \in \mathcal{U}} \log \left(p^{(N)}_{\beta}(D_{N})\pi(\beta)\right).
\end{align}
However, the integral is hard to handle. We instead optimize a lower bound of problem \cref{optimizationup}, and the optimization problem for this lower bound is given by 
\begin{align}\label{opplm}
\mathop{\max}_{\beta \in \mathcal{U}}\mathop{\max}_{Q \in \mathcal{Q}_G^J} \int\log p^{(N)}_{\beta,\theta}(D_{N})dQ(\theta) - D(Q\Vert\Pi_N) + \log\pi(\beta),
\end{align}
which is derived by an application of Jensen’s inequality. We denote the solution to \cref{opplm} by $\hat{\beta}$ and $\hat{Q}$.
To give contraction result of parameter $\theta$ with unknown $\beta$, we need to add some requirements related to $\beta$ into \cref{condreg,condstab}.
\begin{condition}\label{condregmodel}
     Consider a parameter space $\Theta \subseteq L^2(\mathcal{Z},\mathbb{R})$. 
  The forward map $\mathcal{G}_{\beta}:\Theta \rightarrow L^2_{\lambda}(\mathcal{X},V)$ is measurable. 
  Suppose for some normed linear subspace $(\mathcal{R},\Vert \cdot\Vert _{\mathcal{R}})$ of $\Theta$ and all $M>1$, there exist finite constants $C_U > 0$, $C_L > 0$, $C_{\mathcal{U}} >0$, ${\kappa}\geq0$, $p\geq 0$ and $l \geq 0$ such that 
      \begin{gather}
          \mathop{\sup}_{\theta \in \Theta \cap B_{\mathcal{R}}(M)}\mathop{\sup}_{x \in \mathcal{X}}\abs{\mathcal{G}_{\beta}(\theta)(x)}_V\leq C_{U}M^p, \quad \forall \beta \in \mathcal{U},\label{boundmodel}\\
          \norm{\mathcal{G}_{\beta_1}(\theta_1)-\mathcal{G}_{\beta_2}(\theta_2)}_{L^{2}_{\lambda}(\mathcal{X},V)} \leq  C_L(1+\norm{\theta_1}^l_{\mathcal{R}}\vee \norm{\theta_2}^l_{\mathcal{R}})\norm{\theta_1-\theta_2} _{(H^{{\kappa}})^*} + C_{\mathcal{U}} \norm{\beta_1-\beta_2}_{\mathcal{U}},\label{lipmodel}
      \end{gather}
      where $\theta_1,\theta_2$ are in $\mathcal{R}$, and $ \beta_1,\beta_2$ are in $\mathcal{U}$.
\end{condition}
Similarly, the conditional stability condition is given below.
\begin{condition}\label{condstabmodel}
    Consider a parameter space $\Theta \subseteq L^2(\mathcal{Z},\mathbb{R})$. 
  The forward map $\mathcal{G}_{\beta}:\Theta \rightarrow L^2_{\lambda}(\mathcal{X},V)$ is measurable. 
  Suppose for some normed linear subspace $(\mathcal{R},\Vert \cdot\Vert _{\mathcal{R}})$ of $\Theta$ and all $M>1$, there exist a function $F: \mathbb{R}^{+} \rightarrow \mathbb{R}^{+}$ and finite constants $C_T > 0$, $q\geq 0$ such that
        \begin{align}\label{stabmodel}
                  F(\Vert f_{\theta} - f_{\theta_0}\Vert_{\mathcal{F}} ) \leq C_TM^q\Vert \mathcal{G}_{\beta}(\theta)-\mathcal{G}_{\beta_0}(\theta_0)\Vert _{L^{2}_{\lambda}(\mathcal{X},V)},\quad \forall \theta \in \Theta \cap B_{\mathcal{R}}(M). 
          \end{align}
\end{condition}
With these properties in place, we are now in a position to offer a contraction estimate.
\begin{theorem}\label{mainthmModel}
  Suppose \cref{condregmodel} holds for the forward map $\mathcal{G}_{\beta}$, the normed linear subspace $(\mathcal{R},\Vert \cdot\Vert _{\mathcal{R}}) = (H^{\alpha}(\mathcal{Z}),\Vert \cdot\Vert _{H^{\alpha}(\mathcal{Z})})$ with integer $\alpha \geq 0$ and the finite constants $C_U > 0$, $C_L > 0$ , ${\kappa}\geq0$, $p\geq 0$, $l \geq 0$.
Let the probability measure $\Pi'_J$ as in \cref{seiveprior1}, and $J = J_N \in \mathbb{N}$ is such that $2^{J} \simeq N^{\frac{1}{2\alpha +2{\kappa} +d}}$. 
Denote by $\Pi_N$ the rescaled prior as in \cref{rescaledprior} with $\theta' \sim \Pi'_J$, and $\Pi_N(\cdot|(Y_i,X_i)_{i=1}^N) = \Pi_N(\cdot|D_N)$ is the corresponding posterior distribution. Assume that $\alpha + {\kappa} \geq \frac{d(l+1)}{2}$ and $\theta_0\in H_K^\alpha(\mathcal{Z})$.
Then for $\hat{\beta}$ and $\hat{Q}$ the solution to \cref{opplm}, $\varepsilon_N=N^{-\frac{\alpha+{\kappa}}{2\alpha+2{\kappa}+d}}$, we have
\begin{align*}
  P_{\beta_0,\theta_0}^{(N)}\hat{Q}\norm{\mathcal{G}_{\hat{\beta}}(\theta)-\mathcal{G}_{\beta}(\theta_0)}_{L^{2}_{\lambda}(\mathcal{X},\mathbb{R})}^{\frac{2}{p+1}}\lesssim \varepsilon_N^{\frac{2}{p+1}}\log N.
\end{align*}
Moreover, assume that \cref{condstabmodel} also holds for $\mathcal{G}_{\beta}$, $\mathcal{R}$, function $F$ and finite constants $C_T > 0$, $q\geq 0$. Then, we further have
\begin{align}
  P_{\beta_0,\theta_0}^{(N)}\hat{Q}[F(\Vert f_{\theta} - f_{\theta_0}\Vert_{\mathcal{F}} )]^{\frac{2}{p+q+1}}\lesssim \varepsilon_N^{\frac{2}{p+q+1}}\log N.
\end{align}
\end{theorem}
\subsection{Inverse potential problem with unknown fractional order}
In this part, we consider inverse potential problem as in \cref{SubsectionSuddiffuion} but with unknown fractional order.
We recall some settings here. Let domain $\Omega = (0,1)$ and we consider a subdiffusion equation with a nonhomogeneous Dirichlet boundary condition:
\begin{align*}
    \left\{\begin{aligned}
    &\partial_t^{\beta}u - \partial_{xx}u + qu = f \quad \mbox{in } \Omega \times (0,T],\\
    &u(0,t) = a_0, u(1,t) = a_1 \quad \mbox{on } (0,T],\\
    &u(0) = u_0 \quad \mbox{in } \Omega,
    \end{aligned}\right.  
\end{align*}
where $\beta \in (0,1)$ is the unknown fractional order $\beta$, $T > 0$ is a fixed final time, $f>0$ is given source term, $u_0>0$ is given initial data, the nonnegative function $q\in L^{\infty}(\Omega)$ is a spatially dependent potential, and $a_0$ and $a_1$ are positive constants.
For $\alpha\in\mathbb{N}$, we define
\begin{align*}
\mathcal{F}_{\alpha,M_0} =\left\{ q\in H^{\alpha} \cap \mathcal{I}: q\vert_{\partial\Omega}=1, \frac{\partial^jq}{\partial n^j}\Big|_{\partial\Omega} =0 \text{\ for\ } j=1,\dots,\alpha -1 \right\},
\end{align*}
where $\mathcal{I} = \{q \in L^{\infty}: 0< q < M_0\}$ for $M_0 >0$.
We assume $u_0,f \in H^{\alpha}(\Omega)$ with $u_0, f \geq L_0$ a.e. and $a_0, a_1 \geq L_0$ for $L_0>0$.
we also assume that the unknown fractional order $\beta$ belongs to $[\beta^{-},\beta^{+}]$ for any $\beta^{-},\beta^{+}$ such that $0<\beta^{-}<\beta^{+}<1$. The forward map $\mathcal{G}_{\beta}$ is defined by \cref{forwardmapFrac}.
\begin{theorem}\label{Fractionalfinal}
Let $\alpha \in \mathbb{N}$, $\alpha> 2+d/2$, $d = 1$, and ${\kappa}= 2$. Consider the forward map $\mathcal{G}_{\beta}$ as in \cref{forwardmapFrac} with $\beta\in [\beta^{-},\beta^{+}]$ for any $\beta^{-}, \beta^{+}$ such that $0<\beta^{-}<\beta^{+}<1$ and terminal time $T\geq T_0$ where $T_0$ is large enough. Let $\Pi_N$, $\Pi_N(\cdot|D_N)$, $\hat{\beta}$, $\hat{Q}$ be as defined in \cref{mainthmModel}. Assume that $\theta_0\in H_K^\alpha(\Omega)$. Then, for $\varepsilon_N=N^{-\frac{\alpha+{\kappa}}{2\alpha+2{\kappa}+d}}$, we have
\begin{gather}
    P_{\beta_0,\theta_0}^{(N)}\hat{Q} \Vert u_{\hat{\beta},q_{\theta}}(T)-u_{\beta_0,q_0}(T)\Vert_{L^2}^{2}\leq C_1 \varepsilon_N^{2}\log N,\label{FractionalGenerror}\\
    P_{\beta_0,\theta_0}^{(N)}\hat{Q}\norm{q_\theta - q_0}_{L^2(\Omega)}^{\frac{\alpha+2}{\alpha}\cdot\frac{2}{q+1}}\leq c\Big(T^{-\beta_0}\abs{\beta_0-\hat{\beta}} +|T^{-\beta_0}-T^{-\hat{\beta}}|\Big)^{\frac{\alpha+2}{\alpha}\cdot\frac{2}{q+1}}  +C_2\varepsilon_N^{\frac{2}{q+1}}\log N,\label{Fractionalerror}
\end{gather}
where $q = 2+4\alpha$, $C_1,C_2>0$, and $c$ is a positive constant independent of $T$.
\end{theorem}
Estimate \cref{FractionalGenerror} implies a contraction rate of variational posterior in ``prediction'' loss. Unfortunately, for the contraction towards the ground truth $q_0$, we did not verify \cref{condstabmodel} for the forward map $\mathcal{G_{\beta}}$. With the conditional stability estimate (see Appendix B.2)
\begin{align}\label{Fracstabilityestimate}
    \begin{aligned}
    \norm{q - q_0}_{L^2(\Omega)}\leq &c\Big((1+\norm{q}_{H^{\alpha}(\Omega)}^{1+2\alpha})^{\frac{2}{\alpha+2}}\norm{u_{q,\hat{\beta}}(T)-u_{q_0,\beta_0}(T)}_{L^2(\Omega)}^{\frac{\alpha}{\alpha+2}}\\
    &+T^{-\beta_0}\abs{\beta_0 - \hat{\beta}} +|T^{-\beta_0}-T^{-\hat{\beta}}|\Big)
\end{aligned}
\end{align}
we derive \cref{Fractionalerror} that gives an estimate of the contraction error of the variational posterior. Though we did not obtain a convergence rate, we can choose terminal time $T$ to make the terms $T^{-\beta_0}\abs{\beta_0-\hat{\beta}}$ and $|T^{-\beta_0}-T^{-\hat{\beta}}|$ as small as possible, which means that the approximation error of the variational posterior $\hat{Q}$ to the ground truth $q_0$ can be arbitrarily small. 
\par
{\color{black}Though we did not verify \cref{condstabmodel} for the forward map $\mathcal{G_{\beta}}$, this condition is not overly restrictive for inverse problems. Simultaneous recovery of two unknown parameters is widely studied, and conditional stability results satisfying our Condition 6.2 can be established for several problems \cite{simultaneous_yamamoto_simultaneous_2001,simultaneous_yang_multi-parameters_2020,simultaneous_ma_simultaneous_2024}. Specifically, see \cref{extracondition} in the Supplementary Material \cite{Zu2024AoS}, the time-fractional subdiffusion equation studied in \cite{simultaneous_ma_simultaneous_2024}.}
\begin{remark}
 Recently, Jing and Yamamoto \cite{Frac_jing2022simultaneous} demonstrated the simultaneous uniqueness of identifying multiple parameters (including fractional order and spatially dependent potential) in a one-dimensional subdiffusion/diffusion-wave ($0<\alpha<2$) equation. Previous studies have explored the conditional stability of recovering fractional order and spatially dependent potential from terminal measurements $u_{\beta,q}(T)$ respectively \cite{Frac_li2020stability,jin2023inverse}. However, to our knowledge, there is no conditional stability estimate for simultaneously recovering fractional order and spatially dependent potential that satisfies \cref{condstabmodel}. Our derived conditional stability estimate \cref{Fracstabilityestimate} suggests that for the terminal measure $u_{\beta,q}(T)$, which can always be represented by a time-dependent and a time-independent term, the time-dependent term vanishes as the terminal time $T$ becomes sufficiently large, and the information in the time-independent term can accurately recover the truth $q_0$ (see Appendix B.2). However, the estimate \cref{Fracstabilityestimate} did not verify \cref{condstabmodel} for the forward map $\mathcal{G_{\beta}}$ defined by \cref{forwardmapFrac}. This issue is beyond the scope of the current paper and remains an open question for future research.
\end{remark}
\par
From the proof of \cref{Fracminmax}, one can derive the following theorem which states that for the ``PDE-constrained regression'' problem of recovering $u_{\beta_0,q_0}(T)$ in ``prediction'' loss, the convergence rate obtained in \cref{FractionalGenerror} is minimax optimal (up to a logarithmic factor).
\begin{theorem}\label{Fractinalminmax}
For $M_0>1$, $\alpha \in \mathbb{N}$, $q\in\mathcal{F}_{\alpha,M_0}$, $\beta \in [\beta^-,\beta^+]$, consider the solution $u_{\beta,q}(t)$ of the problem \cref{Fractional}. Then there exist fixed $T_0>0$ and a finite constant $C>0$ such that for $N$ large enough, the terminal time $T\geq T_0$ and any $\eta>0$,
\begin{align*} 
\inf_{\Tilde{u}_N}\sup_{\substack{
\beta_0 \in [\beta^-,\beta^+] \\
q_0\in\mathcal{F}_{\alpha,M_0}(R)
}}P^{(N)}_{\beta_0,\theta_0}\hat{Q}\norm{\Tilde{u}_N-u_{\beta_0,q_0}(T)}_{L^2(\Omega)}^{\eta}\geq C N^{-\frac{\alpha+2}{2\alpha+4+1}\cdot\eta},
\end{align*}
where $\theta_0 = \Phi^{-1}(q_0)$ and the infimum ranges over all measurable functions $\Tilde{u}_N = \Tilde{u}_N(\theta)$ that take value in $L^2(\Omega)$ with $\theta$ from the variational posterior $\hat{Q}$ defined in \cref{Fractionalfinal}. 
\end{theorem}

\bibliographystyle{siamplain}
\bibliography{references}

\end{document}


\maketitle

\section{Proofs of results for general inverse problems}
\begin{proof}[Proof of Proposition \ref{le2.1}]
The proofs of \eqref{KL} and \eqref{hellinger} can be founded in \cite[Proposition 1.3.1]{IntroNonLinear_nickl2023bayesian}.
We first denote the probability density of a standard multivariate normal variable $Z$ on the finite-dimensional vector space $V$ by
$\phi_V (z) = (2\pi)^{-p_V/2}e^{-\abs{z}_V^2/2}$. Using the standard identity $Ee^{\pdt{u}{Z}_V} = e^{\abs{u}_V^2/2}, u\in V$, we have
\begin{equation*}
    \begin{aligned}
        &\mexp{D_2(P_{\theta_1}\Vert P_{\theta_2})} =  \int_{V\times\mathcal{X}}\frac{p_{\theta_1}}{p_{\theta_2}}p_{\theta_1}d\mu \\
        = & \frac{1}{(2\pi)^{p_V/2}}\int_{V\times\mathcal{X}}\mexp{\frac{1}{2}\abs{y-\mathcal{G}(\theta_2)}_V^2-\abs{y-\mathcal{G}(\theta_1)}_V^2}d\mu(y,x)\\
        =& \int_{\mathcal{X}}\mexp{\frac{1}{2}\abs{\mathcal{G}(\theta_2)}_V^2-\abs{\mathcal{G}(\theta_1)}_V^2}\int_{V}e^{\pdt{y}{2\mathcal{G}(\theta_1)(x)-\mathcal{G}(\theta_2)(x)}}\phi_V(y) dy d\lambda(x)\\
        =& \int_{\mathcal{X}} \mexp{\frac{1}{2}\abs{\mathcal{G}(\theta_2)}_V^2-\abs{\mathcal{G}(\theta_1)}_V^2 + \frac{1}{2}\abs{2\mathcal{G}(\theta_1)-\mathcal{G}(\theta_2)}_V^2}d\lambda(x)\\
    = &\int_{\mathcal{X}}\mexp{\abs{\mathcal{G}(\theta_1)(x)-\mathcal{G}(\theta_2)(x)}_V^2}d\lambda(x),
    \end{aligned}
\end{equation*}
that is
    \[
    D_2(P_{\theta_1}\Vert P_{\theta_2}) = 
    \log\int_{\mathcal{X}}\mexp{\abs{\mathcal{G}(\theta_1)(x)-\mathcal{G}(\theta_2)(x)}_V^2}d\lambda(x).\]
Combined with an upper bound of $e^z$ given by Taylor {\color{black}expansion}:
        \begin{equation*}
            e^z = 1 + ze^{\eta z} \leq 1 + ze^{z}, \quad \mbox{for}\ z \geq 0,
        \end{equation*}
we deduce that
\begin{equation*}
    \begin{aligned}
        D_2(P_{\theta_1}\Vert P_{\theta_2}) \leq &
    \int_{\mathcal{X}}\abs{\mathcal{G}(\theta_1)(x)-\mathcal{G}(\theta_2)(x)}_V^2\mexp{\abs{\mathcal{G}(\theta_1)(x)-\mathcal{G}(\theta_2)(x)}_V^2}d\lambda(x)\\
    \leq & e^{4U^2} \norm{\mathcal{G}(\theta_1)-\mathcal{G}(\theta_2)}_{L^2_{\lambda}(\mathcal{X},V)}^2.
    \end{aligned}
\end{equation*}
\end{proof}
\subsection{Contraction rates for rescaled Gaussian priors}
\begin{proof}[Proof of Theorem \ref{mainthm}]
We are going to verify the three condition in \cite[Theorem 2.1]{zhang2020convergence}. Steps (i) to (iii) below verify conditions (C1) to (C3) directly.
\par
        We denote $U_{\mathcal{G}}(M),L_{\mathcal{G}}(M),T_{\mathcal{G}}(M)$ by $C_UM^p,C_LM^l,C_TM^q$ respectively.
        Let \[H_N(\varepsilon) = \{\theta: \theta = \theta_1 + \theta_2, \norm{\theta_1}_{(H^{{\kappa}})^*}\leq M^s\varepsilon_N/L_{\mathcal{G}}(M\varepsilon/\varepsilon_N), \Vert \theta_2 \Vert_{\mathcal{H}}\leq M\varepsilon/\varepsilon_{N}\}\] with $s = l + 1$, 
        and we further define $\Theta_N(\varepsilon) = H_N(\varepsilon) \cap B_{\mathcal{R}}(M\varepsilon/\varepsilon_{N})$
        for some $M$ large enough and $\varepsilon > \varepsilon_N$.
        \par
        (i) For (C1), we follow the method used in the proof of \cite[Theorem 7.1.4]{gin2015mathematical}. Let 
        \[S_j = \{\theta \in \Theta_N(\varepsilon):4j \Bar{m} \varepsilon \leq h(p_{\theta},p_{\theta_0})< 4(j+1)\Bar{m}\varepsilon\}, \quad j \in \mathbb{N}, \quad \Bar{m} > 0.\]
        We see that $S_j \subseteq \Theta_N(\varepsilon),$
        so it is sufficient to consider the metric entropy of $\Theta_N(\varepsilon)$.
        Here we introduce the (semi-) metric 
        \[d_{\mathcal{G}}({\theta_1},{\theta_2}) := \Vert \mathcal{G}(\theta_1)-\mathcal{G}(\theta_2)\Vert _{L^2_{\lambda}(\mathcal{X},V)}.\]
        Using Proposition \ref{le2.1}, formula (4.184) in \cite{gin2015mathematical} and Lipschitz condition of $\mathcal{G}$, 
        we have
        \begin{align*}
            \log N(\Theta_N(\varepsilon),h,j\Bar{m}\varepsilon) &\leq \log N(\Theta_N(\varepsilon),h,\Bar{m}\varepsilon)\\
            &\leq \log N(\Theta_N(\varepsilon),d_{\mathcal{G}},2\Bar{m}\varepsilon)\\
            &\leq \log N\Big(\Theta_N(\varepsilon),\norm{\ \cdot \ }_{H^{\kappa})^*},\frac{2\Bar{m}\varepsilon}{C_L(M\varepsilon/\varepsilon_N)^l}\Big).
        \end{align*}
        Because $\Theta_N(\varepsilon) \subset H_N(\varepsilon)$, we only need to bound 
        \[\log N\Big(H_N(\varepsilon),\norm{\ \cdot \ }_{(H^{\kappa})^*},\frac{2\Bar{m}\varepsilon}{C_L(M\varepsilon/\varepsilon_N)^l}\Big).\]
        The definition of $H_N(\varepsilon)$ implies that a $\frac{\Bar{m}\varepsilon}{C_L(M\varepsilon/\varepsilon_N)^l}$-covering in $\norm{\ \cdot \ }_{(H^{\kappa})^*}$ of $B_{\mathcal{H}}(M\varepsilon/\varepsilon_{N})$ is a $\frac{2\Bar{m}\varepsilon}{C_L(M\varepsilon/\varepsilon_N)^l}$-covering in $\norm{\ \cdot \ }_{(H^{\kappa})^*}$ of $H_N(\varepsilon)$ for $\Bar{m} \geq M^s$. 
        It is sufficient to consider \[\log N\Big(B_{\mathcal{H}}(M\varepsilon/\varepsilon_{N}),\norm{\ \cdot \ }_{(H^{\kappa})^*},\frac{\Bar{m}\varepsilon}{C_L(M\varepsilon/\varepsilon_N)^l}\Big).\]
        We record (6.25) in \cite{IntroNonLinear_nickl2023bayesian} as follows:
        \begin{equation}\label{entropy}
            \log N(B_{H_c^{\alpha}}(r), \norm{\ \cdot \ }_{(H^{\kappa})^*}, \delta) \leq C_E\left(\frac{r}{\delta}\right)^{d/(\alpha + {\kappa})},\quad 0< \delta < r,\quad {\kappa} \geq 1/2
        \end{equation}
        with $C_E$ depending on $\alpha, {\kappa}, d$.
        Using the embedding $\norm{\ \cdot \ }_{H^{\alpha}_c} \leq c\norm{\ \cdot \ }_{\mathcal{H}}$ and (\ref{entropy}), for ${\kappa}\geq 1/2$, $\bar{m} = cC_LM^s$, and $\varepsilon > \varepsilon_N$, we deduced that
        \begin{equation} \label{entropybound}
            \begin{aligned} 
            &\log N(B_{\mathcal{H}}(M\varepsilon/\varepsilon_{N}),\norm{\ \cdot \ }_{(H^{\kappa})^*},\frac{\Bar{m}\varepsilon}{C_L(M\varepsilon/\varepsilon_N)^l}) \\
            \leq &\log N(B_{H^{\alpha}}(cM\varepsilon/\varepsilon_{N}),\norm{\ \cdot \ }_{(H^{\kappa})^*},\frac{\Bar{m}\varepsilon}{C_L(M\varepsilon/\varepsilon_N)^l})\\
            \leq &C_E\left(\frac{cM C_L(M\varepsilon/\varepsilon_N)^l}{\Bar{m}\varepsilon_N}\right)^{d/(\alpha +{\kappa})}\\
            \leq &C_E\left(\frac{cM^sC_{L}}{\Bar{m}}\right)^{d/(\alpha+{\kappa})} \cdot N\varepsilon^2 \\
            \leq &C_EN\varepsilon^2,
            \end{aligned}
        \end{equation}          
        where we used $\varepsilon/\varepsilon_N>1$, $ld/({\alpha + \kappa}) < 2$, and $\varepsilon_N^{-d/(\alpha+\kappa)} = N\varepsilon_N^2$.
        We note that (\ref{entropy}) is also true for ${\kappa} < 1/2$ with $H^{\alpha}_c$ replaced by $H^{\alpha}$ as discussed after (6.25) in \cite{IntroNonLinear_nickl2023bayesian}.
        Thus, we have 
        \[N(S_j,h,j\Bar{m}\varepsilon)\leq N(\varepsilon) := \exp\{C_EN\varepsilon^2\}.\]
        Choose a minimum finite set $S_j^{'}$ of points in each set $S_j$
        such that every $\theta \in S_j$ is within Hellinger distance $j\Bar{m}\varepsilon$ of at least one of these points. By metric entropy bound above, for $j$ fixed, there are at most $N(j\varepsilon)$ such points $\theta_{jl} \in S_j^{'}$, and from \cite[Corollary 7.1.3]{gin2015mathematical} for each $\theta_{jl}$ there exists a test $\Psi_{N,jl}$ such that
            \[P^{(N)}_{\theta_0}\Psi_{N,jl} \leq e^{-C_tNj^2\Bar{m}^2\varepsilon^2}, \quad \mathop{\sup}_{\theta \in S_j, h(p_{\theta},p_{\theta_{jl}})<j\Bar{m}\varepsilon} P^{(N)}_{\theta}(1-\Psi_{N,jl}) \leq e^{-C_tNj^2\Bar{m}^2\varepsilon^2}\]
        for some universal constant $C_t > 0$. Let $\Psi_N = \mathop{\max}_{j,l}\Psi_{N,jl}$. Then, we have
        \begin{align}
            P^{(N)}_{\theta_0}\Psi_N & \leq P^{(N)}_{\theta_0}(\sum_j\sum_l\Psi_{N,jl}) \leq \sum_j\sum_l\exp\{-C_tNj^2\Bar{m}^2\varepsilon^2\} \nonumber \\
            &\leq N(\varepsilon)\sum_j  \exp\{-C_tNj^2\Bar{m}^2\varepsilon^2\}  \label{test1}\\
            &\leq \frac{1}{1-\exp\{-(C_t\Bar{m}^2-C_E)\}}\exp\{-(C_t\Bar{m}^2-C_E)N\varepsilon^2\}  \nonumber\\
            &\leq \exp\{-CN\varepsilon^2\} \nonumber
        \end{align}
        and
        \begin{equation}\label{test}
            \mathop{\sup}_{\begin{array}{c}
                 \theta \in \Theta_N(\varepsilon)  \\
                 h(p_{\theta},p_{\theta_0}) \geq 4\bar{m}\varepsilon
            \end{array}}
            P^{(N)}_{\theta}(1-\Psi_N) = \mathop{\sup}_{\theta \in \cup_j S_j}P^{(N)}_{\theta}(1-\Psi_N)\leq \exp\{-C N \varepsilon^2\}
        \end{equation}
        for any $C>0$ when $M$ is large enough.
        Using Proposition \ref{le2.1}, Conditions \ref{condreg} and \ref{condstab}, 
        we have the following inequality for $\theta \in \Theta_N(\varepsilon)$ and constants $C_{U},C_{T}$ from conditions (\ref{bound}) and (\ref{stab}):
        \begin{align*}
            h(p_{\theta},p_{\theta_0}) 
        &\geq \frac{1}{2C_{U}\cdot(M\varepsilon/\varepsilon_N)^p}\Vert \mathcal{G}(\theta)-\mathcal{G}(\theta_0)\Vert _{L^{2}_{\lambda}(\mathcal{X},V)}\\
        &\geq \frac{1}{2C_{U}C_{T}\cdot(M\varepsilon/\varepsilon_N)^{p+q}}F(\Vert f_{\theta} - f_{\theta_0}\Vert ).
        \end{align*}
        From the inequality and direct calculations, we note that
        \[\{\theta \in \Theta_N(\varepsilon) : h(p_{\theta},p_{\theta_0}) \geq 4\bar{m}\varepsilon\} \supseteq
        \{\theta \in \Theta_N(\varepsilon) : N\varepsilon_N^{\frac{2p}{p+1}}\Vert \mathcal{G}(\theta)-\mathcal{G}(\theta_0)\Vert _{L^{2}_{\lambda}}^{\frac{2}{p+1}} \geq \Tilde{C}_1 N \varepsilon^2 \}\]
        and
        \[\{\theta \in \Theta_N(\varepsilon) : h(p_{\theta},p_{\theta_0}) \geq 4\bar{m}\varepsilon\} \supseteq
        \{\theta \in \Theta_N(\varepsilon) : N\varepsilon_N^{\frac{2p+2q}{p+q+1}}[F(\Vert f_{\theta} - f_{\theta_0}\Vert )]^{\frac{2}{p+q+1}} \geq C_1 N \varepsilon^2 \}\]
        for $\Tilde{C}_1 = (8\Bar{m}C_UM^{p})^{\frac{2}{p+1}}$ and $C_1 = (8\Bar{m}C_U C_TM^{p+q})^{\frac{2}{p+q+1}}$. Combined with (\ref{test}), we have 
        \begin{equation}
            \sup_{\substack{\theta \in \Theta_N(\varepsilon)  \\
                 L(P^{(N)}_{\theta},P^{(N)}_{\theta_0}) \geq N \varepsilon^2}}
            P^{(N)}_{\theta}(1-\Psi_N)\leq \exp\{-C N \varepsilon^2\}
        \end{equation}
        with  
        \[L(P^{(N)}_{\theta},P^{(N)}_{\theta_0}) := N\varepsilon_N^{\frac{2p+2q}{p+q+1}}[F(\Vert f_{\theta} - f_{\theta_0}\Vert )]^{\frac{2}{p+q+1}}/C_1 \text{ or } N\varepsilon_N^{\frac{2p}{p+1}}\Vert \mathcal{G}(\theta)-\mathcal{G}(\theta_0)\Vert _{L^{2}_{\lambda}}^{\frac{2}{p+1}}/\Tilde{C}_1.\]
        Thus, we have confirmed (C1) with the statement provided above.

        (ii) By the definition of $\Theta_N(\varepsilon)$, we deduce that 
        \begin{align*}
            \Pi_N(\Theta_N(\varepsilon)^c) &= \Pi_N(H_N(\varepsilon)^c \cup B_{\mathcal{R}}(M\varepsilon/\varepsilon_{N})^c) \\
            & \leq \Pi_N(H_N(\varepsilon)^c) + \Pi_N(B_{\mathcal{R}}(M\varepsilon/\varepsilon_{N})^c).
        \end{align*}
        By Fernique’s theorem , for $\theta'\sim \Pi'$,
        $E\Vert \theta'\Vert _{\mathcal{R}}\leq D$ with a constant $D$ only depends on prior $\Pi'$ (see the proof of \cite[Theorem 2.2.2]{IntroNonLinear_nickl2023bayesian}).
        Next, using \cite[Theorem 2.1.20]{gin2015mathematical}, we have
        \begin{align*}
            &\Pi_N(\Vert \theta\Vert _{\mathcal{R}}>M\varepsilon/\varepsilon_N) = \Pi'(\Vert \theta'\Vert _{\mathcal{R}}>M\sqrt{N}\varepsilon)\\
            &\leq \Pi'(\Vert \theta'\Vert _{\mathcal{R}}-E\Vert \theta'\Vert _{\mathcal{R}}>M\sqrt{N}\varepsilon/2)\\
            &\leq 2\exp(-cM^2N\varepsilon^2) \leq \frac{1}{2}\exp(-CN\varepsilon^2)
        \end{align*}
        for the constant $C$ in (i) and $M$ large enough. Then it is sufficient to prove 
        \[\Pi_N(H_N(\varepsilon)) \geq 1 - \exp\{-BN\varepsilon^2\} \geq 1 - \frac{1}{2}\exp\{-CN\varepsilon^2\}\]
        for $B = C+2$.
        Using (\ref{entropy}) as the proof in (i), we have
        \[\log N(B_{\mathcal{H}}(1),\norm{\ \cdot \ }_{(H^{\kappa})^*},\delta)\lesssim \delta^{-d/(\alpha+{\kappa})},\quad 0<\delta<1.\]
        Then, we apply Theorem 6.2.1 in \cite{IntroNonLinear_nickl2023bayesian} to obtain 
        \begin{align*}
            &-\log\Pi'(\theta' : \norm{\theta'}_{(H^{\kappa})^*} \leq M^s\sqrt{N}\varepsilon_N^2/L_{\mathcal{G}}(M\varepsilon/\varepsilon_N)) \\
            &\leq c(M^s\sqrt{N}\varepsilon_N^2/L_{\mathcal{G}}(M\varepsilon/\varepsilon_N))^{-\frac{2d}{2\alpha+2{\kappa}-d}}\\
            &\leq c[M^s/L_{\mathcal{G}}(M\varepsilon/\varepsilon_N)]^{-\frac{2d}{2\alpha+2{\kappa}-d}}N\varepsilon_N^2
        \end{align*}
        for some $c > 0$.
        Combined with Lipschitz condition (\ref{lip}), we deduced that 
        \[-\log\Pi'(\theta' : \norm{\theta'}_{(H^{\kappa})^*} \leq M^s\sqrt{N}\varepsilon_N^2/L_{\mathcal{G}}(M\varepsilon/\varepsilon_N)) \leq cC_L^{\frac{2d}{2\alpha+2{\kappa}-d}}M^{-\frac{2d}{2\alpha+2{\kappa}-d}}N\varepsilon^2.\]
        Let $M \geq C_L\cdot\left(\frac{B}{c}\right)^{-\frac{2\alpha+2{\kappa} - d}{2d}}$ and we have
        \begin{equation}\label{pbabound}
            -\log\Pi'(\theta' : \norm{\theta'}_{(H^{\kappa})^*} \leq M^s\sqrt{N}\varepsilon_N^2/L_{\mathcal{G}}(M\varepsilon/\varepsilon_N)) \leq BN\varepsilon^2.
        \end{equation}
        We denote\[B_N = -2\Phi(e^{-BN\varepsilon^2_N}),\]
        where $\Phi$ is the standard normal cumulative distribution function. By \cite[Lemma K.6]{IntroMy_ghosal2017fundamentals} we have
        \[\sqrt{BN}\varepsilon \leq B_N \leq 2\sqrt{2BN}\varepsilon.\]
        Hence, for $M \geq 2\sqrt{2B}$, we obtain
        \begin{align*}
            &\Pi'(\theta' : \theta' = \theta'_1 + \theta'_2, \norm{\theta'_1}_{(H^{\kappa})^*} \leq M^s\sqrt{N}\varepsilon_N^2/L_{\mathcal{G}}(M\varepsilon/\varepsilon_N), \Vert \theta'_2 \Vert_{\mathcal{H}}\leq M\sqrt{N}\varepsilon)\\
            &\geq \Pi'(\theta' : \theta' = \theta'_1 + \theta'_2, \norm{\theta'_1}_{(H^{\kappa})^*} \leq M^s\sqrt{N}\varepsilon_N^2/L_{\mathcal{G}}(M\varepsilon/\varepsilon_N), \Vert \theta'_2 \Vert_{\mathcal{H}}\leq B_N).
        \end{align*}
        By (\ref{pbabound}) and \cite[Theorem 2.6.12]{gin2015mathematical},
        \begin{align*}
            &\Pi'(\theta' : \theta' = \theta'_1 + \theta'_2, \norm{\theta'_1}_{(H^{\kappa})^*} \leq M^s\sqrt{N}\varepsilon_N^2/L_{\mathcal{G}}(M\varepsilon/\varepsilon_N), \Vert \theta'_2 \Vert_{\mathcal{H}}\leq B_N) \\
            &\geq \Phi(\Phi^{-1}[\Pi'(\norm{\theta'_1}_{(H^{\kappa})^*} \leq M^s\sqrt{N}\varepsilon_N^2/L_{\mathcal{G}}(M\varepsilon/\varepsilon_N))] + B_N) \\
            &\geq \Phi(\Phi^{-1}[e^{-BN\varepsilon^2}] + B_N) = 1 - e^{-BN\varepsilon^2}
        \end{align*}
        that is
        \[\Pi_N(H_N(\varepsilon)) \geq 1 - e^{-BN\varepsilon^2}.\]
        Finally, we have 
        \begin{equation}\label{setbound}
            \Pi_N(\Theta_N(\varepsilon)^c) \leq e^{-CN\varepsilon^2}.
        \end{equation}
        
        (iii) Next, we check (C3) for $\rho = 2$. Relying on Proposition \ref{le2.1}, we have
        \begin{equation*}
        D_2(P_{\theta_0}\Vert P_{\theta}) \leq e^{2(\norm{\mathcal{G}(\theta_0)}_{\infty}^2+\norm{\mathcal{G}(\theta)}_{\infty}^2)} \norm{\mathcal{G}(\theta_0)-\mathcal{G}(\theta)}_{L^2_{\lambda}(\mathcal{X},V)}^2.
        \end{equation*}
        Employing the above inequality and (\ref{bound}), we obtain
        \begin{align*}
            &\Pi_N(\theta : D_2(P^{(N)}_{\theta_0}\Vert P^{(N)}_{\theta})\leq N\varepsilon_N^2)\\
            &\geq \Pi_N(\theta : D_2(P_{\theta_0}\Vert P_{\theta})\leq \varepsilon_N^2, \norm{\theta-\theta_0}_{\mathcal{R}}\leq M')\\
            &\geq \Pi_N(\theta : \exp\{4U^2_{\mathcal{G}}(\Bar{M})\}\norm{\mathcal{G}(\theta)-\mathcal{G}(\theta_0)}^2_{L^{2}_{\lambda}}\leq \varepsilon_N^2, \norm{\theta-\theta_0}_{\mathcal{R}}\leq M')\\
            &\geq \Pi_N(\theta : \norm{\mathcal{G}(\theta)-\mathcal{G}(\theta_0)}_{L^{2}_{\lambda}}\leq \exp\{-2U_{\mathcal{G}}^2(\Bar{M})\}\varepsilon_N, \norm{\theta-\theta_0}_{\mathcal{R}}\leq M')
        \end{align*}
        for some constant $M'>0$ and $\bar{M}=M' + \Vert \theta_0\Vert _{\mathcal{R}}$.
        Then, using Lipschitz condition of $\mathcal{G}$, Corollary 2.6.18 in \cite{gin2015mathematical} and Theorem 6.2.2 in \cite{IntroNonLinear_nickl2023bayesian}, we have
        \begin{align*}
            &\Pi_N(\theta : D_2(P^{(N)}_{\theta_0}\Vert P^{(N)}_{\theta})\leq N\varepsilon_N^2)\\
            &\geq \Pi_N(\theta :\norm{\theta-\theta_0}_{(H^{\kappa})^*} \leq \exp\{-2U_{\mathcal{G}}^2(\Bar{M})\}\varepsilon_N/L_{\mathcal{G}}(\bar{M}), \Vert \theta-\theta_0\Vert _{\mathcal{R}}\leq M')\\
            &\geq \exp(-\frac{1}{2}N\varepsilon_N^2\Vert \theta_0\Vert _{\mathcal{H}}^2)\cdot\Pi_N(\theta :\norm{\theta}_{(H^{\kappa})^*}\leq C_{\mathcal{G}}({\Bar{M}})\varepsilon_N, \Vert \theta\Vert _{\mathcal{R}}\leq M')\\
            &\geq \exp(-\frac{1}{2}N\varepsilon_N^2\Vert \theta_0\Vert _{\mathcal{H}}^2)\cdot \Pi_N(\theta :\norm{\theta}_{(H^{\kappa})^*}\leq C_{\mathcal{G}}({\Bar{M}})\varepsilon_N)\cdot\Pi_N(\Vert \theta\Vert _{\mathcal{R}}\leq M')
        \end{align*}
        for $C_{\mathcal{G}}({\Bar{M}}) = \exp\{-2U_{\mathcal{G}}^2(\Bar{M})\}/L_{\mathcal{G}}(\bar{M})$.
        From the preceding proof in (ii), we have 
        \[\Pi_N(\Vert \theta\Vert _{\mathcal{R}}\leq M') = \Pi'(\Vert \theta'\Vert _{\mathcal{R}}\leq M'\sqrt{N}\varepsilon_N)\geq 1-2\exp(-cM'N\varepsilon_N^2)\geq 1/2\]
        for $M'$ large enough. So it is sufficient to give the lower bound of 
        \[\Pi_N(\theta :\norm{\theta}_{(H^{\kappa})^*}\leq C_{\mathcal{G}}({\Bar{M}})\varepsilon_N)=\Pi'(\theta' :\norm{\theta'}_{(H^{\kappa})^*}\leq C_{\mathcal{G}}({\Bar{M}})\sqrt{N}\varepsilon_N^2).\]
        Using \cite[Theorem 6.2.1]{IntroNonLinear_nickl2023bayesian} as in (ii), we find that for some $c>0$
        \[\Pi'(\theta' :\norm{\theta'}_{(H^{\kappa})^*}\leq C_{\mathcal{G}}({\Bar{M}})\sqrt{N}\varepsilon_N^2)\geq \exp(-cN\varepsilon_N^2).\]
        In short, we have the lower bound
        \[\Pi_N(\theta : D_2(P^{(N)}_{\theta_0}\Vert P^{(N)}_{\theta})\leq N\varepsilon_N^2)\geq \frac{1}{2}\exp(-(c+\Vert \theta_0\Vert ^2_{\mathcal{H}}/2)N\varepsilon_N^2)\geq \exp(-C_2N\varepsilon_N^2)\]
        for some constant $C_2>0$. Thus, we have proved the condition (C3).
        
        Because the three conditions are verified, Theorem 2.1 in \cite{zhang2020convergence} gives
        \begin{equation*}
      P_{\theta_0}^{(N)}\hat{Q}N\varepsilon_N^{\frac{2p}{p+1}}\Vert \mathcal{G}(\theta)-\mathcal{G}(\theta_0)\Vert _{L^{2}_{\lambda}}^{\frac{2}{p+1}}\lesssim N(\varepsilon_N^2 +\gamma^2_N),
        \end{equation*}
        and
        \[P_{\theta_0}^{(N)}\hat{Q}N\varepsilon_N^{\frac{2p+2q}{p+q+1}}[F(\Vert f_{\theta} - f_{\theta_0}\Vert )]^{\frac{2}{p+q+1}}\lesssim N(\varepsilon_N^{2} +\gamma^2_N).\]
        That is to say, we obtain
        \[P_{\theta_0}^{(N)}\hat{Q}\Vert \mathcal{G}(\theta)-\mathcal{G}(\theta_0)\Vert _{L^{2}_{\lambda}}^{\frac{2}{p+1}}\lesssim \varepsilon_N^{\frac{2}{p+1}} +\gamma^2_N \cdot \varepsilon_N^{-\frac{2p}{p+1}},\]
        and
        \[P_{\theta_0}^{(N)}\hat{Q}[F(\Vert f_{\theta} - f_{\theta_0}\Vert )]^{\frac{2}{p+q+1}}\lesssim \varepsilon_N^{\frac{2}{p+q+1}} +\gamma^2_N \cdot \varepsilon_N^{-\frac{2p+2q}{p+q+1}}.\]
\end{proof}
\begin{proof}[Proof of Theorem 3.5]
        As $D_2(P^{(N)}_{\theta_0}\Vert P^{(N)}_{\theta}) \geq D(P^{(N)}_{\theta_0}\Vert P^{(N)}_{\theta})$ , we have
        \[\tilde{\Theta}_N := \set{\theta : D_2(P^{(N)}_{\theta_0}\Vert P^{(N)}_{\theta}) \leq N\varepsilon_N^2, \textcolor{black}{\Vert \theta-\theta_0\Vert _{\mathcal{R}}\leq M}} \subseteq \set{\theta : D(P^{(N)}_{\theta_0}\Vert P^{(N)}_{\theta}) \leq N\varepsilon_N^2}\]
        From the proof of (C3) in Theorem \ref{mainthm} , we know, for $M$ large enough,
        \[\Pi_N(\tilde{\Theta}_N) \geq \exp\{-C_2N\varepsilon_N^2\}.\]
        We define probability measure $Q_N$ with
        \[ Q_N(B) = \frac{\Pi_N(B \cap \tilde{\Theta}_N)}{\Pi_N(\tilde{\Theta}_N)}\]
        for any measurable set $B$.
        Since $\mathrm{supp}(Q_N)\subseteq\tilde{\Theta}_N \subseteq B_{\mathcal{R}}(M')$, $Q_N$ is a restricted Gaussian measure.
        Follow the proof of \cite[Theorem 2.4]{zhang2020convergence} and we deduce that
        \[R(Q_N) \leq (1+C_2)\varepsilon_N^2.\] \textcolor{black}{Here we give the detailed proof. Because \[\mathrm{supp}(Q_N)\subseteq \tilde{\Theta}_N \subseteq \set{\theta : D(P^{(N)}_{\theta_0}\Vert P^{(N)}_{\theta}) \leq N\varepsilon_N^2},\]
        we have 
        \[ Q_ND(P^{(N)}_{\theta_0}\Vert P^{(N)}_{\theta}) \leq N\varepsilon_N^2.\]
        Moreover, we bound $D(Q_N\Vert\Pi_N)$ as below
        \[ D(Q_N\Vert\Pi_N) = Q_N\log \frac{dQ_N}{d\Pi_N} = -\log \Pi_N(\tilde{\Theta}_N)\leq C_2N\varepsilon_N^2,\]
        where we used the definition of $Q_N$ and $\Pi_N(\tilde{\Theta}_N) \geq \exp\{-C_2N\varepsilon_N^2\}$.
        Then, we have
        \[R(Q_N) = \frac{1}{N}\Big(D(Q_N\Vert\Pi_N)+Q_ND(P^{(N)}_{\theta_0}\Vert P^{(N)}_{\theta})\Big) \leq (1+C_2)\varepsilon_N^2.\]}    \end{proof}
\subsection{Contraction rates for High-dimensional Gaussian sieve priors}
For the consistency theorem with the Gaussian sieve prior, we need the bias \(\|\theta_0 - P_J(\theta_0)\|\) to be controlled by the contraction rate \(\varepsilon_N\) of the variational posterior, which is considered in the following lemma.
\begin{lemma}\label{0con}
Assume $\norm{\theta_0}_{H^{\alpha}(\mathcal{Z})}\leq B$ for some constant $B > 0$, and $2^{J} \simeq N^{\frac{1}{2\alpha +2{\kappa} +d}}$. Then,
\begin{equation}
    \norm{\theta_0 - P_J(\theta_0)}_{(H^{{\kappa}}(\mathcal{Z}))^*}\lesssim \varepsilon_N
\end{equation}
for $\varepsilon_N=N^{-\frac{\alpha+{\kappa}}{2\alpha+2{\kappa}+d}}$.
\end{lemma}
\begin{proof}
This lemma is a simple corollary of (B6) in \cite{IntroNonLinear_giordano2020consistency}, which gives
\begin{equation*}
\norm{\theta_0 - P_J(\theta_0)}_{(H^{{\kappa}}(\mathcal{Z}))^*}\lesssim 2^{-J(\alpha+{\kappa})}\norm{\theta_0}_{H^{\alpha}(\mathcal{Z})}.
\end{equation*}
Because $2^{J} \simeq N^{\frac{1}{2\alpha +2{\kappa} +d}}$, we deduce 
\begin{equation*}
\norm{\theta_0 - P_J(\theta_0)}_{(H^{\kappa}(\mathcal{Z}))^*}\lesssim \varepsilon_N.
\end{equation*}
\end{proof}
\begin{proof}[Proof of Theorem \ref{mainthmsv}]
Following the proof of Theorem \ref{mainthm}, we will now verify conditions (C1) through (C3) as outlined in \cite[Theorem 2.1]{zhang2020convergence}. We define the sets \( H_N(\varepsilon) \) and \( \Theta_N(\varepsilon) \) as stated in Theorem \ref{mainthm}. Since the finite-dimensional Gaussian prior \( \Pi_N \) is supported within the reproducing kernel Hilbert space (RKHS) \( \mathcal{H}_J \), the proof for conditions (C1) and (C2) in Theorem \ref{mainthm} also holds true for the set \( \Theta_N(\varepsilon) \), with the set being replaced by \( \mathcal{H}_J \cap \Theta_N(\varepsilon) \). However, the proof for condition (C3) requires a slight modification because \( \theta_0 \) is no longer within the RKHS \( \mathcal{H}_J \).

Using Lemma \ref{0con} and Lipschitz condition (\ref{lip}), we deduce that
\begin{equation*}
    \norm{\mathcal{G}(\theta_0)-\mathcal{G}(P_J(\theta_0))}_{L^{2}_{\lambda}(\mathcal{X},\mathbb{R})} \leq C_0 \varepsilon_N
\end{equation*}
for some constant $C_0$. Thus, using Proposition \ref{le2.1}, we have
\begin{align*}
            &\Pi_N(\theta : D_2(P^{(N)}_{\theta_0}\Vert P^{(N)}_{\theta})\leq C_3N\varepsilon_N^2)\\
            &\geq \Pi_N(\theta : \norm{\mathcal{G}(\theta)-\mathcal{G}(\theta_0)}_{L^{2}_{\lambda}}\leq C_3\exp\{-2U_{\mathcal{G}}^2(\Bar{M})\}\varepsilon_N, \norm{\theta-\theta_0}_{\mathcal{R}}\leq M')\\
            &\geq \Pi_N(\theta :\norm{\mathcal{G}(\theta)-\mathcal{G}(P_J(\theta_0))}_{L^2_{\lambda}}\leq (C_3\exp\{-2U_{\mathcal{G}}^2(\Bar{M})\}-C_0)\varepsilon_N, \norm{\theta-\theta_0}_{\mathcal{R}}\leq M')
\end{align*}
where the lower bound of the last probability is given in the proof of Theorem \ref{mainthm}.
\end{proof}
\begin{proof}[Proof of Theorem \ref{boundgamfinite}]
The notations $\Tilde{\mathcal{Q}}_G^J, \Psi_J, {\mathcal{Q}}_G^J$ are defined identically as in (\ref{OGMF}-\ref{GMF}). We recall that 
\[R(Q) = \frac{1}{N}\Big(D(Q\Vert \Pi_N) + Q[D(P^{(N)}_{\theta_0}\Vert P^{(N)}_{\theta})]\Big) = \frac{1}{N}D(Q\Vert \Pi_N)+ Q[D(P_{\theta_0}\Vert P_{\theta})].\]
It is necessary to bound $\frac{1}{N}D(Q\Vert \Pi_N)$ and $Q[D(P_{\theta_0}\Vert P_{\theta})]$ respectively. 
We define $\Tilde{Q}_N$ to be 
\[\mathop{\bigotimes}_{l=-1}^J\mathop{\bigotimes}_{r\in R_l}N(\theta_{0,lr},\sigma^2), \quad \sigma = 2^{-J(\alpha+{\kappa}+d/2)},\quad \theta_{0,lr} = \pdt{\theta_0}{\psi_{lr}}_{L^2(\mathcal{Z})}.\]
The probability measure $Q_N$ is defined as the push-forward of $\Tilde{Q}_N$ via $\Psi_J$, that is, \[Q_N = \Tilde{Q}_N \circ \Psi_J^{-1}.\]
It is easy to see $Q_N\in {\mathcal{Q}}_G^J$. We also see that the prior $\Pi_N$ can be represented as 
\[ \Pi_N = \Tilde{\Pi}_N \circ \Psi_J^{-1},\]
where $$\Tilde{\Pi}_N = \mathop{\bigotimes}_{l=-1}^J\mathop{\bigotimes}_{r\in R_l}N(0,(2^{l\alpha}\sqrt{N}\varepsilon_N)^{-2}).$$
We first consider the upper bound of $\frac{1}{N}D(Q_N\Vert \Pi_N)$. Because the KL divergence decreases under push-forward \cite[section 10]{varadhan1984large}, we have 
\begin{equation*}
    \begin{aligned}
        D(Q_N\Vert \Pi_N) = &D(\Tilde{Q}_N \circ \Psi_J^{-1}\Vert \Tilde{\Pi}_N \circ \Psi_J^{-1}) \\
        \leq & D(\Tilde{Q}_N\Vert \Tilde{\Pi}_N) \\
        = &\sum_{l=-1}^J\sum_{r\in R_l} D(N(\theta_{0,lr},\sigma^2)\Vert N(0,(2^{l\alpha}\sqrt{N}\varepsilon_N)^{-2})).
    \end{aligned}
\end{equation*}
Then, it is sufficient to consider the upper bound of $D(N(\theta_{0,lr},\sigma^2)\Vert N(0,(2^{l\alpha}\sqrt{N}\varepsilon_N)^{-2}))$:
\begin{equation*}
    \begin{aligned}
        &D(N(\theta_{0,lr},\sigma^2)\Vert N(0,(2^{l\alpha}\sqrt{N}\varepsilon_N)^{-2}))\\
        =&\int\frac{1}{2}(\log \frac{1}{2\pi\sigma^2}-\frac{(x-\theta_{0,lr})^2}{\sigma^2})\frac{1}{\sqrt{2\pi}\sigma}\mexp{-\frac{(x-\theta_{0,lr})^2}{2\sigma^2}}dx\\
        &-\int\frac{1}{2}(\log \frac{2^{2l\alpha}N\varepsilon_N^2}{2\pi}-2^{2l\alpha}N\varepsilon_N^2x^2)\frac{1}{\sqrt{2\pi}\sigma}\mexp{-\frac{(x-\theta_{0,lr})^2}{2\sigma^2}}dx\\
        =&\frac{1}{2}(\log \frac{1}{2\pi\sigma^2} - \log e - \log \frac{2^{2l\alpha}N\varepsilon_N^2}{2\pi}) + \frac{1}{2}2^{2l\alpha}N\varepsilon_N^2(\theta_{0,lr}^2 + \sigma^2)\\
        \leq&\frac{1}{2}\log\frac{2^{2J(\alpha+{\kappa}+d/2)}}{e2^{2l\alpha}N\varepsilon_N^2} + \frac{1}{2}2^{2l\alpha}N\varepsilon_N^2\theta_{0,lr}^2 + \frac{1}{2}2^{-2J({\kappa}+d/2)}N\varepsilon_N^2.
    \end{aligned}
\end{equation*}
By $2^{Jd} \simeq N\varepsilon_N^2$ and $d_J \leq c_0 2^{Jd}$, we deduce that
\begin{equation*}
    \sum_{l=-1}^J\sum_{r\in R_l}\frac{1}{2}\log\frac{2^{2J(\alpha+{\kappa}+d/2)}}{e2^{2l\alpha}N\varepsilon_N^2} \lesssim 2^{Jd}\log 2^{2J(\alpha+{\kappa})}\lesssim N\varepsilon_N^2 \log N,
\end{equation*}
and
\begin{equation*}
    \sum_{l=-1}^J\sum_{r\in R_l}\frac{1}{2}2^{-2J({\kappa}+d/2)}N\varepsilon_N^2\lesssim 2^{-2J{\kappa}}N\varepsilon_N^2\lesssim N\varepsilon_N^2.
\end{equation*}
For the last term, we have
\begin{equation*}
    \sum_{l=-1}^J\sum_{r\in R_l}\frac{1}{2}2^{2l\alpha}N\varepsilon_N^2\theta_{0,lr}^2 \lesssim \norm{\theta_0}_{H^{\alpha}(\mathcal{Z})}^2N\varepsilon_N^2\lesssim N\varepsilon_N^2.
\end{equation*}
In conclusion, 
\begin{equation*}
    \frac{1}{N}D(Q_N\Vert \Pi_N) \leq \frac{1}{N}\sum_{l=-1}^J\sum_{r\in R_l} D(N(\theta_{0,lr},\sigma^2)\Vert N(0,(2^{l\alpha}\sqrt{N}\varepsilon_N)^{-2}))\lesssim \varepsilon_N^2\log N.
\end{equation*}

Next, we give an upper bound of $Q_N[D(P_{\theta_0}\Vert P_{\theta})]$. We assume independent random variables $\theta_{lr} = \theta_{0,lr} + \sigma Z_{lr}$ where $Z_{lr} \sim N(0,1)$ for $r \in R_l, l \in \{-1,0,\dots,J\}$. Using Proposition \ref{le2.1} and the condition (\ref{lip}), we have 
\begin{equation*}
    \begin{aligned}
        Q_ND(P_{\theta_0}\Vert P_{\theta}) = &\frac{1}{2}Q_N\norm{\mathcal{G}(\theta)-\mathcal{G}(\theta_0)}^2_{L^2_{\lambda}(\mathcal{X},V)} \\
        \lesssim & Q_N[(1+\norm{\theta}_{H^{\alpha}(\mathcal{Z})}^{2l})\norm{\theta-\theta_0}^2_{(H^{{\kappa}}(\mathcal{Z}))^*}]
        \\
        \lesssim & Q_N [e^{\norm{\theta}^2_{H^{\alpha}}}\norm{\theta-\theta_0}^2_{(H^{{\kappa}}(\mathcal{Z}))^*}].
    \end{aligned}     
\end{equation*}
Because $Q_N$ is defined by the law of
\[ \sum_{l=-1}^J\sum_{r\in R_l}\theta_{lr}\chi\psi_{lr}, \quad \theta_{lr}\sim N(\theta_{0,lr},\sigma^2),\]
we further have
\begin{equation}\label{thm3.8_1}
    \begin{aligned}
        Q_ND(P_{\theta_0}\Vert P_{\theta})\lesssim &E \mexp{\norm{\sum_{l=-1}^J\sum_{r\in R_l}\theta_{lr}\chi\psi_{lr}}^2_{H^{\alpha}(\mathcal{Z})}} \\ & \cdot\norm{\sum_{l=-1}^{\infty}\sum_{r\in R_l}(\theta_{lr}-\theta_{0,lr})\chi\psi_{lr}}^2_{(H^{{\kappa}}(\mathcal{Z}))^*},
    \end{aligned}   
\end{equation}
where $\theta_{lr} = 0$ for all $l>J$.
Using \eqref{Sobolevinter2} of the main text and wavelet characterization of Sobolev norms (see \cite[section 4]{gin2015mathematical}), we deduce that
\begin{equation*}
    \begin{aligned}
        \norm{\sum_{l=-1}^J\sum_{r\in R_l}\theta_{lr}\chi\psi_{lr}}_{H^{\alpha}(\mathcal{Z})} =& \norm{\chi\sum_{l=-1}^J\sum_{r\in R_l}\theta_{lr}\psi_{lr}}_{H^{\alpha}(\mathbb{R}^d)}\\
        \lesssim & \norm{\sum_{l=-1}^J\sum_{r\in R_l}\theta_{lr}\psi_{lr}}_{H^{\alpha}(\mathbb{R}^d)} =  \sqrt{\sum_{l=-1}^J\sum_{r\in R_l}2^{2l\alpha}\theta_{lr}^2},
    \end{aligned}
\end{equation*}
and
     \begin{align*}
     \norm{\sum_{l=-1}^{\infty}\sum_{r\in R_l}(\theta_{lr}-\theta_{0,lr})\chi\psi_{lr}}_{(H^{{\kappa}}(\mathcal{Z}))^*} = &\sup_{\norm{\phi}_{H^{\kappa}(\mathcal{Z})}\leq 1}\pdt{\sum_{l=-1}^{\infty}\sum_{r\in R_l}(\theta_{lr}-\theta_{0,lr})\psi_{lr}}{\chi\phi}_{L^2(\mathcal{Z})} \\
     = &\sup_{\norm{\phi}_{H^{\kappa}(\mathcal{Z})}\leq 1}\pdt{\sum_{l=-1}^{\infty}\sum_{r\in R_l}(\theta_{lr}-\theta_{0,lr})\psi_{lr}}{\chi\phi}_{L^2(\mathbb{R}^d)} \\
     \lesssim &\sup_{\norm{\phi}_{H_c^{\kappa}(\mathcal{Z})}\leq 1}\pdt{\sum_{l=-1}^{\infty}\sum_{r\in R_l}(\theta_{lr}-\theta_{0,lr})\psi_{lr}}{\phi}_{L^2(\mathbb{R}^d)}\\
     \lesssim & \norm{\sum_{l=-1}^{\infty}\sum_{r\in R_l}(\theta_{lr}-\theta_{0,lr})\psi_{lr}}_{H^{{-\kappa}}(\mathbb{R}^d)} \\
     \lesssim & \sqrt{\sum_{l=-1}^{\infty}\sum_{r\in R_l}2^{-2l{\kappa}}(\theta_{lr}-\theta_{0,lr})^2}.
     \end{align*}
     Applying the last two inequalities to \eqref{thm3.8_1}, we obtain
    \begin{align*}
        Q_ND(P_{\theta_0}\Vert P_{\theta})\lesssim& E\mexp{\sum_{l=-1}^J\sum_{r\in R_l}2^{2l\alpha}\theta_{lr}^2}\sum_{l=-1}^J\sum_{r\in R_l}2^{-2l{\kappa}}(\theta_{lr}-\theta_{0,lr})^2 \\&+E\mexp{\sum_{l=-1}^J\sum_{r\in R_l}2^{2l\alpha}\theta_{lr}^2}\sum_{l=J+1}^{\infty}\sum_{r\in R_l}2^{-2l{\kappa}}\theta_{0,lr}^2\\
        \lesssim& E\mexp{\sum_{l=-1}^J\sum_{r\in R_l}2^{2l\alpha}\theta_{lr}^2}\sum_{l=-1}^J\sum_{r\in R_l}\sigma^2 Z_{lr}^2
        \\&+E\mexp{\sum_{l=-1}^J\sum_{r\in R_l}2^{2l\alpha}\theta_{lr}^2}2^{-2J({\kappa}+\alpha)}\norm{\theta_0}_{H^{\alpha}}^2\\
        \lesssim& e^{2\norm{\theta_0}_{H^{\alpha}}^2}\cdot E\mexp{2\sum_{l=-1}^J\sum_{r\in R_l}2^{2l\alpha}\sigma^2 Z_{lr}^2}\sum_{l=-1}^J\sum_{r\in R_l}\sigma^2 Z_{lr}^2\\ &+ e^{2\norm{\theta_0}_{H^{\alpha}}^2}\cdot E\mexp{2\sum_{l=-1}^J\sum_{r\in R_l}2^{2l\alpha}\sigma^2 Z_{lr}^2}2^{-2J({\kappa}+\alpha)}\norm{\theta_0}_{H^{\alpha}}^2.
    \end{align*}
Provided that
\begin{equation}\label{Eb1}
    E\mexp{2\sum_{l=-1}^J\sum_{r\in R_l}2^{2l\alpha}\sigma^2 Z_{lr}^2} = O(1),
\end{equation}
and
\begin{equation}\label{Eb2}
    EZ_{\Tilde{l}\Tilde{r}}^2\mexp{2\sum_{l=-1}^J\sum_{r\in R_l}2^{2l\alpha}\sigma^2 Z_{lr}^2} = O(1),\quad \mbox{for} \ \Tilde{r}\in R_{\Tilde{l}},\ \Tilde{l} \in \{-1,0,\dots,J\},
\end{equation}
then, since $2^{-J({\kappa}+\alpha)} \simeq \varepsilon_N$ and $d_J \leq c_0 2^{Jd}$, we have
\begin{equation*}
    \begin{aligned}
        Q_ND(P_{\theta_0}\Vert P_{\theta})&\lesssim E\mexp{2\sum_{l=-1}^J\sum_{r\in R_l}2^{2l\alpha}\sigma^2 Z_{lr}^2}\left(\sum_{l=-1}^J\sum_{r\in R_l}\sigma^2 Z_{lr}^2 + 2^{-2J({\kappa}+\alpha)}\right)\\
        &\lesssim \sigma^2d_J + 2^{-2J({\kappa}+\alpha)}\\
        &\lesssim \varepsilon_N^2.
    \end{aligned}
\end{equation*}

Next, we verify (\ref{Eb1}) and (\ref{Eb2}) to complete the proof. For (\ref{Eb1}),
\begin{equation*}
    \begin{aligned}
        E\mexp{2\sum_{l=-1}^J\sum_{r\in R_l}2^{2l\alpha}\sigma^2 Z_{lr}^2} &\leq E\mexp{2\cdot2^{-2J({\kappa} + d/2)}\sum_{l=-1}^J\sum_{r\in R_l} Z_{lr}^2}\\
        & = E\mexp{2\cdot2^{-2J({\kappa} + d/2)}\chi_{d_J}^2},
    \end{aligned}
\end{equation*}
where $\chi_{d_J}^2$ denotes the random variable distributed according to the chi-squared distribution with $d_J$ be the degrees of freedom.
Further, using the moment-generating function of the chi-square distribution, we deduce that
\begin{equation*}
    \begin{aligned}
       E\mexp{2\sum_{l=-1}^J\sum_{r\in R_l}2^{2l\alpha}\sigma^2 Z_{lr}^2} &\leq E\mexp{2^{-2J({\kappa} + d/2)}\chi_{d_J}^2}\\
       & = \left(1- 4\cdot2^{-2J({\kappa} + d/2)}\right)^{-\frac{d_J}{2}}\\
       & \leq \left(1- 4\cdot2^{-Jd}\right)^{-\frac{c_0}{2}\cdot 2^{Jd}} = O(1),
    \end{aligned}
\end{equation*}
which implies (\ref{Eb1}). For (\ref{Eb2}),  
we split this expectation into to two parts:
\begin{equation*}
    EZ_{\Tilde{l}\Tilde{r}}^2\mexp{2\cdot 2^{2\Tilde{l}\alpha}\sigma^2 Z_{\Tilde{l}\Tilde{r}}^2}\cdot E\mexp{2\sum_{(l,r)\neq (\Tilde{l},\Tilde{r})}2^{2l\alpha}\sigma^2 Z_{lr}^2}.
\end{equation*}
The second part can be bounded by a constant, as indicated by equation (\ref{Eb1}). For the first part, we have
\begin{equation*}
    \begin{aligned}
        EZ_{\Tilde{l}\Tilde{r}}^2\mexp{2\cdot 2^{2\Tilde{l}\alpha}\sigma^2 Z_{\Tilde{l}\Tilde{r}}^2} &\leq  EZ_{\Tilde{l}\Tilde{r}}^2\mexp{2\cdot 2^{-2J({\kappa}+d/2)} Z_{\Tilde{l}\Tilde{r}}^2}\\
        &= \int \frac{1}{\sqrt{2\pi}}x^2 e^{2\cdot 2^{-2J({\kappa}+d/2)}x^2} e^{-x^2/2} dx\\
        &=\int \frac{1}{\sqrt{2\pi}}x^2e^{-\frac{1}{2}(1-4\cdot 2^{-2J({\kappa}+d/2)})x^2}dx\\
        &=(1-4\cdot 2^{-2J({\kappa}+d/2)})^{-\frac{3}{2}} = O(1).
    \end{aligned}  
\end{equation*}
Therefore, formula (\ref{Eb2}) is verified.
\end{proof}

\subsection{Contraction rates for severely ill-posed problem}
\begin{lemma}\label{embeddinglemma}
For an integer $\alpha \geq 0$ , we have $\mathcal{H}_J^{\mathcal{Z}}$ continuously embedding in $C^{\alpha}(\mathcal{Z})$, that is
\[\norm{f}_{C^{\alpha}} \lesssim \norm{f}_{\mathcal{H}_J^{\mathcal{Z}}}, \quad \forall f \in \mathcal{H}_J^{\mathcal{Z}}.\]
\end{lemma}
\begin{proof}
For any $f \in \mathcal{H}_J^{\mathcal{Z}}$ and $|i|\leq \alpha$, we deduce that 
\begin{equation*}
\begin{aligned}
    &|D^if| = \abs{\sum_{l=-1}^J\sum_{r=1}^{N_l}2^{-l(\alpha+d/2)}\bar{l}^{-2}f_{lr}D^i\psi_{lr}^{\mathcal{Z}}} \\
    \leq &c\norm{f}_{\mathcal{H}_J^{\mathcal{Z}}}\sum_{l=-1}^J\bar{l}^{-2}2^{-l(\alpha - |i|)}\sum_{r}|D^i\psi_{0r}|\lesssim \norm{f}_{\mathcal{H}_J^{\mathcal{Z}}}.
\end{aligned}
\end{equation*}
Thus,  \[\norm{f}_{C^{\Tilde{\alpha}}} \lesssim \norm{f}_{\mathcal{H}_J^{\mathcal{Z}}},\]
which implies the embedding.
\end{proof}
\begin{proof}[Proof of Theorem \ref{finalthmsvZ}]
We first define the forward map \( \mathcal{G}^* \) as
\[
\mathcal{G}^* : L^2(\mathcal{Z}) \rightarrow L^2_{\lambda}(\mathcal{X}, V), \quad \theta \mapsto G(\Phi \circ \Psi(\theta)).
\]
Denote by \( \Pi_N(\cdot|D_N) \) the posterior from equation (\ref{Post}) arising from data in the model (\ref{model}) with the forward map \( \mathcal{G}^* \) and prior \( \Pi_N \) as in (\ref{rescaledprior}), where \( \theta' \sim \Pi'_J \) and \( \Pi'_J \) is as defined in (\ref{priorsvZ}). The RKHS of \( \Pi'_J \) is continuously embedded in \( C^{\alpha}(\mathcal{Z}) \).
From the settings in this theorem, we observe that the posterior \( \widetilde{\Pi}_N(\cdot|D_N) \) is the push-forward of \( \Pi_N(\cdot|D_N) \) via \( \Psi \), that is,
\[
\widetilde{\Pi}_N(\cdot|D_N) = \Pi_N(\cdot|D_N) \circ \Psi^{-1}.
\]
Similarly, the prior $\Tilde{\Pi}_N$ can also be presented as
\[\Tilde{\Pi}_N = {\Pi}_N \circ \Psi^{-1}.\]
Consider $\theta_0^*$ such that $\theta_0 = \Psi(\theta_0^*)$, which is defined by
\[\theta_0^* = \sum_{l=-1}^{\infty}\sum_{r=1}^{N_l}2^{-l(\alpha +d/2)}\bar{l}^{-2}\theta_{0,lr}^*\psi_{lr}^{\mathcal{Z}}\]
for $\theta_{0,lr}^* = h^{-1}(2^{l(\alpha+d/2)}\bar{l}^{\, 2}\pdt{\theta_0}{\psi_{lr}^{\mathcal{Z}}}_{L^2(\mathcal{Z})})$.
We note that 
\[ 
\begin{aligned}
    2^{l(\alpha+d/2)}\bar{l}^{\, 2}\abs{\pdt{\theta_0}{\psi_{lr}^{\mathcal{Z}}}_{L^2(\mathcal{Z})}} \leq &2^{-l(\beta-\alpha)}\bar{l}^{\, 2}\sup_{l,r} 2^{l(\beta+d/2)}\abs{\pdt{\theta_0}{\psi_{lr}^{\mathcal{Z}}}_{L^2(\mathcal{Z})}} \\
    \leq &2^{-l(\beta-\alpha)}\bar{l}^{\, 2}B_0 \leq \bar{B}_0
\end{aligned}
\]
for some $\bar{B}_0<B$.
By the properties of $\Psi$, we know $(h^{-1})'$ is bounded on $[-\bar{B}_0,\bar{B}_0]$.
Thus, we have 
\[|\theta_{0,lr}^*| = \abs{h^{-1}(2^{l(\alpha+d/2)}\bar{l}^{\, 2}\pdt{\theta_0}{\psi_{lr}^{\mathcal{Z}}}_{L^2(\mathcal{Z})})}\lesssim2^{-l(\beta-\alpha)}\bar{l}^{\, 2}B_0.\]
For any $\abs{i} \leq \alpha$, we further have
\[\abs{\sum_{l=-1}^{\infty}\sum_{r=1}^{N_l}2^{-l(\alpha +d/2)}\bar{l}^{-2}\theta_{0,lr}^*D^i\psi_{lr}^{\mathcal{Z}}}
    \leq cB_0\sum_{l=-1}^{\infty}2^{-l(\beta - \abs{i})}\sum_{r}|D^i\psi_{0r}|\leq C(B_0),\]
which implies $\theta_0^*\in C^{\alpha}(\mathcal{Z})$.
In conclusion, the setting of this theorem is within the framework of Theorem \ref{mainthm}, which includes a rescaled Gaussian prior \( \Pi_N \), a forward map \( \mathcal{G}^* \), and a ground truth \( \theta_0^* \in C^{\alpha}(\mathcal{Z}) \).

The proof of this theorem is similar to that of Theorem \ref{finalthmsv}. We first verify Conditions \ref{condreg} and \ref{condstab} for the forward map $\mathcal{G}^{*}$. By Condition \ref{loosecondreg}, we have
\begin{equation*}
       \mathop{\sup}_{\Psi(\theta) \in \Theta \cap B_{\mathcal{R}}(M)}\mathop{\sup}_{x \in \mathcal{X}}\abs{\mathcal{G}(\Psi(\theta))(x)}_V\leq U_{\mathcal{G}}(M),
\end{equation*}
\begin{equation*}
       \norm{\mathcal{G}(\Psi(\theta_1))-\mathcal{G}(\Psi(\theta_2))}_{L^{2}_{\lambda}(\mathcal{X},V)} \leq L_{\mathcal{G}}(M)\norm{\Psi(\theta_1)-\Psi(\theta_2)} _{L^{2}(\mathcal{Z})}, \quad \theta_1,\theta_2 \in L^2(\mathcal{Z}),
\end{equation*}
for any $M>1$. By the properties of $\Psi$, we deduce that
\begin{equation*}
    \begin{aligned}
        \norm{\Psi(\theta_1)-\Psi(\theta_2)} _{L^{2}(\mathcal{Z})}^2
        =&\sum_{l=-1}^{\infty}\sum_{r=1}^{N_l}2^{-2l(\alpha+d/2)}\bar{l}^{-4}(h(\theta_{1,lr})-h(\theta_{2,lr}))^2\\
        \lesssim & \sum_{l=-1}^{\infty}\sum_{r=1}^{N_l}2^{-2l(\alpha+d/2)}\bar{l}^{-4}(\theta_{1,lr}-\theta_{2,lr})^2\\
        \lesssim & \norm{\theta_1-\theta_2} _{L^{2}(\mathcal{Z})}^2.
    \end{aligned}
\end{equation*}
Because $\Psi(\theta)$ belongs to a ball in $C^{\gamma}$ with radius $r(B)$ for any $\theta \in L^2(\mathcal{Z})$, we have
\begin{equation*}
       \mathop{\sup}_{\theta \in L^2(\mathcal{Z})}\mathop{\sup}_{x \in \mathcal{X}}\abs{\mathcal{G}^*(\theta)(x)}_V\leq U^*_{\mathcal{G}},
\end{equation*}
\begin{equation}\label{thm5.3_lip}
       \norm{\mathcal{G}^*(\theta_1)-\mathcal{G}^*(\theta_2)}_{L^{2}_{\lambda}(\mathcal{X},V)} \leq L^*_{\mathcal{G}}\norm{\theta_1-\theta_2} _{L^{2}(\mathcal{Z})}, \quad \theta_1,\theta_2 \in L^2(\mathcal{Z}),
\end{equation}
for $U^*_{\mathcal{G}}, L^*_{\mathcal{G}}$ are finite constants depend on $B$. Thus we have verified Condition \ref{condreg} for $\mathcal{G}^{*}$. 
For Condition \ref{condstab}, using Condition \ref{loosecondstab} for $\mathcal{G}$, we have
\begin{equation*}
    F(\Vert f_{\Psi(\theta)} - f_{\Psi(\theta_0^*)}\Vert ) \leq T_{\mathcal{G}}\Vert \mathcal{G}(\Psi(\theta))-\mathcal{G}(\Psi(\theta_0^*))\Vert _{L^{2}_{\lambda}(\mathcal{X},V)}, \forall \Psi(\theta) \in B_{\mathcal{R}}(M).
\end{equation*}
Similar to the proof for Condition \ref{condreg}, we have
\begin{equation*}
    F(\Vert f_{\Psi(\theta)} - f_{\Psi(\theta_0^*)}\Vert ) \leq T^*_{\mathcal{G}}\Vert \mathcal{G}^*(\theta)-\mathcal{G}(\theta_0^*)\Vert _{L^{2}_{\lambda}(\mathcal{X},V)}, \forall \theta \in L^2(\mathcal{Z}),
\end{equation*}
where the function \( F \) and the finite constant \( T^*_{\mathcal{G}} \) may both depend on \( B \).
\par
Then by the proof of Theorem \ref{mainthmsv}, for any $\varepsilon > \varepsilon_N$, $\Theta_N(\varepsilon)$, and the test $\Psi_N$ defined in the proof Theorem \ref{mainthm}, we have
\begin{equation*}
    P^{(N)}_{\Psi(\theta_0^*)}+ \sup_{\substack
                 {\theta^* \in H_J^{\mathcal{Z}}\cap\Theta_N(\varepsilon)  \\L(P^{(N)}_{\Psi(\theta^*)},P^{(N)}_{\Psi(\theta_0^*)}) \geq C_1 N \varepsilon^2}} P^{(N)}_{\Psi(\theta)}(1-\Psi_N) \leq \exp\{-CN\varepsilon^2\} ,
\end{equation*}
\begin{equation*}
\Pi_N\bbra{(H_J^{\mathcal{Z}}\cap\Theta_N(\varepsilon))^c}=\Pi_N\bbra{(\Theta_N(\varepsilon))^c}\leq\exp(-CN\varepsilon^2),
\end{equation*}
and
\begin{equation*}
    \Pi_N\left(\theta^* : D_2(P^{(N)}_{\Psi(\theta_0^*)}\Vert P^{(N)}_{\Psi(\theta^*)})\leq C_3N\varepsilon_N^2\right)\geq\exp(-C_2N\varepsilon_N^2),
\end{equation*}
where $\frac{1}{N}L(P^{(N)}_{\Psi(\theta^*)},P^{(N)}_{\Psi(\theta_0^*)}) := F(\Vert f_{\Psi(\theta^*)} - f_{\Psi(\theta_0^*)}\Vert )^{2}$ or $\Vert \mathcal{G}(\Psi(\theta))-\mathcal{G}(\Psi(\theta_0^*))\Vert _{L^{2}_{\lambda}(\mathcal{X},V)}^{2}$.
Therefore, with $\Tilde{\Theta}_N:= \set{\theta = \Psi(\theta^*): \theta^* \in H_J^{\mathcal{Z}}\cap\Theta_N(\varepsilon)}$ and $\Tilde{\Pi}_N = {\Pi}_N \circ \Psi^{-1}$, we can transform the last three inequalities to $\Tilde{\Pi}_N$ as follows:
\begin{equation*}
    P^{(N)}_{\theta_0}+ \mathop{\sup}_{\begin{array}{c}
                 \theta^* \in H_J^{\mathcal{Z}}\cap\Theta_N(\varepsilon)  \\
                 L(P^{(N)}_{\theta},P^{(N)}_{\theta_0}) \geq C_1 N \varepsilon^2
            \end{array}}P^{(N)}_{\theta}(1-\Psi_N) \leq \exp\{-CN\varepsilon^2\} ,
\end{equation*}
\begin{equation*}
\Tilde{\Pi}_N\bbra{\Tilde{\Theta}_N^c}=\Pi_N\bbra{(H_J^{\mathcal{Z}}\cap\Theta_N(\varepsilon))^c}\leq\exp(-CN\varepsilon^2),
\end{equation*}
and
\begin{equation*}
    \Tilde{\Pi}_N\left(\theta : D_2(P^{(N)}_{\theta_0}\Vert P^{(N)}_{\theta})\leq C_3N\varepsilon_N^2\right)\geq\exp(-C_2N\varepsilon_N^2).
\end{equation*}
Then, similar as in the proof of Theorem \ref{mainthm}, we apply Theorem 2.1 in \cite{zhang2020convergence} to obtain
\begin{equation}
      P_{\theta_0}^{(N)}\hat{Q}\Vert \mathcal{G}(\theta)-\mathcal{G}(\theta_0)\Vert _{L^{2}_{\lambda}}^{2}\lesssim \varepsilon_N^{2} + \gamma^2_N,
  \end{equation}
and
  \begin{equation}
      P_{\theta_0}^{(N)}\hat{Q}[F(\Vert f_{\theta} - f_{\theta_0}\Vert )]^{2}\lesssim \varepsilon_N^{2} +\gamma^2_N,
  \end{equation}
  where $\gamma_N^2 = \frac{1}{N}\inf_{\Tilde{Q}\in\mathcal{Q}_{MF}^J}P_{\theta_0}^{(N)}D(\Tilde{Q}\Vert \Tilde{\Pi}_N(\cdot|D_N))$.
  \par
  To prove this theorem, the remaining task is to demonstrate that \[\gamma_N^2 \lesssim \varepsilon_N^2\log N.\]
We recall that 
\[R(\Tilde{Q}) = \frac{1}{N}(D(\Tilde{Q}\Vert \Tilde{\Pi}_N) + \Tilde{Q}[D(P^{(N)}_{\theta_0^*}\Vert P^{(N)}_{\theta})]) = \frac{1}{N}D(\Tilde{Q}\Vert \Tilde{\Pi}_N)+ \Tilde{Q}[D(P_{\theta_0}\Vert P_{\theta})],\]
and\[\gamma_N^2 \leq \mathop{\inf}_{\Tilde{Q} \in {\mathcal{\Tilde{Q}}}_{MF}^J}R(\Tilde{Q}).\]
We see $Q\circ \Psi^{-1}\in \mathcal{Q}_{MF}^J$ for any $Q \in \mathcal{Q}_G^J$. Thus, it is sufficient to find a $Q_N \in \mathcal{Q}_G^J$ such that
\[ \frac{1}{N}D(Q_N\circ \Psi^{-1}\Vert \Tilde{\Pi}_N)+ Q_N\circ \Psi^{-1}[D(P_{\theta_0}\Vert P_{\theta})] \lesssim \varepsilon_N^2\log N.\]
Since the KL divergence decreases when applying a push-forward, as shown in \cite[section 10]{varadhan1984large}, we can conclude that
\begin{equation*}
    D(Q_N\circ \Psi^{-1}\Vert \Tilde{\Pi}_N) =  D(Q_N\circ \Psi^{-1}\Vert \Pi_N\circ \Psi^{-1}) \leq D(Q_N\Vert \Pi_N).
\end{equation*}
We also have
\[Q_N\circ \Psi^{-1}[D(P_{\theta_0}\Vert P_{\theta})] = Q_N[D(P_{\Psi(\theta^*_0)}\Vert P_{\Psi(\theta)})].\]
As a result, it is essential to establish bounds for \(\frac{1}{N}D(Q_N \parallel \Pi_N)\) and \(Q_N D(P_{\Psi(\theta^*_0)} \parallel P_{\Psi(\theta)})\) individually.
We define \( Q'_N \) as follows:
\[\mathop{\bigotimes}_{l=-1}^J\mathop{\bigotimes}_{r=1}^{N_l}N(2^{-l(\alpha +d/2)}\bar{l}^{-2}\theta_{0,lr}^*,\sigma^2), \quad \sigma = 2^{-J(\alpha+d)}J^{-2}.\]
The mapping \( \Psi_J: \mathbb{R}^{d_J} \to L^2(\mathcal{Z}) \) is given by
\begin{equation*}
    \Psi_J(\Tilde{\theta}) = \sum_{l=-1}^J\sum_{r=1}^{N_l}\Tilde{\theta}_{lr}\psi_{lr}^{\mathcal{Z}},\quad \forall \Tilde{\theta} = (\Tilde{\theta}_{lr})\in \mathbb{R}^{d_J}.
\end{equation*}
where $d_J := \sum_{l=-1}^{J}N_l\simeq 2^{Jd}$.
The probability measure \( Q_N \) is defined as the push-forward of \( Q'_N \) under the mapping \( \Psi_J \), which can be expressed as \[Q_N = Q'_N \circ \Psi_J^{-1}.\]
It is straightforward to observe that \( Q_N \) belongs to \( \mathcal{Q}_G^J \). Additionally, we can express the prior \( \Pi_N \) as
\[ \Pi_N = \Pi'_N \circ \Psi_J^{-1},\]
where $\Pi'_N = \mathop{\bigotimes}_{l=-1}^J\mathop{\bigotimes}_{r=1}^{N_l}N(0,(2^{l(\alpha +d/2)}\bar{l}^{\,2}\sqrt{N}\varepsilon_N)^{-2})$.
We first consider the upper bound of $\frac{1}{N}D(Q_N\Vert \Pi_N)$. Given that the KL divergence decreases under push-forwards, as demonstrated in \cite[section 10]{varadhan1984large}, we have
\begin{equation*}
    \begin{aligned}
        D(Q_N\Vert \Pi_N) = &D(Q'_N \circ \Psi_J^{-1}\Vert \Pi'_N \circ \Psi_J^{-1}) \\
        \leq & D(Q'_N\Vert \Pi'_N) \\
        = &\sum_{l=-1}^J\sum_{r=1}^{N_l} D(N(\Tilde{\theta}^*_{0,lr},\sigma^2)\Vert N(0,(2^{l(\alpha +d/2)}\bar{l}^{\,2}\sqrt{N}\varepsilon_N)^{-2})),
    \end{aligned}
\end{equation*}
where $\Tilde{\theta}^*_{0,lr}$ represents $2^{-l(\alpha +d/2)}\bar{l}^{-2}\theta_{0,lr}^*$.
Then, it is sufficient to consider the upper bound of $D(N(\Tilde{\theta}^*_{0,lr},\sigma^2)\Vert N(0,(2^{l(\alpha +d/2)}\bar{l}^{\,2}\sqrt{N}\varepsilon_N)^{-2}))$:
\begin{align*}
        &D(N(\Tilde{\theta}^*_{0,lr},\sigma^2)\Vert N(0,(2^{l(\alpha +d/2)}\bar{l}^{\,2}\sqrt{N}\varepsilon_N)^{-2}))\\
        =&\int\frac{1}{2}(\log \frac{1}{2\pi\sigma^2}-\frac{(x-\Tilde{\theta}^*_{0,lr})^2}{\sigma^2})\frac{1}{\sqrt{2\pi}\sigma}\mexp{-\frac{(x-\Tilde{\theta}^*_{0,lr})^2}{2\sigma^2}}dx\\
        &-\int\frac{1}{2}(\log \frac{2^{2l(\alpha +d/2)}\bar{l}^{\,4}N\varepsilon_N^2}{2\pi}-2^{2l(\alpha +d/2)}\bar{l}^{\,4}N\varepsilon_N^2x^2)\frac{1}{\sqrt{2\pi}\sigma}\mexp{-\frac{(x-\Tilde{\theta}^*_{0,lr})^2}{2\sigma^2}}dx\\
        =&\frac{1}{2}(\log \frac{1}{2\pi\sigma^2} - \log e - \log \frac{2^{2l(\alpha +d/2)}\bar{l}^{\,4}N\varepsilon_N^2}{2\pi}) + \frac{1}{2}2^{2l(\alpha +d/2)}\bar{l}^{\,4}N\varepsilon_N^2(\Tilde{\theta}^{*2}_{0,lr}+\sigma^2)\\
        \leq&\frac{1}{2}\log\frac{2^{2J(\alpha+d/2)}J^4}{e2^{2l(\alpha +d/2)}\bar{l}^{\,4}N\varepsilon_N^2}+ \frac{1}{2}{\theta}_{0,lr}^{*2}N\varepsilon_N^2 + \frac{1}{2}2^{-Jd}N\varepsilon_N^2.
\end{align*}
By $2^{Jd}\simeq N\varepsilon_N^2$ and $N_l \leq c_0 2^{ld}$, we deduce that
\begin{equation*}
    \sum_{l=-1}^J\sum_{r=1}^{N_l}\frac{1}{2}\log\frac{2^{2J(\alpha+d/2)}J^4}{e2^{2l(\alpha +d/2)}\bar{l}^{\,4}N\varepsilon_N^2} \lesssim 2^{Jd}\log \bbra{J^42^{2J(\alpha+d/2)}}\lesssim N\varepsilon_N^2\log N,
\end{equation*}
and
\begin{equation*}
    \sum_{l=-1}^J\sum_{r=1}^{N_l}\frac{1}{2}2^{-Jd}N\varepsilon_N^2 \lesssim 2^{-Jd}N\varepsilon_N^2 \sum_{l=-1}^J2^{ld}\lesssim N\varepsilon_N^2.
\end{equation*}
Because $ \alpha < \beta - d/2$, for the last term, we have
\begin{equation*}
    \sum_{l=-1}^J\sum_{r=1}^{N_l}\frac{1}{2}N\varepsilon_N^2\theta_{0,lr}^2\lesssim N\varepsilon_N^2\sum_{l=-1}^J2^{2l(\alpha + d/2-\beta )}\lesssim N\varepsilon_N^2.
\end{equation*}
In conclusion, 
\begin{equation*}
    \frac{1}{N}D(Q_N\Vert \Pi_N)\lesssim \varepsilon_N^2.
\end{equation*}

Next, we give an upper bound of $Q_ND(P_{\Psi(\theta^*_0)}\Vert P_{\Psi(\theta)})$. We assume independent random variables $\theta_{lr} = \Tilde{\theta}_{0,lr}^* + \sigma Z_{lr}$ where $Z_{lr} \sim N(0,1)$ for $r \leq N_l, l \in \{-1,0,\dots,J\}$. Using Proposition \ref{le2.1} and the Lipschitz condition \eqref{thm5.3_lip}, we have 
\begin{equation*}
    \begin{aligned}
        Q_ND(P_{\Psi(\theta^*_0)}\Vert P_{\Psi(\theta)}) =& \frac{1}{2}Q_N\norm{\mathcal{G}^*(\theta)-\mathcal{G}^*(\theta_0^*)}^2_{L^2_{\lambda}(\mathcal{X},V)}\\
        \lesssim& Q_N\norm{\theta-\theta_0}^2_{L^2}\\
        \lesssim& E\sum_{l=-1}^J\sum_{r=1}^{N_l}(\theta_{lr}-\Tilde{\theta}_{0,lr}^*)^2 + \sum_{l=J+1}^{\infty}\sum_{r=1}^{N_l}2^{-2l(\alpha +d/2)}\bar{l}^{-4}\theta_{0,lr}^{*2}\\
        \lesssim& E\sum_{l=-1}^J\sum_{r=1}^{N_l}\sigma^2 Z_{lr}^2 +\sum_{l=J+1}^{\infty}\sum_{r=1}^{N_l}2^{-2l(\alpha +d/2)}\bar{l}^{-4}\theta_{0,lr}^{*2}\\
        \lesssim& E\sum_{l=-1}^J\sum_{r=1}^{N_l}\sigma^2 Z_{lr}^2+\sum_{l=J+1}^{\infty}\sum_{r=1}^{N_l}2^{-2l(\alpha + d/2)}\\
        \lesssim& \sigma^2 2^{Jd}+ 2^{-2J\alpha} \lesssim \varepsilon_N^2.
    \end{aligned}
\end{equation*} 
Therefore, $R(Q_N\circ \Psi^{-1}) \lesssim \varepsilon_N^2\log N$ and $\gamma_N^2 \leq \mathop{\inf}_{Q \in \mathcal{Q}_{MF}^J}R(Q)$ imply that
\[\gamma^2_N \lesssim \varepsilon_N^2\log N,\]
which completes the proof.
\end{proof}
\subsection{Contraction rates for inverse problems with extra unknown parameters}
{\color{black}\begin{example}\label{extracondition}
Consider the following time-fractional subdiffusion equation on $\Omega:=(0,1)$
\begin{align*}
    \left\{\begin{aligned}
    &\partial_t^{\alpha} u + Lu -c(t)u = r(x)p(t), \quad \mbox{on } \Omega \times (0,T],\\
    &u = 0 \quad \mbox{on } \partial \Omega \times (0,T],\\
    &u(0) = u_0 \quad \mbox{on } \Omega,
    \end{aligned}\right.  
\end{align*}
where the goal is to simultaneously recover $c(t)$ and $p(t)$ from measured data $u(x_i,t)$ for $0<t<T$ and fixed measure points $x_1,x_2$. The forward map $G$ is defined by 
\[ G({c,p})(t)= \big(u(x_1,t),u(x_2,t)\big)^T,\]
where $u$ is the solution corresponding to $c,p$.
The differential operator $L$ is defined by
\[Lu = -\frac{\partial}{\partial x}\big(a(x)\frac{\partial u}{\partial x}(x,t)\big)+b(x)u(x,t),
\]
with smooth functions $a(x),b(x)$ satisfying $a(x) \ge \mu>0$ and $b(x)\geq 0$ for some $\mu>0$ and all $x \in \bar{\Omega}$.
Let $u_i$ be the solution with $c_i$ and $p_i$ such that $\norm{c_i}_{C^1(0,T)},\norm{p_i}_{C^1(0,T)}\leq M$ for $i={1,2}$. Suppose there exists $\nu>0$ such that 
\[
    \nu \leq u_1(x_1,t)r(x_2)-u_1(x_2,t)r(x_1), \quad t\in(0,T).
\]
Then the conditional stability result in \cite{simultaneous_ma_simultaneous_2024} directly yields
\[
\norm{c_1-c_2}_{\infty}\leq C\sum_{i=1}^2\norm{u_1(x_i,\cdot)-u_2(x_i,\cdot)}_{C[0,T]}^{1-\alpha},
\]
where $C>0$ depending on $M$. Using the regularity results from \cite{simultaneous_ma_simultaneous_2024}, together with Sobolev embedding and interpolation inequalities, we further assume $\norm{c_i}_{{C}^3(0,T)},\norm{p_i}_{{C}^3(0,T)}\leq M'$, and obtain the following estimate for the conditional stability and verify Condition \ref{condstabmodel} (though without the polynomial growth rate in $\norm{c_i}_{C^3(0,T)}$):
\[
    \norm{c_1-c_2}_{\infty}^{{3}/{(1-\alpha)}}\leq C\sum_{i=1}^2\norm{u_1(x_i,\cdot)-u_2(x_i,\cdot)}_{L^2[0,T]},
\]
where $C>0$ depending on $M'$. Although we do not have the polynomial growth rate, a contraction theory similar to Theorem \ref{mainthmModel} can still be obtained with the prior supported on a ball in $C^3$ as constructed in Section \ref{SectionContractionRateUnstable}.  (\emph{Given that this example is not the central focus of our work, we do not provide an exhaustive proof of polynomial growth. Nevertheless, the growth rate is expected to be polynomial, as the underlying PDE is similar to the subdiffusion equation example we analyze in Sections \ref{ApplicationSection} and \ref{SectionContractionExtra}. Furthermore, verifying that existing results exactly align with our Condition \ref{condstabmodel} is nontrivial, as the growth rate of the stability coefficient is rarely the primary concern in the relevant literature.})
\end{example}}
\begin{proof}[Proof of Theorem \ref{mainthmModel}]
We begin by establishing a lemma that provides an upper bound for \[P_{\beta_0,\theta_0}^{(N)}\hat{Q}L(P^{(N)}_{\hat{\beta},\theta},P^{(N)}_{\beta_0,\theta_0})\] for any loss function $L(\cdot,\cdot)$.
\begin{lemma}
    For \( \hat{\beta} \), \( \hat{Q} \) is the solution to equation (\ref{opplm}), we have
    \begin{equation*}
        \begin{aligned}
            &P_{\beta_0,\theta_0}^{(N)}\hat{Q}L(P^{(N)}_{\hat{\beta},\theta},P^{(N)}_{\beta_0,\theta_0})\\
            \leq&\log\pi(\hat{\beta})+\mathop{\inf}_{a>0}\frac{1}{a}\left\{\mathop{\min}_{\beta \in \mathcal{U}}\mathop{\min}_{Q^ \in S_G^J}\left[D\left(Q\Vert\Pi\right) + QD\left(P^{(N)}_{\beta_0,\theta_0}\Vert P^{(N)}_{\beta,\theta}\right)-\log\pi(\beta)\right]\right.\\
            &\left.+ P^{(N)}_{\beta_0,\theta_0}\log\Pi\left(\exp\left(aL(P^{(N)}_{\hat{\beta},\theta},P^{(N)}_{\beta_0,\theta_0})\right)\vert D^{(N)}, \hat{\beta}\right)\right\}.
        \end{aligned}
    \end{equation*}
\end{lemma}
\begin{proof}
We denote $\int\log p(D^{(N)}\vert\beta,\theta)dQ(\theta) - D(Q\Vert\Pi_N) + \log\pi(\beta)$ by $H(Q,\beta)$. For any $a >0$, any $\beta \in \mathcal{U}$, and any $Q \in S_G^J$, we deduce that 
    \begin{align*}
        &D\bbra{Q\Vert\Pi} + QD\bbra{P^{(N)}_{\beta_0,\theta_0}\Vert P^{(N)}_{\beta,\theta}} - \log\pi(\beta) - aP_{\beta_0,\theta_0}^{(N)}\hat{Q}L(P^{(N)}_{\hat{\beta},\theta},P^{(N)}_{\beta_0,\theta_0})\\
        =&P_{\beta_0,\theta_0}^{(N)}\bbra{-H(Q,\beta)+\log p^{(N)}_{\beta_0,\theta_0}(D^{(N)})} - aP_{\beta_0,\theta_0}^{(N)}\hat{Q}L(P^{(N)}_{\hat{\beta},\theta},P^{(N)}_{\beta_0,\theta_0})\\
        \geq&P_{\beta_0,\theta_0}^{(N)}\bbra{-H(\hat{Q},\hat{\beta})+\log p^{(N)}_{\beta_0,\theta_0}(D^{(N)})} - aP_{\beta_0,\theta_0}^{(N)}\hat{Q}L(P^{(N)}_{\hat{\beta},\theta},P^{(N)}_{\beta_0,\theta_0})\\
        =&P_{\beta_0,\theta_0}^{(N)}\hat{Q}\log\bbra{\frac{d\hat{Q}(\theta)}{d\Pi(\theta)}\cdot\frac{p^{(N)}_{\beta_0,\theta_0}(D^{(N)})}{\pi(\hat{\beta})p^{(N)}_{\hat{\beta},\theta}(D^{(N)})\exp\bbra{aL(P^{(N)}_{\hat{\beta},\theta},P^{(N)}_{\beta_0,\theta_0})}}}\\
        =&D\bbra{P_{\beta_0,\theta_0}^{(N)}\Vert P^{(N)}_{\hat{\beta}}}  \\
        &+P_{\beta_0,\theta_0}^{(N)}\hat{Q}\log\bbra{\frac{d\hat{Q}(\theta)}{d\Pi(\theta)}\cdot\frac{p^{(N)}_{\hat{\beta}}(D^{(N)})}{\pi(\hat{\beta})p^{(N)}_{\hat{\beta},\theta}(D^{(N)})\exp\bbra{aL(P^{(N)}_{\hat{\beta},\theta},P^{(N)}_{\beta_0,\theta_0})}}}\\
        =&D\bbra{P_{\beta_0,\theta_0}^{(N)}\Vert P^{(N)}_{\hat{\beta}}} + P_{\beta_0,\theta_0}^{(N)}D\bbra{Q\Vert\Tilde{\Pi}_{\hat{\beta}}}\\
        &-P_{\beta_0,\theta_0}^{(N)}\log\frac{\int\pi(\hat{\beta})p^{(N)}_{\hat{\beta},\theta}(D^{(N)})\exp\bbra{aL(P^{(N)}_{\hat{\beta},\theta},P^{(N)}_{\beta_0,\theta_0})}d\Pi(\theta)}{p^{(N)}_{\hat{\beta}}(D^{(N)})}\\
        \geq&-P_{\beta_0,\theta_0}^{(N)}\log\frac{\int\pi(\hat{\beta})p^{(N)}_{\hat{\beta},\theta}(D^{(N)})\exp\bbra{aL(P^{(N)}_{\hat{\beta},\theta},P^{(N)}_{\beta_0,\theta_0})}d\Pi(\theta)}{p^{(N)}_{\hat{\beta}}(D^{(N)})}\\
        =&-P_{\beta_0,\theta_0}^{(N)}\log\Pi\bbra{\exp\bbra{aL(P^{(N)}_{\hat{\beta},\theta},P^{(N)}_{\beta_0,\theta_0})}\vert D^{(N)},\hat{\beta}} - \log \pi(\hat{\beta}),
    \end{align*}
where $P^{(N)}_{\hat{\beta}}$ is the probability measure with the density $p^{(N)}_{\hat{\beta}}(D^{(N)})$ and
\begin{equation*}
    d\Tilde{\Pi}_{\beta}(\theta) = \frac{p^{(N)}_{\beta,\theta}(D^{(N)})\exp\bbra{aL(P^{(N)}_{\beta,\theta},P^{(N)}_{\beta_0,\theta_0})}d\Pi(\theta)}{\int p^{(N)}_{\beta,\theta}(D^{(N)})\exp\bbra{aL(P^{(N)}_{\beta,\theta},P^{(N)}_{\beta_0,\theta_0})}d\Pi(\theta)}.
\end{equation*}
Then the inequality is proved.
\end{proof}
For $L(P^{(N)}_{\hat{\beta},\theta},P^{(N)}_{\beta_0,\theta_0}) := N\varepsilon_N^{\frac{2p+2q}{p+q+1}}[F(\Vert f_{\theta} - f_{\theta_0}\Vert )]^{\frac{2}{p+q+1}}$, we bound each term in the upper bound given above. First, the proof of Theorem \ref{boundgamfinite} implies that there exists a probability measure $Q\in Q_G^J$ such that 
\[D(Q\Vert\Pi)\lesssim N\varepsilon_N^2\log N,\] 
and 
\[QD\bbra{P^{(N)}_{\beta_0,\theta_0},P^{(N)}_{\beta_0,\theta}}\lesssim N\varepsilon_N^2.\]
Then, since $N\varepsilon_N^2\rightarrow \infty$, we have 
\[\log\pi(\hat{\beta})- \log\pi(\beta_0) \leq N\varepsilon_N^2\]
for large enough $N$. 
\par
Next, it remains to bound the term
\[P^{(N)}_{\beta_0,\theta_0}\log\Pi\left(\exp\left(a L(P^{(N)}_{\hat{\beta},\theta},P^{(N)}_{\beta_0,\theta_0})\right)\vert D^{(N)}, \hat{\beta}\right).\]
By Jensen Inequality, we consider its upper bound
\begin{equation}\label{upms}
    \log P^{(N)}_{\beta_0,\theta_0}\Pi\left(\exp\left(aL(P^{(N)}_{\hat{\beta},\theta},P^{(N)}_{\beta_0,\theta_0})\right)\vert D^{(N)}, \hat{\beta}\right).
\end{equation}
To bound (\ref{upms}), we follow the technique employed in the proof of \cite[theorem 4.1]{zhang2020convergence}. We define a discrete probabiliy meausre $\nu$ on $\mathcal{U}$ with $\nu(\beta_0) = 0.5, \nu(\hat{\beta}) =0.5$ and the prior on $\mathcal{A} \times \Theta$ with $\hat{\Pi} = \nu \times \Pi_N$. Recalling that the conditions (C1) to (C3) in \cite[theorem 2.1]{zhang2020convergence} is true here, that is for any $\varepsilon > \varepsilon_N$, 
\begin{equation}\label{C1ms}
    P^{(N)}_{\beta_0,\theta_0}+ \mathop{\sup}_{\begin{array}{c}
                 (\beta,\theta) \in \{\beta_0,\hat{\beta}\}\times\Theta_N(\varepsilon)  \\
                 L(P^{(N)}_{\beta,\theta},P^{(N)}_{\beta_0,\theta_0}) \geq C_1 N \varepsilon^2
            \end{array}}P^{(N)}_{\beta,\theta}(1-\Psi_N) \leq \exp\{-CN\varepsilon^2\} ,
\end{equation}
\begin{equation}\label{C2ms}
    \hat{\Pi}\bbra{(\{\beta_0,\hat{\beta}\}\times\Theta_N(\varepsilon))^c}\leq\exp(-CN\varepsilon^2),
\end{equation}
and
\begin{equation}\label{C3ms}
    \hat{\Pi}\left((\beta,\theta) : D_2(P^{(N)}_{\beta_0,\theta_0}\Vert P^{(N)}_{\beta,\theta})\leq C_3N\varepsilon_N^2\right)\geq\exp(-C_2N\varepsilon_N^2).
\end{equation}
Together with Lemma B.3 and Lemma B.4 in \cite{zhang2020convergence}, we have 
\begin{equation*}
    \log P^{(N)}_{\beta_0,\theta_0}\hat{\Pi}\left(\exp\left(aL(P^{(N)}_{\beta,\theta},P^{(N)}_{\beta_0,\theta_0})\right)\vert D^{(N)}\right) \lesssim N\varepsilon_N^2
\end{equation*}
with $a = 1/2C_1$. Furthermore, we deduce that
\begin{equation}
    \begin{aligned}
        &\log P^{(N)}_{\beta_0,\theta_0}\hat{\Pi}\left(\exp\left(aL(P^{(N)}_{\beta,\theta},P^{(N)}_{\beta_0,\theta_0})\right)\vert D^{(N)}\right) \\
        =&\log P^{(N)}_{\beta_0,\theta_0}\sum_{\beta\in\{\beta_0,\hat{\beta}\}}\nu(\beta)\Pi\left(\exp\left(aL(P^{(N)}_{\beta,\theta},P^{(N)}_{\beta_0,\theta_0})\right)\vert \beta, D^{(N)}\right)\\
        \geq&\log P^{(N)}_{\beta_0,\theta_0}\Pi\left(\exp\left(aL(P^{(N)}_{\hat{\beta},\theta},P^{(N)}_{\beta_0,\theta_0})\right)\vert \hat{\beta}, D^{(N)}\right) + \log\nu(\hat{\beta}).
    \end{aligned}
\end{equation}
Thus, 
\begin{equation*}
    \log P^{(N)}_{\beta_0,\theta_0}\Pi\left(\exp\left(aL(P^{(N)}_{\hat{\beta},\theta},P^{(N)}_{\beta_0,\theta_0})\right)\vert \hat{\beta}, D^{(N)}\right) \lesssim N\varepsilon_N^2.
\end{equation*}
\par
In conclusion, we have
\begin{equation*}
    P^{(N)}_{\beta_0,\theta_0}\hat{Q}L(P^{(N)}_{\hat{\beta},\theta},P^{(N)}_{\beta_0,\theta_0}) \lesssim N \varepsilon_N^2 \log N,
\end{equation*}
that is
\begin{equation*}
    P_{\beta_0,\theta_0}^{(N)}\hat{Q}[F(\Vert f_{\theta} - f_{\theta_0}\Vert )]^{\frac{2}{p+q+1}}\lesssim \varepsilon_N^{\frac{2}{p+q+1}}\log N.
\end{equation*}
Note that the above proof is also true for $L(P^{(N)}_{\hat{\beta},\theta},P^{(N)}_{\beta_0,\theta_0}) := N\varepsilon_N^{\frac{2p}{p+1}}\lVert \mathcal{G}_{\hat{\beta}}({\theta}) - \mathcal{G}_{\beta_0}({\theta_0})\rVert^{\frac{2}{p+1}}_{L^2_{\lambda}}$.
\par
Finally, we verify (\ref{C1ms}-\ref{C3ms}) in a manner similar to the proof of Theorem \ref{mainthm}, thereby completing this proof. We define the set \( \Theta_N(\varepsilon) \) to be the same as the one used in the proof of Theorem \ref{mainthm}.
\par
(i) For (\ref{C1ms}), we first estimate the metric entropy of the set $\{\beta_0,\hat{\beta}\}\times \Theta_N(\varepsilon)$. Utilizing inequality (\ref{lipmodel}), we obtain
\begin{equation*}
    2h(p_{\beta_1,\theta_1},p_{\beta_2,\theta_2}) \leq \norm{\mathcal{G}_{\beta_1}(\theta_1) -\mathcal{G}_{\beta_2}(\theta_2)}_{L^2_{\lambda}} \leq L_{\mathcal{G}}(M)\norm{\theta_1-\theta_2} _{(H^{{\kappa}})^*} + C_{\mathcal{U}} \norm{\beta_1-\beta_2}_{\mathcal{U}}.
\end{equation*}
Thus, for $\bar{m} >0 $, 
\begin{equation}
    \begin{aligned}
        &\log N(\{\beta_0,\hat{\beta}\}\times \Theta_N(\varepsilon),h,\bar{m}\varepsilon) \\
        \leq &\log[N(\Theta_N(\varepsilon),\Vert\cdot\Vert_{(H^{\kappa})^*},\bar{m}\varepsilon/L(M\varepsilon/\varepsilon_N))\cdot N(\{\beta_0,\hat{\beta}\},\Vert\cdot\Vert_{\mathcal{U}},\bar{m}\varepsilon/C_{\mathcal{U}})]\\
        \leq & \log N(\Theta_N(\varepsilon),\Vert\cdot\Vert_{(H^{\kappa})^*},\bar{m}\varepsilon/L(M\varepsilon/\varepsilon_N)) + \log 2.
    \end{aligned}
\end{equation}
Combined with the upper bound (\ref{entropybound}), we have
\begin{equation*}
    \log N(\{\beta_0,\hat{\beta}\}\times \Theta_N(\varepsilon),h,\bar{m}\varepsilon) \leq C_EN\varepsilon^2,
\end{equation*} 
where \( C_E \) is a constant depending on \( \alpha, \kappa, d \). Then, following the same argument as in the proof of Theorem \ref{mainthm}, we verified inequality (\ref{C1ms}).
\par
(ii) For (\ref{C2ms}), by the definition of probability measure $\nu$, we have
\begin{equation*}
    \hat{\Pi}\bbra{(\{\beta_0,\hat{\beta}\}\times\Theta_N(\varepsilon))^c} = \Pi_N(\Theta_N(\varepsilon)^c).
\end{equation*}
The upper bound of the last probability is given by (\ref{setbound}), which implies
\begin{equation*}
    \hat{\Pi}\bbra{(\{\beta_0,\hat{\beta}\}\times\Theta_N(\varepsilon))^c} \leq \exp(-CN\varepsilon^2).
\end{equation*}
\par
(iii) For (\ref{C3ms}), through direct computation, we obtain
\begin{equation*}
    D_2(P^{(N)}_{\beta_0,\theta_0}\Vert P^{(N)}_{\beta,\theta}) = N\log\int_{\mathcal{X}}\exp\{[\mathcal{G}_{\beta}(\theta)(x)-\mathcal{G}_{\beta_0}(\theta_0)(x)]^2\}d\lambda(x).
\end{equation*}
Then, combining this with (\ref{boundmodel}), we deduce that
\begin{equation*}
    \begin{aligned}
        &\hat{\Pi}\bbra{D_2(P^{(N)}_{\beta_0,\theta_0}\Vert P^{(N)}_{\beta,\theta}) \leq C_3N\varepsilon_N^2}\\
        &=\hat{\Pi}\bbra{\log\int_{\mathcal{X}}\exp\{[\mathcal{G}_{\beta}(\theta)(x)-\mathcal{G}_{\beta_0}(\theta_0)(x)]^2\}d\lambda(x) \leq C_3\varepsilon_N^2}\\
        &\geq \hat{\Pi}\bbra{\log\int_{\mathcal{X}}\exp\{[\mathcal{G}_{\beta}(\theta)(x)-\mathcal{G}_{\beta_0}(\theta_0)(x)]^2\}d\lambda(x) \leq C_3\varepsilon_N^2, \norm{\theta - \theta_0}_{\mathcal{R}}\leq M'}\\
        & \geq \hat{\Pi}\bbra{\mexp{4U^2_{\mathcal{G}}(\Bar{M})}\norm{\mathcal{G}_{\beta}(\theta)-\mathcal{G}_{\beta_0}(\theta_0)}^2_{L_{\lambda}^2}\leq C_3\varepsilon_N^2, \norm{\theta - \theta_0}_{\mathcal{R}}\leq M'}\\
        & \geq \hat{\Pi}\bbra{\norm{\mathcal{G}_{\beta}(\theta)-\mathcal{G}_{\beta_0}(\theta_0)}_{L_{\lambda}^2}\leq \sqrt{C_3}U^{-1}\varepsilon_N, \norm{\theta - \theta_0}_{\mathcal{R}}\leq M'}
    \end{aligned}
\end{equation*}
for some constants $M'>0$, $\bar{M} = M' + \norm{\theta_0}_{\mathcal{R}}$ and $U = \mexp{2U^2_{\mathcal{G}}(\Bar{M})}$. Next, using {\color{black}Lipschitz} condition (\ref{lipmodel}), we have
\begin{equation*}
    \begin{aligned}
        &\hat{\Pi}\bbra{D_2(P^{(N)}_{\beta_0,\theta_0}\Vert P^{(N)}_{\beta,\theta}) \leq C_3N\varepsilon_N^2}\\
        \geq&\hat{\Pi}\bbra{L_{\mathcal{G}}(\Bar{M})\norm{\theta-\theta_0} _{(H^{{\kappa}})^*} + C_{\mathcal{U}} \norm{\beta-\beta_0}_{\mathcal{U}}\leq \sqrt{C_3}U^{-1}\varepsilon_N, \norm{\theta - \theta_0}_{\mathcal{R}}\leq M'}\\
        \geq&\Pi_N\bbra{L_{\mathcal{G}}(\Bar{M})\norm{\theta-\theta_0} _{(H^{{\kappa}})^*}\leq \sqrt{C_3}U^{-1}\varepsilon_N/2,\norm{\theta - \theta_0}_{\mathcal{R}}\leq M'}\\
        &\cdot\nu\bbra{C_{\mathcal{U}} \norm{\beta-\beta_0}_{\mathcal{U}}\leq \sqrt{C_3}U^{-1}\varepsilon_N/2}\\
        \geq& \frac{1}{2}\Pi_N\bbra{L_{\mathcal{G}}(\Bar{M})\norm{\theta-\theta_0} _{(H^{{\kappa}})^*}\leq \sqrt{C_3}U^{-1}\varepsilon_N/2,\norm{\theta - \theta_0}_{\mathcal{R}}\leq M'}.
    \end{aligned}
\end{equation*}
From the proof of Theorem \ref{mainthmsv}, we know 
\begin{equation*}
    \Pi_N\bbra{L_{\mathcal{G}}(\Bar{M})\norm{\theta-\theta_0} _{(H^{{\kappa}})^*}\leq \sqrt{C_3}U^{-1}\varepsilon_N/2,\norm{\theta - \theta_0}_{\mathcal{R}}\leq M'} \geq \exp(-cN\varepsilon_N^2)
\end{equation*}
for some constant $c>0$. Thus, we have 
\begin{equation*}
    \hat{\Pi}\left((\beta,\theta) : D_2(P^{(N)}_{\beta_0,\theta_0}\Vert P^{(N)}_{\beta,\theta})\leq C_3N\varepsilon_N^2\right)\geq\exp(-C_2N\varepsilon_N^2)
\end{equation*}
for some $C_2 >0$.
\end{proof}
\section{Proofs of results for specific inverse problems}
\subsection{Darcy flow problem}
In this subsection, we will prove Theorem \ref{mainthmDarcy} by verifying Conditions \ref{condreg} and \ref{condstab} using lemmas introduced below. We will provide regularity and conditional stability estimates for the forward map \( \mathcal{G}(\theta) \) as defined in \eqref{forwardmapDarcyflow} of the main text. We begin by documenting an example of the link function from Section 4.1 of \cite{IntroNonLinear_nickl2020convergence}.
\begin{example}\label{Darcylinkex}
Define the function $\phi: \mathbb{R} \mapsto (0,\infty)$ by
\[\phi(t) = e^t\textbf{1}_{\{t<0\}}+(1+t)\textbf{1}_{\{t\geq0\}}.\]
Let $\psi: \mathbb{R} \mapsto (0,\infty)$ be a smooth, compactly supported function with $\int_{\mathbb{R}}\psi=1$, and write $\phi*\psi= \int_{\mathbb{R}}\phi(\cdot-y)\psi(y)dy$ for their convolution. It follows from elementary calculations that, for any $K_{\min}\in \mathbb{R},$
\[\Phi:\mathbb{R} \mapsto (K_{\min},\infty), \Phi:=K_{\min}+\frac{1-K_{\min}}{\psi*\phi(0)}\psi*\phi,\]
is a link function satisfying the properties in Section \ref{SubsectionDarcyFlow}.
\end{example}
\begin{lemma}\label{Dcbound}
there exists a constant $C = C(\alpha,d,\mathcal{X},K_{\min})$ such that
\begin{equation}
    \Vert u_{f_{\theta}}\Vert _{H^{\alpha+1}} \leq CM^{\alpha^3 + \alpha^2}(\Vert g\Vert _{H^{\alpha-1}}^{\alpha + 1} \vee \Vert g\Vert _{H^{\alpha-1}}^{1/(\alpha + 1)})
\end{equation}
for integer $\alpha >2 + d/2$, and $\norm{\theta}_{H^{\alpha}}\leq M$ with $M>1$.
\end{lemma}
\begin{proof}
Using \cite[Lemma 23]{IntroNonLinear_nickl2020convergence}, we have 
\[ \Vert u_{f_{\theta}}\Vert _{H^{\alpha+1}} \leq C(1+\Vert f_\theta\Vert _{H^{\alpha}}^{\alpha^2 + \alpha})(\Vert g\Vert _{H^{\alpha-1}}^{\alpha + 1} \vee \Vert g\Vert _{H^{\alpha-1}}^{1/(\alpha + 1)})\]
for some constant $C = C(\alpha,d,\mathcal{X},K_{\min})$. Combined with \cite[Lemma 29]{IntroNonLinear_nickl2020convergence}, we obtain
\[\Vert u_{f_{\theta}}\Vert _{H^{\alpha+1}} \leq C(1 + (1 + \Vert \theta\Vert _{H^{\alpha}}^{\alpha})^{\alpha^2 + \alpha})(\Vert g\Vert _{H^{\alpha-1}}^{\alpha + 1} \vee \Vert g\Vert _{H^{\alpha-1}}^{1/(\alpha + 1)}).\]
\end{proof}
\begin{lemma}\label{Dclip}
There exists a constant $C = C(\mathcal{X},d,K_{\min},\norm{g}_{\infty})$ such that
\begin{equation}
    \Vert \mathcal{G}(\theta_1)-\mathcal{G}(\theta_2)\Vert _{L^2} \leq C (1+\norm{\theta_1}_{C^1}^3\vee\norm{\theta_2}_{C^1}^3)\Vert \theta_1-\theta_2\Vert _{(H^1)^*}
\end{equation}
for integer $\alpha > 2 + d/2$, and $\theta_1, \theta_2 \in C^1(\mathcal{X})$.
\end{lemma}
\begin{proof}
For the operator $\mathcal{L}_{f}$ defined by
\[\mathcal{L}_{f}u = \nabla\cdot(f\nabla u)\]
with Dirichlet boundary condition, we deduce that
\begin{align*}
    \mathcal{L}_{f_{\theta_2}}[\mathcal{G}(\theta_1)-\mathcal{G}(\theta_2)] &= (\mathcal{L}_{f_{\theta_2}} - \mathcal{L}_{f_{\theta_1}})[u_{f_{\theta_1}}]\\
    & = \nabla \cdot((f_{\theta_2}-f_{\theta_1})\nabla u_{f_{\theta_1}}).
\end{align*}
Using (5.19) in \cite{IntroNonLinear_nickl2020convergence} and the property of link function, we have
\begin{align*}
    \Vert \mathcal{G}(\theta_1)-\mathcal{G}(\theta_2)\Vert _{L^2} &= \Vert \mathcal{L}_{f_{\theta_2}}^{-1}[\nabla \cdot((f_{\theta_1}-f_{\theta_2})\nabla u_{f_{\theta_1}})]\Vert _{L^2}\\
    &\leq C(1 + \norm{f_{\theta_2}}_{C^1}) \Vert \nabla \cdot((f_{\theta_1}-f_{\theta_2})\nabla u_{f_{\theta_1}})\Vert _{(H^2)^*}\\
    &\leq C(1 + \norm{\theta_2}_{C^1}) \Vert \nabla \cdot((f_{\theta_1}-f_{\theta_2})\nabla u_{f_{\theta_1}})\Vert _{(H^2)^*}
\end{align*}
for a constant $C = C(K_{\min},\mathcal{X},d)$. Then it is sufficient to give an upper bound of \\ $\Vert \nabla \cdot((f_{\theta_1}-f_{\theta_2})\nabla u_{f_{\theta_1}})\Vert _{(H^2)^*}$ as follows:
\begin{align*}
    \Vert\nabla \cdot((f_{\theta_1}-f_{\theta_2})\nabla u_{f_{\theta_1}})\Vert_{(H^2)^*} 
    & = \mathop{\sup}_{\varphi \in H_0^2,\ \Vert \varphi\Vert _{H^2} \leq 1}\left\vert\int_{\mathcal{X}}\varphi\nabla \cdot((f_{\theta_1}-f_{\theta_2})\nabla u_{f_{\theta_1}})\right\vert\\
    & = \mathop{\sup}_{\varphi \in H_0^2,\ \Vert \varphi\Vert _{H^2} \leq 1}\left\vert\int_{\mathcal{X}}(f_{\theta_1}-f_{\theta_2}) \nabla u_{f_{\theta_1}} \cdot \nabla\varphi\right\vert\\
    & \leq \norm{f_{\theta_1}-f_{\theta_2}}_{(H^1)^*} \cdot \mathop{\sup}_{\varphi \in H_0^2,\ \Vert \varphi\Vert _{H^2} \leq 1} \norm{\nabla u_{f_{\theta_1}} \cdot \nabla\varphi}_ {H^1}\\
    & \lesssim \norm{u_{f_{\theta_1}}}_{C^2}\norm{f_{\theta_1}-f_{\theta_2}}_{(H^1)^*}.
\end{align*}
Using Lemmas 22 and 29 in \cite{IntroNonLinear_nickl2020convergence}, we obtain
\begin{align*}
    \norm{u_{f_{\theta_1}}}_{C^2}\norm{f_{\theta_1}-f_{\theta_2}}_{(H^1)^*} 
    &\leq C(1 + \norm{f_{\theta_1}}^2_{C^1}) (1+\norm{\theta_1}_{C^1}\vee\norm{\theta_2}_{C^1})\norm{\theta_1-\theta_2}_{(H^1)^*}\\
    &\leq C(1 + \norm{\theta_1}_{C^1}^2) (1+\norm{\theta_1}_{C^1}\vee\norm{\theta_2}_{C^1})\norm{\theta_1-\theta_2}_{(H^1)^*}.
\end{align*}
Then, for $\theta_1, \theta_2 \in C^1(\mathcal{X})$, we have
\begin{equation*}
    \Vert \mathcal{G}(\theta_1)-\mathcal{G}(\theta_2)\Vert _{L^2} \leq C (1+\norm{\theta_1}_{C^1}^3\vee\norm{\theta_2}_{C^1}^3)\Vert \theta_1-\theta_2\Vert _{(H^1)^*}
\end{equation*}
with $C = C(\mathcal{X},d,K_{\min},\norm{g}_{\infty})$.
\end{proof}
\begin{lemma}\label{Dcstab}
there exists a constant $C = C(B,\alpha,K_{\min},g_{\min},\mathcal{X},d)$ such that
\begin{equation}
    \Vert f_{\theta} - f_{\theta_0}\Vert _{L^2} \leq C M^{2\alpha^2 + 1}\Vert u_f-u_{f_0}\Vert _{L^2}^{\frac{\alpha-1}{\alpha+1}}
\end{equation}
for $\Vert f_{\theta_0}\Vert _{C^1} \vee \Vert u_{f_{\theta_0}}\Vert _{C^2} \leq B$, integer $\alpha > 2 + d/2$, and $\Vert \theta\Vert _{H^{\alpha}} \leq M$ with $M>1$.
\end{lemma}
\begin{proof}
    From \cite[Lemma 24]{IntroNonLinear_nickl2020convergence}, we know
    \[\Vert f_{\theta} - f_{\theta_0}\Vert _{L^2} \leq C\Vert f_{\theta} \Vert_{C^1} \Vert u_{f_\theta}-u_{f_{\theta_0}} \Vert_{H^2}\]
    for a constant $C = C(B,\alpha,K_{\min},g_{\min},\mathcal{X},d)$.
    Then we use interpolation inequality for Sobolev norms and Lemmas 23 and 29 in \cite{IntroNonLinear_nickl2020convergence} to bound $\Vert u_{f_\theta}-u_{f_{\theta_0}} \Vert_{H^2}$ as follows:
    \begin{align*}
        \Vert u_{f_\theta}-u_{f_{\theta_0}} \Vert_{H^2} &\lesssim \Vert u_{f_\theta}-u_{f_{\theta_0}} \Vert_{L^2}^{\frac{\alpha-1}{\alpha+1}} \Vert u_{f_\theta}-u_{f_{\theta_0}} \Vert_{H^{\alpha+1}}^{\frac{2}{\alpha+1}}\\
        &\lesssim \Vert u_{f_\theta}-u_{f_{\theta_0}} \Vert_{L^2}^{\frac{\alpha-1}{\alpha+1}} (2 + \Vert f_{\theta} \Vert_{H^\alpha}^{\alpha^2 +\alpha} + \Vert f_{\theta_0} \Vert_{H^\alpha}^{\alpha^2 +\alpha})^{\frac{2}{\alpha+1}}\\
        &\lesssim \Vert u_{f_\theta}-u_{f_{\theta_0}} \Vert_{L^2}^{\frac{\alpha-1}{\alpha+1}} \left[2 + (1+\Vert \theta \Vert_{H^\alpha}^{\alpha})^{\alpha^2 +\alpha} + (1+\Vert \theta_0 \Vert_{H^\alpha}^{\alpha})^{\alpha^2 +\alpha}\right]^{\frac{2}{\alpha+1}}.
    \end{align*}
    We use \cite[Lemma 29]{IntroNonLinear_nickl2020convergence} again to give a bound of $\Vert f_{\theta} \Vert_{C^1}$:
    \[\Vert f_{\theta} \Vert_{C^1} \lesssim 1 + \Vert \theta \Vert_{C^1}.\]
    Finally, we have a bound of $\Vert f_{\theta} - f_{\theta_0}\Vert _{L^2}$ as follows:
    \begin{equation*}
        C(1 + \Vert \theta \Vert_{C^1}) \left[2 + (1+\Vert \theta \Vert_{H^\alpha}^{\alpha})^{\alpha^2 +\alpha} + (1+\Vert \theta_0 \Vert_{H^\alpha}^{\alpha})^{\alpha^2 +\alpha}\right]^{\frac{2}{\alpha+1}} \Vert u_{f_\theta}-u_{f_{\theta_0}} \Vert_{L^2}^{\frac{\alpha-1}{\alpha+1}}.
    \end{equation*}
    Thus, for $M>1$, we obtain
    \begin{equation*}
        \Vert f_{\theta} - f_{\theta_0}\Vert _{L^2} \leq C M^{2\alpha^2 + 1}\Vert u_f-u_{f_0}\Vert _{L^2}^{\frac{\alpha-1}{\alpha+1}}.
    \end{equation*}
\end{proof}
\begin{proof}[Proof of Theorem \ref{mainthmDarcy}]
We will verify Conditions \ref{condreg} and \ref{condstab} for the forward map \( \mathcal{G} \) of problem (\ref{Darcy}) with \( \mathcal{R} = H^{\alpha} \). From Theorem \ref{Dcbound}, we have
\[\mathop{\sup}_{\theta \in B_{\mathcal{R}}(M)}\Vert u_{f_{\theta}} \Vert_{H^{t+1}} \leq CM^{t^3+t^2}\]
for $t = \lceil d/2 \rceil+2$. The above inequality, combined with the Sobolev embedding \( H^2 \subset C^0 \), implies that
\[\mathop{\sup}_{\theta \in B_{\mathcal{R}}(M)}\mathop{\sup}_{x \in \mathcal{X}}\vert\mathcal{G}(\theta)(x)\vert \leq CM^p.\]
Therefore, we have condition (\ref{bound}) verified with $p = t^3 + t^2$.

Theorem \ref{Dclip} implies that
\[\Vert \mathcal{G}(\theta_1)-\mathcal{G}(\theta_2)\Vert _{L^2} \leq C (1+\norm{\theta_1}_{C^1}^3\vee\norm{\theta_2}_{C^1}^3)\Vert \theta_1-\theta_2\Vert _{(H^1)^*}\]
for $\theta_1,\theta_2 \in \mathcal{R}$.
Then, for an integer $\alpha > 2 + d/2$,
\[\Vert \mathcal{G}(\theta_1)-\mathcal{G}(\theta_2)\Vert _{L^2} \leq C (1+\norm{\theta_1}_{H^{\alpha}}^3\vee\norm{\theta_2}_{H^{\alpha}}^3)\Vert \theta_1-\theta_2\Vert _{(H^1)^*}\]
which verifies condition (\ref{lip}) with ${\kappa}=1$ and $l=3$.

For condition (\ref{stab}), Theorem \ref{Dcstab} implies that
\[\Vert f_{\theta} - f_{\theta_0}\Vert _{L^2}^{\frac{\alpha+1}{\alpha-1}} \leq C M^{\frac{(2\alpha^2+1)(\alpha+1)}{\alpha-1}}\Vert u_f-u_{f_0}\Vert _{L^2}\]
for $\theta \in B_{\mathcal{R}}(M)$.

Given our requirement on \( \alpha \), we have
\[\alpha + {\kappa} \geq 2d = \frac{d(l+1)}{2}.
\]
In conclusion, using Theorem \ref{finalthm} with the conditions verified above, we obtain
\[P_{\theta_0}^{(N)}\hat{Q} \Vert f_{\theta} - f_{\theta_0}\Vert_{L^2}^{\frac{2}{p+q+1}\cdot\frac{\alpha+1}{\alpha-1}}\lesssim \varepsilon_N^{\frac{2}{p+q+1}}\log N,\]
where $p = t^3 + t^2$ and $q = \frac{(2\alpha^2+1)(\alpha+1)}{\alpha-1}$.
\end{proof}
\subsection{Inverse potential problem for a subdiffusion equation}
In this subsection, we will prove the theorems in Sections 4.2 and 6.2. We first provide some preliminary knowledge.  In the followings, we assume $q\in \mathcal{I}$. 
We define the operator $A = -\Delta  + qI$ on the domain $\Omega$ with a homogeneous Dirichlet boundary condition and 
$\{ \lambda_j(q) \}_{j=1}^{\infty}$  and $\{ \varphi_j(q) \}_{j=1}^{\infty}$ are, respectively, the eigenvalues and eigenfunctions.
\par
Moreover, for any $v \in  H^2_0 (\Omega )$ and $q \in \mathcal{I}$, the following two-sided inequality
holds:
\begin{equation}\label{Fracnormestimate}
c_1\norm{v}_{H^2}(\Omega)\leq\norm{A_qv}_{L^2(\Omega)} + \norm{v}_{L^2}(\Omega) \leq c_2\norm{v}_{H^2(\Omega)}
\end{equation}
with constants $c_1>0$ and $c_2>0$ independent of $q$.
Because
\begin{equation*}
\norm{A_qv}_{(H^{1}_0)^*} = \sup_{\norm{\phi}_{H^1_0}\leq 1}\pdt{A_qv}{\phi}_{L^2} \geq  \frac{\pdt{A_qv}{v}_{L^2}}{\norm{v}_{H^1}}
    =  \frac{\pdt{\nabla v}{\nabla v}_{L^2} + \pdt{qv}{v}_{L^2}}{\norm{v}_{H^1}}
    \geq C(\Omega)\norm{v}_{H^1},
\end{equation*}
we further have
\begin{equation}\label{Fracnormestimate*}
c_1\norm{v}_{H^2(\Omega)}\leq\norm{A_qv}_{L^2(\Omega)}\leq c_2\norm{v}_{H^2(\Omega)},
\end{equation}
with constants $c_1>0$ and $c_2>0$ independent of $q$.
\par
For any $s\in \mathbb{R}$, we define the Hilbert space $\Dot{H}^s_q(\Omega)$ by
\[\Dot{H}^s_q(\Omega) = \Big\{v\in L^2(\Omega) : \sum_{j=1}^{\infty}\lambda^s_j\abs{(v,\varphi_j)}^2 < \infty \Big\}\]
with the norm
\[\norm{v}_{\Dot{H}^s_q(\Omega)}^2 = \sum_{j=1}^{\infty}\lambda^s_j\abs{(v,\varphi_j)}^2,\]
where we will omit the $q$ in the notation $\lambda_j,\phi_j$ when there is no ambiguity.
If $\Omega$ has a smooth boundary, the space $\Dot{H}^s_q(\Omega)$ can also be defined by
\begin{equation*}
\Dot{H}^s_q(\Omega) = \Big\{v\in H^s(\Omega) : v=A_qv=\dots=A_q^{j-1}v=0\ \text{on}\ \partial\Omega \Big\}.
\end{equation*}
for $s\in\mathbb{N}$ and $j\in\mathbb{N}$ such that $2j-3/2 < s < 2j+1/2$.
In particular, $\Dot{H}^2_q(\Omega) = H_0^2$ (see \cite[A.2.4]{jin2021fractional}).
\par
We denote $h_q$ as the solution to the equation:
\begin{equation*}
    \left\{\begin{aligned}
    & - \partial_{xx}h_q + qh_q = f \quad \mbox{in } \Omega,\\
    &h_q(0) = a_0, h_q(1) = a_1.
    \end{aligned}\right.  
\end{equation*}
\par
We can represent the solution $u_{\beta,q}$ to the problem \eqref{Fractional} in the main text using the eigenpairs $\{ \lambda_j(q) \}_{j=1}^{\infty}$ and $\{ \varphi_j(q) \}_{j=1}^{\infty}$ as described in \cite[section 6.2]{jin2021fractional}. Specifically, we define two solution operators $F_{\beta,q}(t)$ and $E_{\beta,q}(t)$ by
\begin{equation*}
F_{\beta,q}(t)v = \sum_{j=1}^{\infty} E_{\beta,1}(-\lambda_jt^{\beta})(v,\varphi_j)\varphi_j,\quad \text{and\ }E_{\beta,q}(t)v = \sum_{j=1}^{\infty} t^{\beta-1}E_{\beta,\beta}(-\lambda_jt^{\beta})(v,\varphi_j)\varphi_j,
\end{equation*}
where we will omit the $q$ in the notation $\lambda_j,\phi_j$ when there is no ambiguity, and $E_{\beta,1}, E_{\beta,\beta}$ are the Mittag-Leffler functions (see \cite[section 3.1]{jin2021fractional}.
Then, the solution $u_{\beta,q}$ of problem \eqref{Fractional} can be represented as
\begin{equation}\label{Fracsolutions}
u_{\beta,q}(t) = h_q + F_{\beta,q}(t)(u_0-h_q).
\end{equation}
\textbf{Proof of theorems in Section 4.2 :} 
We will prove Theorem \ref{mainthmFrac} by verifying Conditions \ref{condreg} and \ref{condstab}.
Next, we estimate the regularity and conditional stability of problem \eqref{Fractional}.
\begin{lemma}\label{FracoperatorReg}
Let $q \in \mathcal{I}$, $\beta \in [\beta^-,\beta^+]$ for $0<\beta^-<\beta^+<1$. Then, for any $t>0$ and any $v \in L^2 (\Omega )$, there exists $c>0$ independent of $q$ and $\beta$ such that
\begin{equation*}
    \norm{F_{\beta,q}(t)v}_{L^2} \leq ct^{-\beta}\norm{v}_{(H^2_0)^*},\quad \norm{F'_{\beta,q}(t)v}_{L^2} \leq ct^{-\beta-1}\norm{v}_{(H^2_0)^*}.
\end{equation*}
\end{lemma}
\begin{proof}
For any $\phi \in  H^2_0 (\Omega )$ and $q \in \mathcal{I}$, definition of $\Dot{H}^s_q(\Omega)$ and inequality \eqref{Fracnormestimate*} imply that 
\begin{equation}\label{Fracnormestimate**}
    c_1\norm{\phi}_{H^2(\Omega)}\leq\norm{\phi}_{\Dot{H}^2_q(\Omega)}\leq c_2\norm{\phi}_{H^2(\Omega)}
\end{equation}
with $c_1>0$ and $c_2>0$ independent of $q$. Using the above inequality, there exists a constant $c>0$ independent of $q$ such that
\begin{equation}\label{FracoperatorReg1}
    \norm{v}_{\Dot{H}^{-2}_q(\Omega)}= \sup_{\norm{\phi}_{\Dot{H}^{2}_q(\Omega)}\leq 1}\pdt{v}{\phi}_{L^2} \leq \sup_{\norm{\phi}_{H_0^{2}(\Omega)}\leq c}\pdt{v}{\phi}_{L^2} \leq c \norm{v}_{(H_0^2(\Omega))^*}.
\end{equation}
By the definition of $F_{\beta,q}(t)$, part (a) of Lemma 2.3 in \cite{Frac_dang2018continuity}, and inequality \eqref{FracoperatorReg1}, we have
\begin{equation*}
    \begin{aligned}
        &\norm{F_{\beta,q}(t)v}_{L^2}^2 = \sum_{j=1}^{\infty} \abs{E_{\beta,1}(-\lambda_jt^{\beta})}^2\pdt{v}{\varphi}_{L^2}^2\\
=&t^{-2\beta}\sum_{j=1}^{\infty}\lambda_j^2t^{2\beta}\abs{E_{\beta,1}(-\lambda_jt^{\beta})}^2\lambda_j^{-2}\pdt{v}{\varphi}_{L^2}^2\\
\leq&C(\beta^-,\beta^+)t^{-2\beta}\sum_{j=1}^{\infty}\frac{\lambda_j^2t^{2\beta}}{(1+\lambda_jt^{\beta})^2}\lambda_j^{-2}\pdt{v}{\varphi}_{L^2}^2.\\
\leq&C(\beta^-,\beta^+)t^{-2\beta}\norm{v}_{\Dot{H}^{-2}_q(\Omega)}^2\leq C(\beta^-,\beta^+)t^{-2\beta}\norm{v}_{(H_0^2(\Omega))^*}^2.
    \end{aligned}
\end{equation*}
\par
For the operator $F'_{\beta,q}(t)$, we have
$\norm{F'_{\beta,q}(t)v}_{L^2} = \norm{A_qE_{\beta,q}(t)v}_{L^2}$ (see \cite[Lemma 6.2]{jin2021fractional}).
Again the definition of $E_{\beta,q}(t)$, part (a) of Lemma 2.3 in \cite{Frac_dang2018continuity}, and inequality \eqref{FracoperatorReg1} imply that
\begin{equation*}
    \begin{aligned}
        &\norm{A_qE_{\beta,q}(t)v}_{L^2} = \sum_{j=1}^{\infty} \lambda_j^2t^{2\beta-2}\abs{E_{\beta,\beta}(-\lambda_jt^{-\beta})}^2\pdt{v}{\varphi}_{L^2}^2\\
        =&t^{-2\beta-2}\sum_{j=1}^{\infty}\lambda_j^4t^{4\beta}\abs{E_{\beta,\beta}(-\lambda_jt^{-\beta})}^2\lambda_j^{-2}\pdt{v}{\varphi}_{L^2}^2\\
        \leq&C(\beta^-,\beta^+)t^{-2\beta-2}\sum_{j=1}^{\infty}\frac{\lambda_j^4t^{4\beta}}{(1+(\lambda_jt^{-\beta})^2)^2}\lambda_j^{-2}\pdt{v}{\varphi}_{L^2}^2\\
        \leq& C(\beta^-,\beta^+)t^{-2\beta-2}\norm{v}_{\Dot{H}^{-2}_q(\Omega)}\leq C(\beta^-,\beta^+)t^{-2\beta-2}\norm{v}_{(H^2_0(\Omega))^*}.
    \end{aligned}
\end{equation*}
\end{proof}
\begin{lemma}\label{FracderivativeReg}
Let $q_1,q_2 \in \mathcal{I}$, $\beta \in [\beta^-,\beta^+]$. Then, for a sufficiently large fixed $T_0 >0$ and any $t \geq T_0$, there exists $c>0$ independent of $q_1,q_2,\beta$, and $T_0$ such that
\begin{equation*}
    \norm{\partial_t^{\beta}u_{q_1}(t) - \partial_t^{\beta}u_{q_2}(t)}_{L^2(\Omega)} \leq ct^{-\beta}\norm{q_1-q_2}_{(H^2_0(\Omega))^*}.
\end{equation*}
\end{lemma}
\begin{proof}
Let $w = u_{q_1}-u_{q_2}.$ Then $w$ solves
\begin{equation}\label{FracderivativeRegEquation}
    \left\{\begin{array}{ll}
     \partial_t^{\beta} w - \Delta w + q_1w = (q_2-q_1)u_{q_2}&  \mbox{in } \Omega\times(0,T],\\
     w=0 & \mbox{on } \partial\Omega\times(0,T],\\
     w(0) = 0 & \mbox{in } \Omega.
\end{array}\right.  
\end{equation}
Referring to the solution representation formula outlined in \cite[Section 6.2.1]{jin2021fractional}, the solution for equation \eqref{FracderivativeRegEquation} is given by:
\[w(t) = \int_0^t E_{q_1}(t-s)(q_2-q_1)u_{q_2}(s)ds.\]
Then, \eqref{FracderivativeRegEquation} combined with the identities $\partial_t^{\beta}F_{q}=-A_qF_q(t)$ and $A_qE_{q}=-A_qF_q(t)$ \cite[Lemma 6.3]{jin2021fractional} leads to
\begin{equation*}
    \begin{aligned}
       \partial_t^{\beta} w(t) = &-A_q \int_0^t E_{q_1}(t-s)(q_2-q_1)u_{q_2}(s)ds + (q_2-q_1)u_{q_2}(t)\\
       =&\int_0^t F'_{q_1}(t-s)(q_2-q_1)u_{q_2}(s)ds + (q_2-q_1)u_{q_2}(t)\\
       =&\partial_t\int_0^tF_{q_1}(t-s)(q_2-q_1)u_{q_2}(s)ds.
    \end{aligned}
\end{equation*}
Let $\phi(t)=\int_0^tF_{q_1}(t-s)(q_2-q_1)u_{q_2}(s)ds$. Combining inequality \eqref{Sobolevinter1} of the main text with Lemma \ref{FracoperatorReg} and Lemma 4.3 from \cite{jin2023inverse}, it implies that
\begin{equation*}
    \begin{aligned}
        \norm{\phi(t)}_{L^2(\Omega)}\leq&\int_0^t\norm{F_{q_1}(t-s)(q_2-q_1)u_{q_2}(s)}_{L^2(\Omega)}ds\\
        \leq&c\int_0^t(t-s)^{-\beta}\norm{(q_2-q_1)u_{q_2}(s)}_{(H_0^2(\Omega))^*}ds\\
        \leq&c\norm{q_1-q_2}_{(H_0^2(\Omega))^*}\int_0^t(t-s)^{-\beta}\norm{u_{q_2}(s)}_{H^2(\Omega)}ds\\
        \leq&c\norm{q_1-q_2}_{(H_0^2(\Omega))^*}\int_0^t(t-s)^{-\beta}(1+s^{-\beta})ds\\
        \leq&c(t^{1-\beta}+t^{1-2\beta})\norm{q_1-q_2}_{(H_0^2(\Omega))^*}\leq ct^{1-\beta}\norm{q_1-q_2}_{(H_0^2(\Omega))^*},
    \end{aligned}
\end{equation*}
where $t\geq T_0$ and $T_0$ is a large enough constant. Next, we easily deduce 
\[t\phi(t)=\int_0^t(t-s)F_{q_1}(t-s)(q_2-q_1)u_{q_2}(s)ds + \int_0^tF_{q_1}(s)(q_2-q_1)(t-s)u_{q_2}(t-s)ds,\]
which leads to
\begin{equation*}
    \begin{aligned}
        \partial_t(t\phi(t)) = &\int_0^t[(t-s)F'_{q_1}(t-s)+F_{q_1}(t-s)](q_2-q_1)u_{q_2}(s)ds \\
        &+\int_0^tF_{q_1}(s)(q_2-q_1)[u_{q_2}(t-s)+(t-s)u'_{q_2}(t-s)]ds =: I_1+I_2.
    \end{aligned}
\end{equation*}
Then, we bound \( I_1 \) and \( I_2 \) respectively. For \( I_1 \), inequalities \eqref{Sobolevinter1} in the main text and Lemma \ref{FracoperatorReg}, in conjunction with Lemma 4.3 from \cite{jin2023inverse}, imply that
\begin{equation*}
    \begin{aligned}
        \norm{I_1}_{L^2(\Omega)} \leq &\int_0^t\Big((t-s)\norm{F'_{q_1}(t-s)(q_2-q_1)u_{q_2}(s)}_{L^2(\Omega)} \\&+ \norm{F_{q_1}(t-s)(q_2-q_1)u_{q_2}(s)}_{L^2(\Omega)}\Big)ds\\
        \leq &c\norm{q_1-q_2}_{(H_0^2(\Omega))^*}\int_0^t(t-s)^{-\beta}\norm{u_{q_2}(s)}_{H^2(\Omega)}ds \\ \leq& ct^{1-\beta}\norm{q_1-q_2}_{(H_0^2(\Omega))^*},
    \end{aligned}
\end{equation*}
for any \( t \geq T_0 \), provided that \( T_0 \) is sufficiently large.
Similarly, we also have
\[\norm{I_2}_{L^2(\Omega)} \leq ct^{1-\beta}\norm{q_1-q_2}_{(H_0^2(\Omega))^*}.\]
Finally, the triangle inequality implies that for any $t>T_0$, there exists a constant $c>0$ independent of $q_1,q_2,\beta,T_0$ such that
\[t\norm{\phi'(t)}_{L^2(\Omega)}\leq\norm{\partial_t(t\phi(t))}_{L^2(\Omega)} + \norm{\phi(t)}_{L^2(\Omega)} \leq ct^{1-\beta}\norm{q_1-q_2}_{(H_0^2(\Omega))^*},\]
that is
\[\norm{\partial_t^{\beta}u_{q_1}(t) - \partial_t^{\beta}u_{q_2}(t)}_{L^2(\Omega)} = \norm{\phi'(t)}_{L^2(\Omega)} \leq ct^{-\beta}\norm{q_1-q_2}_{(H_0^2(\Omega))^*}.\]
\end{proof}
\begin{lemma}\label{FracLip}
Let \( q_1, q_2 \in \mathcal{I} \), \( \beta \in [\beta^-, \beta^+] \). Then, for a sufficiently large fixed \( T_0 > 0 \) and any \( t \geq T_0 \), there exists a constant \( c > 0 \) independent of \( q_1, q_2, \beta, \) and \( T_0 \) such that
\begin{equation*}
    \norm{u_{q_1}(t) - u_{q_2}(t)}_{L^2(\Omega)} \leq c(1+t^{-\beta})\norm{q_1-q_2}_{(H_0^2(\Omega))^*}.
\end{equation*}
\end{lemma}
\begin{proof}
Because \[\partial_t^{\beta}u_{q} - \Delta u_{q}  + qu_{q}  = f,\]
we have \[ - \Delta (u_{q_1}-u_{q_2})  + q_2(u_{q_1}-u_{q_2})  = (q_2-q_1)u_{q_1}-(\partial_t^{\beta}u_{q_1}-\partial_t^{\beta}u_{q_2}).\]
Thus,
\begin{equation*}
    \begin{aligned}
    \norm{A_{q_2}(u_{q_1}-u_{q_2})(t)}_{(H_0^2(\Omega))^*}
        \leq \norm{(q_2-q_1)u_{q_1}(t)}_{(H_0^2(\Omega))^*} + \norm{\partial_t^{\beta}(u_{q_1}-u_{q_2})(t)}_{(H_0^2(\Omega))^*}.
    \end{aligned}
\end{equation*}
Lemma 4.3 in \cite{jin2023inverse} implies that 
\begin{equation*}
    \begin{aligned}
        \norm{(q_2-q_1)u_{q_1}(t)}_{(H_0^2(\Omega))^*} \leq& \norm{u_{q_1}(t)}_{H^2(\Omega)}\norm{q_2-q_1}_{(H_0^2(\Omega))^*}\\
        \leq&c(1+t^{-\beta})\norm{q_2-q_1}_{(H_0^2(\Omega))^*}.
    \end{aligned}
\end{equation*}
Combined with Lemma \ref{FracderivativeReg}, we have 
\begin{equation*}
    \norm{A_{q_2}(u_{q_1}-u_{q_2})(t)}_{(H_0^2(\Omega))^*} \leq c(1+t^{-\beta})\norm{q_1-q_2}_{(H_0^2(\Omega))^*}.
\end{equation*}
\par
Next, we prove $\norm{A_{q}v}_{(H_0^2(\Omega))^*} \geq c\norm{v}_{L^2(\Omega)}$ to complete the proof.
Using \eqref{FracoperatorReg1}, we deduce that
\begin{equation*}\label{Fraclip1}
    \begin{aligned}
       &\norm{A_{q}v}_{(H_0^2(\Omega))^*} \geq c\norm{A_{q}v}_{\Dot{H}_q^{-2}(\Omega)} \\  
       = &\sup_{\norm{A_q\phi}_{L^2}\leq 1}c\pdt{A_qv}{\phi}_{L^2}= \sup_{\norm{A_q\phi}_{L^2}\leq 1}c\pdt{v}{A_q\phi}_{L^2} = c\norm{v}_{L^2(\Omega)},
    \end{aligned}
\end{equation*}
with constant $c>0$ independent of $q$.
Thus, we finally have
\begin{equation*}
    \begin{aligned}
        \norm{u_{q_1}(t) - u_{q_2}(t)}_{L^2(\Omega)} \leq c\norm{A_{q_2}(u_{q_1}-u_{q_2})(t)}_{(H_0^2(\Omega))^*}
        \leq c(1+t^{-\beta})\norm{q_1-q_2}_{(H_0^2(\Omega))^*}
    \end{aligned}
\end{equation*}
with $c>0$ independent of $q_1,q_2,\beta,T_0$.
\end{proof}
\begin{lemma}\label{FracReg}
Let $u_0\in H^{\alpha}(\Omega)$, $f \in L^{\infty}(\Omega)\cap H^{\alpha}(\Omega)$, and $q\in H^{\alpha}(\Omega)\cap\mathcal{I}$ for some integer $\alpha \geq 0$. Let $\beta \in [\beta^-,\beta^+]$ for $0<\beta^-<\beta^+<1$. Then, there exists a constant $c>0$ independent of $q$ and $\beta$ such that
\begin{equation*}
    \norm{u_{q}(t)}_{H^{\alpha+2}} \leq c(1+t^{-\beta})(1+\norm{q}_{H^{\alpha}}^{1+2\alpha})(1+\norm{u_0}_{H^{\alpha}} + \norm{f}_{H^{\alpha}}^{1+\alpha/2} + \norm{f}_{\infty}).
\end{equation*}
\end{lemma}
\begin{proof}
Define $v_q := u_q - h_q$ and $v_0 = u_0 - h_q$. We observe that $v_q, v_0$ vanish on $\partial \Omega$. In the following proof, let $s,k\in\mathbb{N}$. From the representation of solutions \eqref{Fracsolutions}, we have
\[v_q = F_{\beta,q}(t)v_0.\]
We first estimate the $\Dot{H}^{s}_{q}$ norm of $v_q$. 
Assume that $s:=2k$ is even.
Using part (a) of Lemma 2.3 in \cite{Frac_dang2018continuity}, the $L^2$ norm of $A_q^k v_q$ can be bounded by
 \begin{equation*}
    \begin{aligned}
        &\norm{A_q^kv_q(t)}_{L^2}^2 = \norm{A_q^kF_q(t)v_0}_{L^2}^2 \\
        =&\norm{A_q^k\bbra{\sum_{j=1}^{\infty}E_{\beta,1}(-\lambda_j(q)t^{\beta})\pdt{v_0}{\varphi_j(q)}_{L^2}\varphi_j(q)}}_{L^2}^2\\
        =&\sum_{j=1}^{\infty} \lambda_j^{2k}(q) E_{\beta,1}^2(-\lambda_j(q)t^{\beta})\pdt{v_0}{\varphi_j(q)}_{L^2}^2\\
        \lesssim & \sum_{j=1}^{\infty} \frac{\lambda_j^{2k}(q)}{\bbra{1+\lambda_j(q)t^{\beta}}^2}\pdt{v_0}{\varphi_j(q)}_{L^2}^2\\
        \lesssim &t^{-2\beta} \sum_{j=1}^{\infty} \lambda_j^{2k-2}(q) \pdt{v_0}{\varphi_j(q)}_{L^2}^2
        \lesssim t^{-2\beta} \norm{A_q^{k-1}v_0}_{L^2}^2.
    \end{aligned} 
\end{equation*}
Combined with the definition of $A_q$ and Sobolev inequality \eqref{Sobolevinter2} of the main text, we further deduce that 
\begin{equation}\label{FracReg1}
    \norm{A_q^kv_q(t)}_{L^2} \lesssim t^{-\beta} \norm{A_q^{k-1}v_0}_{L^2} \lesssim t^{-\beta}(1+\norm{q}_{H^{2k-2}}^{k-1})\norm{v_0}_{H^{2k-2}}.
\end{equation}
By regularity estimates of elliptic problems \cite[Proposition 6.1.5]{IntroNonLinear_nickl2023bayesian} , we estimates the regularity for $h_q$ and have
\begin{equation*}
    \norm{h_q}_{H^2(\Omega)} \lesssim \bbra{1 + \norm{f}_{\infty}}, 
\end{equation*}
\begin{equation*}
    \norm{h_q}_{H^{k+2}(\Omega)} \lesssim \bbra{1+\norm{q}_{H^k(\Omega)}^{1+k/2}}\bbra{1 + \norm{f}_{H^k(\Omega)}^{1+k/2}}.
\end{equation*}
Combined with \eqref{FracReg1}, we have
\begin{equation}\label{FracReg2}
    \begin{aligned}
        \norm{A_q^kv_q(t)}_{L^2} \lesssim &t^{-\beta}(1+\norm{q}_{H^{2k-2}}^{k-1})\norm{v_0}_{H^{2k-2}} \\
        \lesssim & t^{-\beta}(1+\norm{q}_{H^{2k-2}}^{k-1})(\norm{u_0}_{H^{2k-2}}+\norm{h_q}_{H^{2k-2}})\\
        \lesssim &t^{-\beta}(1+\norm{q}_{H^{2k-2}}^{2k-2})(1+\norm{u_0}_{H^{2k-2}} + \norm{f}_{H^{2k-4}}^{k-1} + \norm{f}_{\infty}).
    \end{aligned} 
\end{equation}
where we let the term $\norm{f}_{H^{2k-4}}^{k-1}$ equal to $1$ when $k=1$.
Next, we assume that $s:=2k+1$ is odd.
Using the regularity for $h_q$ again, we estimate the $L^2$ norm of $A_q^{k+\frac{1}{2}} v_q$ similarly and obtain
\begin{equation}\label{FracRegdotodd}
    \begin{aligned}
        \norm{A_q^{k+\frac{1}{2}}v_q(t)}_{L^2}\lesssim & t^{-\beta}\norm{A_q^{k-\frac{1}{2}}v_0}_{L^2} \lesssim t^{-\beta}\norm{A_q^{k-1}v_0}_{H^1}\\
        \lesssim & t^{-\beta}(1+\norm{q}_{H^{2k-1}}^{k-1})\norm{v_0}_{H^{2k-1}}\\
        \lesssim & t^{-\beta}(1+\norm{q}_{H^{2k-1}}^{2k-\frac{3}{2}})(1+\norm{u_0}_{H^{2k-1}} + \norm{f}_{H^{2k-3\vee0}}^{k-\frac{1}{2}} + \norm{f}_{\infty}).
    \end{aligned}
\end{equation} 
\par
Next, we estimate the $H^{s}$ norm of $v_q$ with the mathematical induction method. Assume that $s:=2k$ is even. From \eqref{Fracnormestimate}, we have
\[\norm{\phi}_{H^2} \leq c(\norm{\phi}_{L^2}+\norm{A_q\phi}_{L^2})\quad \text{if}\ \phi \in \Dot{H}^{2}_{q},\]
where $c$ is independent of $q$. Assume that this statement is true for $\phi \in \Dot{H}^{2l}_{q}$, which is
\begin{equation}\label{Fracreg3}
    \norm{\phi}_{H^{2l}} \lesssim (1+\norm{q}_{H^{2l-2}}^{l-1})\sum_{j=0}^{l}\norm{A_q^j\phi}_{L^2}\quad \text{if}\ \phi \in \Dot{H}^{2l}_{q}.
\end{equation}
If $\phi \in \Dot{H}^{2l+2}_{q}\subset H^{2l+2}_0$, then $A_q\phi \in \Dot{H}^{2l}_{q}$.
Using \eqref{Sobolevinter2} and inequality $\norm{\Delta\phi}_{H^{2l}}\gtrsim \norm{\phi}_{H^{2l+2}}$ , we deduce that
\begin{equation*}
    \begin{aligned}
        \norm{A_q\phi}_{H^{2l}} \geq & \norm{\Delta\phi}_{H^{2l}} - \norm{q}_{H^{2l}}\norm{\phi}_{H^{2l}} \\
        \gtrsim & \norm{\phi}_{H^{2l+2}} - \norm{q}_{H^{2l}}\norm{\phi}_{H^{2l}},
    \end{aligned}
\end{equation*}
that is,
\begin{equation*}
    \norm{\phi}_{H^{2l+2}} \lesssim \norm{A_q\phi}_{H^{2l}} + \norm{q}_{H^{2l}}\norm{\phi}_{H^{2l}}.
\end{equation*}
Combined with inequalities \eqref{Fracreg3} of $\phi,A_q\phi \in \Dot{H}^{2l}_{q}$, we have
\begin{equation*}
    \begin{aligned}
        \norm{\phi}_{H^{2l+2}} \lesssim &(1+\norm{q}_{H^{2l-2}}^{l-1})\sum_{j=0}^{l}\norm{A_q^{j+1}\phi}_{L^2} + (1+\norm{q}_{H^{2l}}^{l})\sum_{j=0}^{l}\norm{A_q^j \phi}_{L^2} \\
        \lesssim &(1+\norm{q}_{H^{2l}}^{l})\sum_{j=0}^{l+1}\norm{A_q^j \phi}_{L^2}.
    \end{aligned}
\end{equation*}
Thus, by the method of mathematical induction, we have proved that
\begin{equation}\label{FracRegeven}
    \norm{v_q}_{H^{2k}} \lesssim (1+\norm{q}_{H^{2k-2}}^{k-1})\sum_{j=0}^{k}\norm{A_q^j {\color{black}v_q}}_{L^2}\quad \text{if}\ v_q \in \Dot{H}^{2k}_{q}.
\end{equation}
Next, we assume that $s:=2k+1$ is odd.
For $\phi \in \Dot{H}^{3}_{q}\subset H^{3}_{0}$, using \eqref{Sobolevinter2} of the main text and the inequality $\norm{\Delta\phi}_{H^{1}}\gtrsim \norm{\phi}_{H^{3}}$ , we deduce that
\begin{equation*}
    \begin{aligned}
        \norm{A_q\phi}_{H^{1}} \geq & \norm{\Delta\phi}_{H^{1}} - \norm{q}_{H^{1}}\norm{\phi}_{H^{1}} \\
        \gtrsim & \norm{\phi}_{H^{3}} - \norm{q}_{H^{1}}\norm{\phi}_{H^{1}},
    \end{aligned}
\end{equation*}
which implies
\begin{equation*}
    \norm{\phi}_{H^{3}} \lesssim \norm{A_q\phi}_{H^{1}} + \norm{q}_{H^{1}}\norm{\phi}_{H^{1}}.
\end{equation*}
By the method of mathematical induction, we can similarly prove that
\begin{equation}\label{Fracreg4}
    \norm{v_q}_{H^{2k+1}} \lesssim (1+\norm{q}_{H^{2k-1}}^{k})\sum_{j=0}^{k}\norm{A_q^j{\color{black}v_q}}_{H^1}\quad \text{if}\ v_q \in \Dot{H}^{2k+1}_{q}.
\end{equation}
For $v_q \in \Dot{H}^{2k+1}_{q}$, we note that $A_q^kv_q \in \Dot{H}^{1}_{q} \subset {H}^{1}_{0}$.
Thus, the Poincaré inequality $\norm{\nabla A_q^kv_q}_{L^2} \gtrsim \norm{A_q^kv_q}_{H^1}$ implies that 
\begin{equation*}
    \begin{aligned}
        \norm{A_q^{k+\frac{1}{2}}v_q}_{L^2}^2 = & \pdt{A_q^k\phi}{A_q^{k+1}v_q}_{L^2} =-\pdt{A_q^kv_q}{\Delta A_q^kv_q}_{L^2} + \pdt{A_q^kv_q}{qA_q^kv_q}_{L^2} \\
        = &\pdt{\nabla A_q^kv_q}{\nabla A_q^kv_q}_{L^2} + \pdt{A_q^kv_q}{qA_q^kv_q}_{L^2} \\
        \gtrsim& \norm{A_q^kv_q}_{H^1}^2 
    \end{aligned}
\end{equation*}
The last inequality combined with \eqref{Fracreg4} implies that
\begin{equation}\label{Fracreg5}
    \norm{v_q}_{H^{2k+1}} \lesssim (1+\norm{q}_{H^{2k-1}}^{k})\sum_{j=0}^{k}\norm{A_q^{j+\frac{1}{2}}{\color{black}v_q}}_{L^2}\quad \text{if}\ v_q \in \Dot{H}^{2k+1}_{q}.
\end{equation}
For $\phi \in \Dot{H}^{1}_{q}\subset H^{1}_{0}$, we deduce that
\begin{equation*}
    \begin{aligned}
        \norm{A_q^{\frac{1}{2}}\phi}_{L^2}^2 = & \pdt{\phi}{A_q\phi}_{L^2} =-\pdt{\phi}{\Delta\phi}_{L^2} + \pdt{\phi}{q\phi}_{L^2} \\
        = &\pdt{\nabla\phi}{\nabla\phi}_{L^2} + \pdt{\phi}{q\phi}_{L^2} \\
        \gtrsim& \norm{\phi}_{H^1}^2 ,
    \end{aligned}
\end{equation*}
which means that inequality \eqref{Fracreg5} is also true when $k=0$.
\par
Assume that $\alpha:= 2l$ is even integer for $l\in \mathbb{N}\cup\{0\}$. The estimate \eqref{FracReg2} implies that
\[\norm{A_q^{k+1}v_q(t)}_{L^2}\lesssim t^{-\beta}(1+\norm{q}_{H^{2k}}^{2k})(1+\norm{u_0}_{H^{2k}} + \norm{f}_{H^{2k}}^{k} + \norm{f}_{\infty})\]
for $k\in \mathbb{N}\cup\{0\}$ such that $0 \leq k\leq l$, which means $v_q \in \Dot{H}^{2l+2}_{q}$.
Using \eqref{FracRegeven}, we further have
\begin{equation}\label{FracReg6}
    \begin{aligned}
        \norm{v_q}_{H^{\alpha+2}} \lesssim &(1+\norm{q}_{H^{2l}}^{l})\sum_{j=0}^{l+1}\norm{A_q^j{\color{black}v_q}}_{L^2}\\ \lesssim &t^{-\beta}(1+\norm{q}_{H^{\alpha}}^{3\alpha/2})(1+\norm{u_0}_{H^{\alpha}} + \norm{f}_{H^{\alpha}}^{\alpha/2} + \norm{f}_{\infty}).
    \end{aligned}
\end{equation}
\par
Next, we assume that $\alpha:= 2l+1$ is an odd integer for $l\in \mathbb{N}\cup\{0\}$. The estimate \eqref{FracRegdotodd} implies that
\begin{equation*}
    \norm{A_q^{k+\frac{3}{2}}v_q(t)}_{L^2} \lesssim t^{-\beta}(1+\norm{q}_{H^{2k+1}}^{2k+\frac{1}{2}})(1+\norm{u_0}_{H^{2k+1}} + \norm{f}_{H^{2k+1}}^{k+\frac{1}{2}} + \norm{f}_{\infty}).
\end{equation*}
for $k\in \mathbb{N}\cup\{0\}$ such that $0 \leq k\leq l$, which means $v_q \in \Dot{H}^{2l+3}_{q}$. Using \eqref{FracReg6} when $\alpha = 0$, we deduce that
\begin{equation*}
    \begin{aligned}
        \norm{A_q^{\frac{1}{2}}v_q}_{L^2}^2 = & \pdt{{\color{black}v_q}}{A_qv_q}_{L^2} =-\pdt{v_q}{\Delta v_q}_{L^2} + \pdt{v_q}{qv_q}_{L^2} \\
        = &\pdt{\nabla v_q}{\nabla v_q}_{L^2} + \pdt{v_q}{qv_q}_{L^2} \\
        \leq &\pdt{\nabla v_q}{\nabla v_q}_{L^2} + M_0\pdt{v_q}{v_q}_{L^2} \\
        \lesssim& \norm{v_q}_{H^1}^2 \lesssim t^{-2\beta}(1+\norm{u_0}_{L^2} + \norm{f}_{\infty})^2.
    \end{aligned}
\end{equation*}
With these estimates of the $L^2$ norm for $A_q^{k+\frac{3}{2}} v_q$ and $A_q^{\frac{1}{2}}v_q$, inequality \eqref{Fracreg5} implies that
\begin{equation*}
    \begin{aligned}
        \norm{v_q}_{H^{\alpha+2}} \lesssim &(1+\norm{q}_{H^{2l+1}}^{l+1})\sum_{j=0}^{l+1}\norm{A_q^{j+\frac{1}{2}}{\color{black}v_q}}_{L^2}\\ \lesssim &t^{-\beta}(1+\norm{q}_{H^{\alpha}}^{3\alpha/2})(1+\norm{u_0}_{H^{\alpha}} + \norm{f}_{H^{\alpha}}^{\alpha/2} + \norm{f}_{\infty}).
    \end{aligned}
\end{equation*}
\par
The above proof demonstrates that for any integer \( \alpha \geq 0 \), we have
\[\|v_q\|_{H^{\alpha+2}} \lesssim t^{-\beta} \left(1 + \|q\|_{H^{\alpha}}^{3\alpha/2}\right) \left(1 + \|u_0\|_{H^{\alpha}} + \|f\|_{H^{\alpha}}^{\alpha/2} + \|f\|_{\infty}\right).\]
Using the regularity estimate for \( h_q \) and the triangle inequality, we ultimately obtain
\[
\begin{aligned}
	\|u_q(t)\|_{H^{\alpha+2}} &\leq \|v_q(t)\|_{H^{\alpha+2}} + \|h_q\|_{H^{\alpha+2}} \\
	&\lesssim \left(1 + t^{-\beta}\right) \left(1 + \|q\|_{H^{\alpha}}^{1+2\alpha}\right) \left(1 + \|u_0\|_{H^{\alpha}} + \|f\|_{H^{\alpha}}^{1+\alpha/2} + \|f\|_{\infty}\right).
\end{aligned}
\]
\end{proof}
\begin{lemma}\label{FracStab}
Let $u_0\in H^{\alpha}(\Omega)$, $f \in H^{\alpha}(\Omega)$ for $\alpha \in \mathbb{N}$ with $u_0, f \geq L_0$ a.e. and $a_0, a_1 \geq L_0$ for $L_0>0$. Let $q_1, q_2 \in H^{\alpha}(\Omega)\cap\mathcal{I}$, $\beta \in [\beta^-,\beta^+]$. Then, for a sufficiently large fixed $T_0 >0$ and any $t \geq T_0$, there exists $c>0$ independent of $q_1, q_2, T_0, \beta$ such that
\begin{equation*}
    \norm{q_1 - q_2}_{L^2(\Omega)}\leq c(1+\norm{q_1}_{H^{\alpha}(\Omega)}^{1+2\alpha}\vee\norm{q_2}_{H^{\alpha}(\Omega)}^{1+2\alpha})^{\frac{2}{\alpha+2}}\norm{u_{q_1}(t)-u_{q_2}(t)}_{L^2(\Omega)}^{\frac{\alpha}{\alpha+2}}.
\end{equation*}
\end{lemma}
\begin{proof}
According to Theorem 4.5 in \cite{jin2023inverse}, for $T_0>0$ sufficiently large and any $t \geq T_0$, there exists a constant $c$ such that
\begin{equation}\label{theorem4.5}
    \begin{aligned}
    \norm{q_1 - q_2}_{L^2(\Omega)} \leq c\norm{u_{q_1}(t)-u_{q_2}(t)}_{H^2(\Omega)}.
    \end{aligned}
\end{equation}
Next, using Sobolev interpolation inequality and Lemma \ref{FracReg}, we have
\[
\begin{aligned}
    \norm{u_{q_1}(t)-u_{q_2}(t)}_{H^2(\Omega)} \leq& \norm{u_{q_1}(t)-u_{q_2}(t)}_{H^{\alpha +2}(\Omega)}^{\frac{2}{\alpha+2}}\norm{u_{q_1}(t)-u_{q_2}(t)}_{L^2(\Omega)}^{\frac{\alpha}{\alpha+2}}\\
    \lesssim& (1+\norm{q_1}_{H^{\alpha}(\Omega)}^{1+2\alpha}\vee\norm{q_2}_{H^{\alpha}(\Omega)}^{1+2\alpha})^{\frac{2}{\alpha+2}}\norm{u_{q_1}(t)-u_{q_2}(t)}_{L^2(\Omega)}^{\frac{\alpha}{\alpha+2}}.
\end{aligned}
\]
Combining the above inequality with (\ref{theorem4.5}), we have
\[
\begin{aligned}
    \norm{q_1 - q_2}_{L^2(\Omega)}\leq &c(1+\norm{q_1}_{H^{\alpha}(\Omega)}^{1+2\alpha}\vee\norm{q_2}_{H^{\alpha}(\Omega)}^{1+2\alpha})^{\frac{2}{\alpha+2}}\norm{u_{q_1}(t)-u_{q_2}(t)}_{L^2(\Omega)}^{\frac{\alpha}{\alpha+2}}.
\end{aligned}
\]
for any $t \geq T_0$.
\end{proof}
Using the regularity and stability estimates presented above, we now aim to prove Theorem \ref{mainthmFrac}.
\begin{proof}[Proof of Theorem \ref{mainthmFrac}]
We will verify the Conditions \ref{condreg} and \ref{condstab} for the forward map $\mathcal{G}$ as defined in (\ref{forwardmapFrac}) 
with $\mathcal{R} = H^{\alpha}$.
From Lemma \ref{FracReg}, we have
\[\mathop{\sup}_{\theta \in B_{\mathcal{R}}(M)}\Vert u_{q_{\theta}}(T) \Vert_{H^{2}} \leq c(1+T^{-\beta}).\]
This inequality, combined with the Sobolev embedding $H^2 \subset C$, imply that
\[\mathop{\sup}_{\theta \in B_{\mathcal{R}}(M)}\mathop{\sup}_{x \in \Omega}\vert\mathcal{G}(\theta)(x)\vert \leq c(1+T^{-\beta}).\]
Therefore, we have verified the condition \eqref{bound} with $p = 0$.

Theorem \ref{FracLip}, in conjunction with Lemma 29 in \cite{IntroNonLinear_nickl2020convergence} and Sobolev embedding theorem, implies that
\[\lVert \mathcal{G}(\theta_1)-\mathcal{G}(\theta_2)\rVert _{L^2} \leq c(1+T^{-\beta})(1+\norm{\theta_1}_{\mathcal{R}}^2\vee\norm{\theta_2}_{\mathcal{R}}^2)\lVert \theta_1-\theta_2\rVert _{(H_0^2(\Omega))^*}\]
for $\theta_1,\theta_2 \in \mathcal{R}$, which verifies condition (\ref{lip}) with ${\kappa}=2$, $l=2$.

For the condition \eqref{stab}, Theorem \ref{FracStab}, combined with Lemma 29 in \cite{IntroNonLinear_nickl2020convergence}, implies that
\[\Vert q_{\theta} - q_{0}\Vert_{L^2}^{\frac{2+\alpha}{\alpha}} \leq C M^{2+4\alpha}\Vert \mathcal{G}(\theta)-\mathcal{G}(\theta_0)\Vert_{L^2}\]
for $\theta \in B_{\mathcal{R}}(M)$.

Given our requirement on \( \alpha \), we have
\[\alpha + {\kappa} \geq \frac{3}{2} = \frac{d(l+1)}{2}.\]
In conclusion, using {\color{black}Theorem \ref{finalthmsv}} with conditions verified above, we obtain
\[P_{\theta_0}^{(N)}\hat{Q} \Vert q_{\theta} - q_{0}\Vert_{L^2}^{\frac{2+\alpha}{\alpha}\cdot\frac{2}{p+q+1}}\lesssim \varepsilon_N^{\frac{2}{p+q+1}}\log N\]
for $p = 0$, $q = 2+4\alpha$.
\end{proof}
\begin{proof}[Proof of Theorem \ref{Fracminmax}: ]
    Let $\hat{q} = 1$ and $\psi_{lr}$ be 1-dimensional, compactly supported, as least $(\alpha + 1)$-regular Daubechies wavelets in \eqref{Dbase} of the main text. Define a fixed interval $[a,b]\subset K \subset \Omega$. Due to the compact support of $\psi_{lr}$ (see \cite[Theorem 4.2.10]{gin2015mathematical}), there exists a constant $c_0$ independent of $l$, ensuring that for any $l\geq l_0$,  we can choose a set of indices $r$ noted as $R_l$ such that $n_l := \abs{R_l}=c_02^{l}$ and $\psi_{lr}, r \in R_l$ are supported in $[a,b]$. Here, $l_0 \in \mathbb{N}$ is a large enough fixed integer that depends only on $\alpha, a, b$. For convenience, we denote $(\psi_{lr})_{r\in R_l}$ by $(\varphi_{lr})_{r=1}^{n_l}$ in the followings. 
    For a sufficiently small constant $\delta >0$, we define
    \begin{equation*}
        q_m := q + \delta2^{-l(\alpha+1/2)}\sum_{r=1}^{n_l}\gamma_{mr}\varphi_{lr}, \qquad m=1,\dots,M,
    \end{equation*}
    where $\gamma_m := (\gamma_{mr})_{r=1}^{n_l} \in \{-1,1\}^{n_l}, m=1,\dots,M$ will be chosen later.
    \par
    Then, we choose $\delta$ small enough as follows. By wavelet characterization of Sobolev norms in \cite[Chapter 4.3]{gin2015mathematical},
    we have
    \[\norm{q_m-\hat{q}}_{H^{\alpha}(\Omega)} = \delta\sqrt{2^{-l}\sum_{r=1}^{n_l}\gamma_{mr}^2} = \delta\sqrt{c_0}.\]
    Thus, $\norm{q_m-\hat{q}}_{H^{\alpha}(\Omega)}$ can be made as small as desired with $\delta$ small enough.
    By Sobolev embedding $H^{\alpha}\subset C$, we can choose suitable $\delta$ such that
    \[ 0 < q_m < M_0,\quad \norm{q_m}_{H^{\alpha}(\Omega)} < R.\]
    which combined with definition of $q_m$ further imply that $q_m \in \mathcal{F}_{\alpha,M_0}(R)$.
    \par
    For any $q\in \mathcal{F}_{\alpha,M_0}$ and measurable sets $A,B$ belonging to $(\mathbb{R}\times\Omega)^N, L^2(\Omega)$ respectively, we define the probability measures $\hat{Q}_{q}$ on $(\mathbb{R}\times\Omega)^N\times L^2(\Omega)$ as
    \begin{equation*}
        \hat{Q}_{q}(A\times B) = \int_{A}\hat{Q}(B\mid D_N) dP^{(N)}_q,
    \end{equation*}
    where $P^{(N)}_q$ denotes $P^{(N)}_{\theta}$ defined in \eqref{modeldensity} of the main text for $\Phi^{-1}(q) = \theta$ (in slight abuse of notations).
    Note that all $\hat{Q}_{q}$ are mutually absolutely continuous because 
    \[ 
       \hat{Q}_{q_1}(A\times B) = \int_{A} \hat{Q}(B) dP^{(N)}_{q_1} = \int_{A} \hat{Q}(B) \frac{dP^{(N)}_{q_1}}{dP^{(N)}_{q_2}}dP^{(N)}_{q_2} = \int_{A\times B} \frac{dP^{(N)}_{q_1}}{dP^{(N)}_{q_2}}d\hat{Q}_{q_2}, 
    \]
    that is,
    \[\frac{d\hat{Q}_{q_1}}{d\hat{Q}_{q_2}} = \frac{dP^{(N)}_{q_1}}{dP^{(N)}_{q_2}},\qquad \forall q_1,q_2 \in \mathcal{F}_{\alpha,M_0}.\]
    \par
    Next, we apply \cite[Theorem 6.3.2]{gin2015mathematical} to prove the conclusion, where two steps are needed : Choose a list of $\gamma_m$ such that the $L^2$ distance between $u_{q_m}(T)$ and $q_m$ has an appropriate lower bound respectively, and give an appropriate upper bound on the KL-divegence of the laws $\hat{Q}_{q_m},\hat{Q}_{\hat{q}}$. We give the prove of two steps respectively in below i) ,ii).
    \par
    i) By \eqref{Fracnormestimate**}, for all $v\in L^2(\Omega)$ and $q\in\mathcal{I}$, there exists a constant $c>0$ independent of $q$ such that
    \begin{equation}\label{minmax1}
            \norm{A_qv}_{(H^2_0(\Omega))^*} = \sup_{\norm{\phi}_{H^2_0}\leq 1}\pdt{A_qv}{\phi} \leq \sup_{\norm{A_q\phi}_{L^2}\leq c}\pdt{v}{A_q\phi}\leq c \norm{v}_{L^2}
    \end{equation}
    For all $m,m'\in\{1,\dots,M\}$, by triangle inequality, we have
    \begin{equation*}
        \begin{aligned}
            \norm{A_{q_{m'}}(u_{q_m}-u_{q_{m'}})(t)}_{(H^2_0(\Omega))^*} \geq & \norm{(q_{m'}-q_{m})u_{q_{m}}(t)}_{(H^2_0(\Omega))^*}\\&- \norm{\partial_t^{\beta}(u_{q_{m}}-u_{q_{m'}})(t)}_{(H^2_0(\Omega))^*}\\
            := &I_1 - I_2
        \end{aligned}
    \end{equation*}
    Because $u_0, f \geq L_0$ a.e. and $a_0, a_1 \geq L_0$, the maximum principle of time-fractional diffusion implies that $u(t) \geq L_0$.
    Therefore, we deduce that
    \begin{equation*}
        \begin{aligned}
            I_1 = &\sup_{\norm{\phi}_{H^2_0}\leq1}\pdt{(q_{m'}-q_{m})u_{q_{m}}(t)}{\phi}_{L^2} \\
            \geq & \frac{\pdt{(q_{m'}-q_{m})u_{q_{m}}(t)}{q_{m'}-q_{m}}_{L^2}}{\norm{q_{m'}-q_{m}}_{H^2_0}}\\
            \geq & L_0 \frac{\pdt{q_{m'}-q_{m}}{q_{m'}-q_{m}}_{L^2}}{\norm{q_{m'}-q_{m}}_{L^2}} \\
            \geq & L_0 \norm{q_{m'}-q_{m}}_{L^2} \geq L_0 \norm{q_{m'}-q_{m}}_{(H^2_0)^*}
        \end{aligned}
    \end{equation*}
    Lemma \ref{FracderivativeReg} implies that for fixed large enough $T_0>0$ and all $t>T_0$, there exists $c>0$ independent of $m,m',T_0$ such that
    \begin{equation*}
        \begin{aligned}
           I_2 \leq ct^{-\beta}\norm{q_{m'}-q_{m}}_{(H^2_0(\Omega))^*}
        \end{aligned}
    \end{equation*}
    Using the bound of $I_1, I_2$, we have
    \begin{equation*}
        \begin{aligned}
            \norm{A_{q_{m'}}(u_{q_m}-u_{q_{m'}})(t)}_{(H^2_0(\Omega))^*} \geq &(L_0-ct^{-\beta})\norm{q_{m'}-q_{m}}_{(H^2_0(\Omega))^*} \\
            \geq & \frac{L_0}{2}\norm{q_{m'}-q_{m}}_{H^{-2}(\Omega)}
        \end{aligned}
    \end{equation*}
    for any $t\geq T_0$ where $T_0 >0$ is chosen large enough such that $L_0-cT_0^{-\beta}> L_0/2$.
    This inequality, wavelet characterization of Sobolev norms and \eqref{minmax1} imply that
    \begin{equation}\label{minmaxlowB}
        \begin{aligned}
            \norm{u_{q_m}(t)-u_{q_{m'}}(t)}_{L^2(\Omega)} \geq & \norm{A_{q_{m'}}(u_{q_m}-u_{q_{m'}})(t)}_{(H^2_0(\Omega))^*}\\
            \geq & \frac{L_0}{2}\norm{q_{m'}-q_{m}}_{H^{-2}(\Omega)}\\
            \gtrsim & \sqrt{2^{-2l(\alpha + 2 + 1/2)}\sum_{r=1}^{n_l}\abs{\gamma_{mr}-\gamma_{m'r}}^2}.
        \end{aligned}
    \end{equation}
     By the Varshamov-Gilbert-bound \cite[Example 3.1.4]{gin2015mathematical}, for constants $c_1,c_2>0$ independent of $l$, there exists a subset $\mathcal{M}_l\subset \{-1,1\}^{c_02^{l}}$ with $M_l := \abs{\mathcal{M}_l} = 2^{c_12^{l}}$ such that 
     \begin{equation*}
         \sum_{r=1}^{n_l}\abs{\gamma_{mr}-\gamma_{m'r}}^2 \geq c_22^l
     \end{equation*}
     for all $m,m'\in \mathcal{M}_l$ when $m\neq m'$. Combined with \eqref{minmaxlowB}, we further have
     \begin{equation*}
         \norm{u_{q_m}(T)-u_{q_{m'}}(T)}_{L^2(\Omega)} \gtrsim 2^{-l(\alpha + 2)}
     \end{equation*}
     for $T\geq T_0$.
     \par
     For all $m,m'\in \mathcal{M}_l$ when $m\neq m'$, we also deduce that
     \[ \norm{q_{m'}-q_{m}}_{L^2(\Omega)} = \sqrt{2^{-2l(\alpha + 1/2)}\sum_{r=1}^{n_l}\abs{\gamma_{mr}-\gamma_{m'r}}^2} \gtrsim 2^{-l\alpha}.\]
     ii) We see 
     \begin{equation*}
         \begin{aligned}
             D(\hat{Q}_{q_m}\Vert \hat{Q}_{\hat{q}}) = & \int_{(\mathbb{R}\times \Omega)^N \times L^2(\Omega)} \log \Big(\frac{d\hat{Q}_{q_m}}{d\hat{Q}_{\hat{q}}}\Big) \frac{d\hat{Q}_{q_m}}{d\hat{Q}_{\hat{q}}} d\hat{Q}_{\hat{q}}\\
             = & \int_{(\mathbb{R}\times \Omega)^N \times L^2(\Omega)} \log \Big(\frac{dP^{(N)}_{q_m}}{dP^{(N)}_{\hat{q}}}\Big) \frac{dP^{(N)}_{q_m}}{dP^{(N)}_{\hat{q}}} d\hat{Q}_{\hat{q}} \\
             = & \int_{(\mathbb{R}\times \Omega)^N}\log \Big(\frac{dP^{(N)}_{q_m}}{dP^{(N)}_{\hat{q}}}\Big) \frac{dP^{(N)}_{q_m}}{dP^{(N)}_{\hat{q}}} dP^{(N)}_{\hat{q}}
             = D(P^{(N)}_{q_m}\Vert P^{(N)}_{\hat{q}}).
         \end{aligned}
     \end{equation*}
     From Lemma \ref{FracLip} and Proposition \ref{le2.1}, we further deduce that
     \begin{equation}\label{minmax3}
         \begin{aligned}
             D(\hat{Q}_{q_m}\Vert \hat{Q}_{\hat{q}}) = &\frac{N}{2}\norm{u_{q_m}(T)-u_{\hat{q}}(T)}^2_{L_{\lambda}^2(\Omega)}\\
             \leq & cN\norm{q_{m}-\hat{q}}^2_{(H^2_0(\Omega))^*} \\
         \end{aligned}
     \end{equation}
     By our definition of $q_{m}$, we see $q_m-\hat{q}$ is supported in $[a,b]$. For compact set $[a,b]$, we can have a compact set $K'$ such that $[a,b] \subsetneq K' \subset \Omega$, and a cut-off function $\chi \in C_c^{\infty}(\Omega)$ such that $\chi = 1$ on $K'$. By \eqref{Sobolevinter2} in the main text, for any $v\in H^2(\Omega)$, we have
     \[\norm{\chi v}_{H^2_c(\Omega)}\lesssim \norm{v}_{H^2(\Omega)}.\]
     By this inequality, we have
     \begin{equation*}
         \begin{aligned}
             \norm{q_{m}-\hat{q}}_{(H^2_0(\Omega))^*} =& \sup_{\norm{\phi}_{H^2_0}\leq 1}\int_{\Omega}(q_{m}-\hat{q})\phi dx \\
             = & \sup_{\norm{\phi}_{H^2_0}\leq 1}\int_{\Omega}(q_{m}-\hat{q})\chi\phi dx \\
             \lesssim & \sup_{\norm{\phi}_{H^2_c}\leq 1}\int_{\Omega}(q_{m}-\hat{q})\phi dx\\
             \lesssim &\norm{q_{m}-\hat{q}}_{H^{-2}(\Omega)}
         \end{aligned}
     \end{equation*}
     Using this inequality, \eqref{minmax3} and wavelet characterization of Sobolev norms, we have
     \begin{equation*}
         \begin{aligned}
             D(\hat{Q}_{q_m}\Vert \hat{Q}_{\hat{q}}) \lesssim& N\norm{q_{m}-\hat{q}}^2_{H^{-2}(\Omega)}\\
             \lesssim &\delta N 2^{-2l(\alpha + 2 + 1/2)}\sum_{r=1}^{n_l}\abs{\gamma_{mr}}^2\\
             \lesssim &\delta N 2^{-2l(\alpha + 2)}
         \end{aligned}
     \end{equation*}
     \par
     Let $2^l \simeq N^{\frac{1}{2\alpha+4+1}}$. From the proof in i), ii), we have
     \begin{equation*}
         \norm{u_{q_m}(T)-u_{q_{m'}}(T)}^{\eta_1}_{L^2(\Omega)} \gtrsim 2^{-l(\alpha + 2)\eta_1}\gtrsim N^{-\frac{\alpha+2}{2\alpha+4+1}\cdot\eta_1},
     \end{equation*}
     and
     \begin{equation*}
         \begin{aligned}
             D(\hat{Q}_{q_m}\Vert \hat{Q}_{\hat{q}}) \lesssim & \delta N 2^{-2l(\alpha + 2)} \lesssim \delta N^{\frac{1}{2\alpha+4+1}} \lesssim \delta \log M_l
         \end{aligned}
     \end{equation*}
     for all $m,m'\in \mathcal{M}_l$ and $m\neq m'$. Then, apply \cite[Theorem 6.3.2]{gin2015mathematical} and we have
     \begin{equation*}
         \inf_{\Tilde{u}_N}\sup_{q_0\in\mathcal{F}_{\alpha,M_0}(R)}\hat{Q}_{q_0}\norm{\Tilde{u}_N-u_{q_0}(T)}^{\eta_1}_{L^2(\Omega)}\gtrsim N^{-\frac{\alpha+2}{2\alpha+4+1}\cdot\eta_1},
     \end{equation*}
     that is,
     \begin{equation*}
         \inf_{\Tilde{u}_N}\sup_{q_0\in\mathcal{F}_{\alpha,M_0}(R)}P^{(N)}_{q_0}\hat{Q}\norm{\Tilde{u}_N-u_{q_0}(T)}^{\eta_1}_{L^2(\Omega)}\gtrsim N^{-\frac{\alpha+2}{2\alpha+4+1}\cdot\eta_1}
     \end{equation*}
     where the infimum ranges over all measurable functions $\Tilde{u}_N = \Tilde{u}_N((Y_i,X_i)_{i=1}^N,\theta)$ that take value in $L^2(\Omega)$ with $(Y_i,X_i)_{i=1}^N$ from $P^{(N)}_{q_0}$ and $\theta$ from $\hat{Q}$.
     Similarly, with
     \[\norm{q_{m'}-q_{m}}_{L^2(\Omega)}^{\frac{\alpha+2}{\alpha}\cdot \eta_2} \gtrsim 2^{-l(\alpha+2)\eta_2} \gtrsim N^{-\frac{\alpha+2}{2\alpha+4+1}\cdot \eta_2}\]
     for all $m,m'\in \mathcal{M}_l$ and $m\neq m'$, Theorem 6.3.2 in \cite{gin2015mathematical} also implies that
     \begin{equation*}
         \inf_{\Tilde{q}_N}\sup_{q_0\in\mathcal{F}_{\alpha,M_0}(R)}P^{(N)}_{q_0}\hat{Q}\norm{\Tilde{q}_N-q_0}^{\frac{\alpha+2}{\alpha}\cdot \eta_2}_{L^2(\Omega)}\gtrsim N^{-\frac{\alpha+2}{2\alpha+4+1}\cdot \eta_2}
     \end{equation*}
      where the infimum ranges over all measurable functions $\Tilde{q}_N = \Tilde{q}_N((Y_i,X_i)_{i=1}^N,\theta)$ that take value in $L^2(\Omega)$ with $(Y_i,X_i)_{i=1}^N$ from $P^{(N)}_{q_0}$ and $\theta$ from $\hat{Q}$. 
\end{proof}
\textbf{Proof of Theorem \ref{Fractionalfinal}:}  Next, we will prove Theorem \ref{Fractionalfinal} with the lemmas introduced below. With an unknown fractional order $\beta$, we will need additional regularity and stability estimates for the forward map $\mathcal{G}_{\beta}(\theta)$, as shown in \eqref{forwardmapFrac} of the main text, to verify Conditions \ref{condregmodel} and \ref{condstabmodel}.
\begin{lemma}\label{FracElip}
For any $\beta,\Tilde{\beta}\in [\beta^{-},\beta^{+}]$ and any nonnegative sequence $(\lambda_j)_{j=1}^{\infty}$, there exists a constant $C=C(\beta^{-},\beta^{+})$ such that
\[
\abs{E_{\beta,1}(-\lambda_jt^{\beta})-E_{\Tilde{\beta},1}(-\lambda_jt^{\Tilde{\beta}})} \leq C\lambda_j^{-1}(t^{-\beta}\abs{\beta - \Tilde{\beta}} +|t^{-\beta}-t^{-\Tilde{\beta}}| ),
\]
and
\[
\abs{E_{\beta,1}(-\lambda_jt^{\beta})-E_{\Tilde{\beta},1}(-\lambda_jt^{\Tilde{\beta}})} \leq C(1 + |\log t|)\abs{\beta - \Tilde{\beta}}.
\]
\end{lemma}
\begin{proof}
First, we divide left side of this inequality into two part:
\begin{equation*}
    \begin{aligned}
        &\abs{E_{\beta,1}(-\lambda_jt^{\beta})-E_{\Tilde{\beta},1}(-\lambda_jt^{\Tilde{\beta}})}\\
        \leq & \abs{E_{\beta,1}(-\lambda_jt^{\beta})-E_{\Tilde{\beta},1}(-\lambda_jt^{\beta})} + \abs{E_{\Tilde{\beta},1}(-\lambda_jt^{\beta})-E_{\Tilde{\beta},1}(-\lambda_jt^{\Tilde{\beta}})}\\
        = & I_1 + I_2
    \end{aligned}
\end{equation*}
For the estimate of $I_1$, by part (a) of Lemma 2.3 in \cite{Frac_dang2018continuity}, there exists a constant $C(\beta^-,\beta^+)>0$ that depends only on $\beta^-,\beta^+$ such that
\begin{equation*}
    \abs{\partial_\zeta E_{\zeta,1}(-\lambda_j t^{\beta})} \leq \frac{C(\beta^-,\beta^+)}{1+\lambda_jt^{\beta}}
\end{equation*}
for any $\zeta\in [\beta^-,\beta^+]$ and $t>0$. Then,
\begin{equation*}
    \begin{aligned}
        I_1 = &\abs{E_{\beta,1}(-\lambda_jt^{\beta})-E_{\Tilde{\beta},1}(-\lambda_jt^{\beta})} \\
          \leq&\abs{\int_{\beta}^{\Tilde{\beta}}\partial_\zeta E_{\zeta,1}(-\lambda_j t^{\beta})d\zeta}\\
          \leq&\frac{C(\beta^-,\beta^+)}{1+\lambda_jt^{\beta}}\abs{\beta - \Tilde{\beta}} \leq C(\beta^-,\beta^+)\min(1,\lambda_j^{-1}t^{-\beta})\abs{\beta - \Tilde{\beta}}.
    \end{aligned} 
\end{equation*}
Next, we provide estimates for $I_2$. From (3.21) in \cite{jin2021fractional}, we have
\[\partial_tE_{\beta,1}(-\lambda_jt^{\beta}) = -\lambda_jt^{\beta-1}E_{\beta,\beta}(-\lambda_jt^{\beta})\]
for any $\beta \in [\beta^-,\beta^+]$ and $t>0$. 
Using part (a) of Lemma 2.3 in \cite{Frac_dang2018continuity} again, combined with (3.21) in \cite{jin2021fractional}, there exists a constant $C(\beta^-,\beta^+)>0$ that depends only on $\beta^-,\beta^+$ such that
\begin{equation*}
    \abs{E_{\beta,\beta}(-\lambda_j t^{\beta})} \leq \frac{C(\beta^-,\beta^+)}{1+(\lambda_jt^{\beta})^2}
\end{equation*}
for any $\beta \in [\beta^-,\beta^+]$ and $t>0$.
Then,we deduce that
\begin{equation*}
    \begin{aligned}
        I_2 = &\abs{E_{\Tilde{\beta},1}(-\lambda_jt^{\beta})-E_{\Tilde{\beta},1}(-\lambda_jt^{\Tilde{\beta}})} \\
        =&\abs{\int_{t}^{t^{\beta/{\Tilde{\beta}}}} \partial_rE_{\Tilde{\beta},1}(-\lambda_jr^{\Tilde{\beta}}) dr} \\
        \leq& C(\beta^-,\beta^+) \lambda_j^{-1}\abs{\int_{t}^{t^{\beta/{\Tilde{\beta}}}} r^{-\beta-1}\frac{(\lambda_jr^{\beta})^2}{1+(\lambda_jr^{\beta})^2} dr}\\
        \leq& C(\beta^-,\beta^+) \lambda_j^{-1}\abs{\int_{t}^{t^{\beta/{\Tilde{\beta}}}} r^{-\beta-1}dr}\\
        \leq& C(\beta^-,\beta^+) \lambda_j^{-1}\abs{t^{-\beta}-t^{-\Tilde{\beta}}}.
    \end{aligned}
\end{equation*}
And we can also have
\begin{equation*}
    \begin{aligned}
        I_2 \leq &C(\beta^-,\beta^+) \abs{\int_{t}^{t^{\beta/{\Tilde{\beta}}}} r^{-1}\frac{\lambda_jr^{\beta}}{1+(\lambda_jr^{\beta})^2} dr} \\
        =&C(\beta^-,\beta^+)\abs{\log(t)}\abs{\beta-\Tilde{\beta}},
    \end{aligned}
\end{equation*}
where we used the fact that $\frac{\lambda_jr^{\beta}}{1+(\lambda_jr^{\beta})^2} \leq 1$.
\end{proof}
\begin{lemma}\label{FractionalDerivativeReg}
Let $q,\Tilde{q} \in \mathcal{I}$ and let $\beta, \Tilde{\beta} \in [\beta^{-},\beta^{+}]$ for $\beta^{-}>0$. Then, for a fixed and sufficiently large $T_0 >0$ and any $t \geq T_0$, there exists a constant $c>0$ independent of $q,\Tilde{q},\beta,\Tilde{\beta},T_0$ such that
\begin{equation*}
    \norm{\partial_t^{\beta}u_{q,\beta}(t) - \partial_t^{\Tilde{\beta}}u_{\Tilde{q},\Tilde{\beta}}(t)}_{L^2(\Omega)} \leq c(t^{-\beta}\norm{q-\Tilde{q}}_{(H^2_0(\Omega))^*} + (1+\abs{\log t})\abs{\beta - \Tilde{\beta}})
\end{equation*}
and
\begin{equation*}
    \norm{\partial_t^{\beta}u_{q,\beta}(t) - \partial_t^{\Tilde{\beta}}u_{\Tilde{q},\Tilde{\beta}}(t)}_{L^2(\Omega)} \leq c(t^{-\beta}\norm{q-\Tilde{q}}_{(H^2_0(\Omega))^*} + t^{-\beta}\abs{\beta - \Tilde{\beta}} +|t^{-\beta}-t^{-\Tilde{\beta}}|).
\end{equation*}
\end{lemma}
\begin{proof}
By Lemma \ref{FracderivativeReg}, we have 
\[\norm{\partial_t^{\beta}u_{q,\beta}-\partial_t^{\beta}u_{\Tilde{q},\beta}}_{L^2(\Omega)} \leq ct^{-\beta}\norm{q-\Tilde{q}}_{(H^2_0(\Omega))^*}.\]
With notation $v_q := u_0 - h_q$, we have
\begin{equation*}
    \begin{aligned}
        \norm{\partial_t^{\beta}u_{\Tilde{q},\beta}-\partial_t^{\Tilde{\beta}}u_{\Tilde{q},\Tilde{\beta}}}_{L^2(\Omega)} =& \norm{A_{\Tilde{q}}(u_{\Tilde{q},\beta}-u_{\Tilde{q},\Tilde{\beta}})}_{L^2(\Omega)}\\
        =& \norm{A_{\Tilde{q}}(F_{\Tilde{q},\beta}(t)-F_{\Tilde{q},\Tilde{\beta}}(t))v_{\Tilde{q}}}_{L^2(\Omega)}.
    \end{aligned}
\end{equation*}
Using Lemma \ref{FracElip}, we deduce that
\begin{equation*}
    \begin{aligned}
        \norm{\partial_t^{\beta}u_{\Tilde{q},\beta}-\partial_t^{\Tilde{\beta}}u_{\Tilde{q},\Tilde{\beta}}}_{L^2(\Omega)} =& \norm{\sum_{j=1}^{\infty}\lambda_j\bbra{E_{\beta,1}(-\lambda_jt^{\beta})-E_{\Tilde{\beta},1}(-\lambda_jt^{\Tilde{\beta}})}(v_{\Tilde{q}},\varphi_j)\varphi_j}_{L^2(\Omega)}\\
        =&\sqrt{\sum_{j=1}^{\infty}\lambda_j^2\bbra{E_{\beta,1}(-\lambda_jt^{\beta})-E_{\Tilde{\beta},1}(-\lambda_jt^{\Tilde{\beta}})}^2(v_{\Tilde{q}},\varphi_j)^2}\\
        \leq &c\abs{\beta-\Tilde{\beta}}(1+\abs{\log t})\sqrt{\sum_{j=1}^{\infty}\lambda_j^2(v_{\Tilde{q}},\varphi_j)^2}\\
        \leq &c\abs{\beta-\Tilde{\beta}}(1+\abs{\log t})\norm{v_{\Tilde{q}}}_{H^2},
    \end{aligned}
\end{equation*}
and also
\begin{equation*}
    \begin{aligned}
        \norm{\partial_t^{\beta}u_{\Tilde{q},\beta}-\partial_t^{\Tilde{\beta}}u_{\Tilde{q},\Tilde{\beta}}}_{L^2(\Omega)} =&\sqrt{\sum_{j=1}^{\infty}\lambda_j^2\bbra{E_{\beta,1}(-\lambda_jt^{\beta})-E_{\Tilde{\beta},1}(-\lambda_jt^{\Tilde{\beta}})}^2(v_{\Tilde{q}},\varphi_j)^2}\\
        \leq &c(t^{-\beta}\abs{\beta - \Tilde{\beta}} +|t^{-\beta}-t^{-\Tilde{\beta}}|)\sqrt{\sum_{j=1}^{\infty}(v_{\Tilde{q}},\varphi_j)^2}\\
        \leq &c(t^{-\beta}\abs{\beta - \Tilde{\beta}} +|t^{-\beta}-t^{-\Tilde{\beta}}|)\norm{v_{\Tilde{q}}}_{H^2}.
    \end{aligned}
\end{equation*}
Thus, we estimate the norm $\norm{v_{\tilde{q}}}_{H^2}$ to complete the proof:
\begin{equation*}
\norm{v_{\Tilde{q}}}_{H^2} \leq \norm{u_0}_{H^2} + \norm{h_{\Tilde{q}}}_{H^2} \lesssim \norm{u_0}_{H^2} + \norm{f}_{L^2}.
\end{equation*}
\end{proof}
\begin{lemma}\label{FractionalReg}
Let $q,\Tilde{q} \in \mathcal{I}$ and let $\beta, \Tilde{\beta} \in [\beta^{-},\beta^{+}]$. Then, for a fixed and sufficiently large $T_0 >0$ and any $t \geq T_0$, there exists a constant $c>0$ independent of $q,\Tilde{q},\beta,\Tilde{\beta},T_0$ such that
\begin{equation*}
    \norm{u_{q,\beta}(t) - u_{\Tilde{q},\Tilde{\beta}}(t)}_{L^2(\Omega)} \leq c((1+t^{-\beta})\norm{q-\Tilde{q}}_{(H^2_0(\Omega))^*} + (1+\abs{\log t})\abs{\beta - \Tilde{\beta}}).
\end{equation*}
\end{lemma}
\begin{proof}
Because \[\partial_t^{\beta}u_{q,\beta} - \Delta u_{q,\beta}  + qu_{q,\beta}  = f,\]
we have \[ - \Delta (u_{q,\beta}-u_{\Tilde{q},\Tilde{\beta}})  + \Tilde{q}(u_{q,\beta}-u_{\Tilde{q},\Tilde{\beta}})  = (\Tilde{q}-q)u_{q,\beta}-(\partial_t^{\beta}u_{q,\beta}-\partial_t^{\Tilde{\beta}}u_{\Tilde{q},\Tilde{\beta}}).\]
Thus,
\begin{equation*}
    \begin{aligned}
\norm{A_{\Tilde{q}}(u_{q,\beta}-u_{\Tilde{q},\Tilde{\beta}})}_{(H^2_0(\Omega))^*} 
        \leq \norm{(\Tilde{q}-q)u_{q,\beta}}_{(H^2_0(\Omega))^*} + \norm{\partial_t^{\beta}u_{q,\beta}-\partial_t^{\Tilde{\beta}}u_{\Tilde{q},\Tilde{\beta}}}_{(H^2_0(\Omega))^*}.
    \end{aligned}
\end{equation*}
Lemma 4.3 in \cite{jin2023inverse} and \eqref{Sobolevinter1} in the main text imply that 
\begin{equation*}
    \begin{aligned}
        \norm{(\Tilde{q}-q)u_{q,\beta}}_{(H^2_0(\Omega))^*} \leq& \norm{u_{q,\beta}}_{H^2(\Omega)}\norm{\Tilde{q}-q}_{(H^2_0(\Omega))^*}\\
        \leq&c(1+t^{-\beta})\norm{\Tilde{q}-q}_{(H^2_0(\Omega))^*}.
    \end{aligned}
\end{equation*}
Combined with the first inequality in Lemma \ref{FractionalDerivativeReg} and \eqref{Fraclip1}, we have 
\begin{equation*}
    \norm{u_{q,\beta}(t) - u_{\Tilde{q},\Tilde{\beta}}(t)}_{L^2(\Omega)} \leq c((1+t^{-\beta})\norm{q-\Tilde{q}}_{(H^2_0(\Omega))^*} + (1+\abs{\log t})\abs{\beta - \Tilde{\beta}}).
\end{equation*}
\end{proof}
\begin{lemma}\label{FractionalStab}
Let $u_0 \in H^{\alpha}(\Omega)$, $f \in H^{\alpha}(\Omega)$ for $\alpha \in \mathbb{N}$ with $u_0, f \geq L_0$ a.e. and let $a_0, a_1 \geq L_0$ for $L_0>0$. Let $q,\Tilde{q} \in H^{\alpha}(\Omega)\cap\mathcal{I}$ and let $\beta, \Tilde{\beta} \in [\beta^{-},\beta^{+}]$. Then, for a fixed and sufficiently large $T_0>0$, there exists a constant $c>0$ independent of $q,\Tilde{q},\beta,\Tilde{\beta},T_0$ such that
\begin{equation*}
    \begin{aligned}
    \norm{q - \Tilde{q}}_{L^2(\Omega)}\leq &c\Big((1+\norm{q}_{H^{\alpha}(\Omega)}^{1+2\alpha}\vee\norm{\Tilde{q}}_{H^{\alpha}(\Omega)}^{1+2\alpha})^{\frac{2}{\alpha+2}}\norm{u_{q,\beta}(t)-u_{\Tilde{q},\Tilde{\beta}}(t)}_{L^2(\Omega)}^{\frac{\alpha}{\alpha+2}}\\
    &+t^{-\beta}\abs{\beta - \Tilde{\beta}} +|t^{-\beta}-t^{-\Tilde{\beta}}|\Big)
\end{aligned}
\end{equation*}
for any $t \geq T_0$.
\end{lemma}
\begin{proof}
Equation (\ref{Fractional}) implies that
\[ u_{q,\beta}(t) q = {f -\partial_t^{\beta}u_{q,\beta}(t) + \Delta u_{q,\beta}(t)}.\]
Then we have
\[ \begin{aligned}
    &u_{q,\beta}(t)u_{\Tilde{q},\Tilde{\beta}}(t)(q - \Tilde{q}) \\
    =&  f\bbra{u_{\Tilde{q},\Tilde{\beta}}(t)-u_{q,\beta}(t)} + \bbra{u_{q,\beta}(t)\partial_t^{\Tilde{\beta}}u_{\Tilde{q},\Tilde{\beta}}(t)-u_{\Tilde{q},\Tilde{\beta}}(t)\partial_t^{\beta}u_{q,\beta}(t)}\\
   &+ \bbra{u_{\Tilde{q},\Tilde{\beta}}(t)\Delta u_{q,\beta}(t)-u_{q,\beta}(t)\Delta u_{\Tilde{q},\Tilde{\beta}}(t)}\\
    =& \sum_{i=1}^{3} I_i.
\end{aligned} \]
For $I_1$, we have 
\[ \norm{I_1}_{L^2(\Omega)} \leq \norm{f}_{L^2(\Omega)} \norm{(u_{\Tilde{q},\Tilde{\beta}}-u_{q,\beta})(t)}_{H^2(\Omega)}.\]
By Lemma 3.1 in \cite{zhang2022identification} and Lemma 4.3 in \cite{jin2023inverse}, we have
\begin{equation}\label{FractionalStabL1}
    \norm{\partial_t^{\beta}u_{q,\beta}(t)}_{H^2(\Omega)} \leq ct^{-\beta}, \qquad \norm{u_{q,\beta}(t)}_{H^2(\Omega)} \leq c(1+t^{-\beta}).
\end{equation}
Those imply that
\[ 
\begin{aligned}
    \norm{I_3}_{L^2(\Omega)} \leq &c\left(\norm{u_{q,\beta}(t)}_{L^{\infty}(\Omega)}\norm{\Delta (u_{q,\beta}-u_{\Tilde{q},\Tilde{\beta}})(t)}_{L^2(\Omega)}\right. \\
    &+ \left.\norm{\Delta u_{q,\beta}(t)}_{L^2(\Omega)}\norm{(u_{q,\beta}-u_{\Tilde{q},\Tilde{\beta}})(t)}_{L^{\infty}(\Omega)}\right)\\
    \leq & c\norm{(u_{q,\beta}-u_{\Tilde{q},\Tilde{\beta}}(t))}_{H^2(\Omega)}.
\end{aligned}\]
By \eqref{FractionalStabL1} and the second inequality in Lemma \ref{FractionalDerivativeReg}, we deduce that
\[ 
\begin{aligned}
    \norm{I_2}_{L^2(\Omega)} \leq &c\left(\norm{u_{q,\beta}(t)}_{L^{\infty}(\Omega)}\norm{\partial_t^{\beta}u_{q,\beta}(t)-\partial_t^{\Tilde{\beta}}u_{\Tilde{q},\Tilde{\beta}}(t)}_{L^2(\Omega)}\right. \\
    &\left.+ \norm{\partial_t^{\beta}u_{q,\beta}(t)}_{L^2(\Omega)}\norm{u_{q,\beta}(t)-u_{\Tilde{q},\Tilde{\beta}}(t)}_{L^{\infty}(\Omega)}\right)\\
    \leq & c\bbra{\norm{u_{q,\beta}(t)-u_{\Tilde{q},\Tilde{\beta}}(t)}_{H^2(\Omega)} + t^{-\beta}\norm{q-\Tilde{q}}_{L^2(\Omega)} + t^{-\beta}\abs{\beta - \Tilde{\beta}} +|t^{-\beta}-t^{-\Tilde{\beta}}|}.
\end{aligned}\]
Because $u_0, f \geq L_0$ a.e. and $a_0, a_1 \geq L_0$, the maximum principle of time-fractional diffusion implies that $u(t) \geq L_0$. Thus, we have
\begin{equation}\label{FractionalStabL2}
    \begin{aligned}
    \norm{q - \Tilde{q}}_{L^2(\Omega)}\leq &L_0^{-2}\norm{u_{q,\beta}(t)u_{\Tilde{q},\Tilde{\beta}}(t)(q - \Tilde{q})}_{L^2(\Omega)} \\ 
    \leq & c\bigg(\norm{u_{q,\beta}(t)-u_{\Tilde{q},\Tilde{\beta}}(t)}_{H^2(\Omega)} + t^{-\beta}\norm{q-\Tilde{q}}_{L^2(\Omega)} \\ &+ t^{-\beta}\abs{\beta - \Tilde{\beta}} +|t^{-\beta}-t^{-\Tilde{\beta}}|\bigg).
    \end{aligned}
\end{equation}
Then, for $T_0>0$ sufficiently large such that $1 - cT_0^{-\beta^+} \geq 1/2 $, we have
\[
\begin{aligned}
\norm{q - \Tilde{q}}_{L^2(\Omega)} \leq c\bbra{\norm{u_{q,\beta}(t)-u_{\Tilde{q},\Tilde{\beta}}(t)}_{H^2(\Omega)}+t^{-\beta}\abs{\beta - \Tilde{\beta}} +|t^{-\beta}-t^{-\Tilde{\beta}}|}.
\end{aligned}
\]
for a constant $c>0$ independent of $q,\Tilde{q},\beta,\Tilde{\beta},T_0$ and any $t\geq T_0$.
Next, using Sobolev interpolation inequality and Lemma \ref{FracReg}, we have
\[
\begin{aligned}
    \norm{u_{q,\beta}(t)-u_{\Tilde{q},\Tilde{\beta}}(t)}_{H^2(\Omega)} \leq& \norm{u_{q,\beta}(t)-u_{\Tilde{q},\Tilde{\beta}}(t)}_{H^{\alpha +2}(\Omega)}^{\frac{2}{\alpha+2}}\norm{u_{q,\beta}(t)-u_{\Tilde{q},\Tilde{\beta}}(t)}_{L^2(\Omega)}^{\frac{\alpha}{\alpha+2}}\\
    \lesssim& (1+\norm{q}_{H^{\alpha}(\Omega)}^{1+2\alpha}\vee\norm{\Tilde{q}}_{H^{\alpha}(\Omega)}^{1+2\alpha})^{\frac{2}{\alpha+2}}\norm{u_{q,\beta}(t)-u_{\Tilde{q},\Tilde{\beta}}(t)}_{L^2(\Omega)}^{\frac{\alpha}{\alpha+2}}.
\end{aligned}
\]
Combining the above inequality with (\ref{FractionalStabL2}), we have
\[
\begin{aligned}
    \norm{q - \Tilde{q}}_{L^2(\Omega)}\leq &c\Big((1+\norm{q}_{H^{\alpha}(\Omega)}^{1+2\alpha}\vee\norm{\Tilde{q}}_{H^{\alpha}(\Omega)}^{1+2\alpha})^{\frac{2}{\alpha+2}}\norm{u_{q,\beta}(t)-u_{\Tilde{q},\Tilde{\beta}}(t)}_{L^2(\Omega)}^{\frac{\alpha}{\alpha+2}}\\
    &+t^{-\beta}\abs{\beta - \Tilde{\beta}} +|t^{-\beta}-t^{-\Tilde{\beta}}|\Big)
\end{aligned}
\]
for any $t \geq T_0$.
\end{proof}
Using the regularity and stability estimates mentioned above, we are now here to prove Theorem \ref{Fractionalfinal}.
\begin{proof}[Proof of Theorem \ref{Fractionalfinal}]
We verify the conditions in Theorem \ref{mainthmModel} for the forward map $\mathcal{G}_{\beta}$ as defined in (\ref{forwardmapFrac}) with $\mathcal{R} = H^{\alpha}$.
From \cite[lemma 4.3]{jin2023inverse}, we have
\[\mathop{\sup}_{\theta \in B_{\mathcal{R}}(M)}\Vert u_{q_{\theta}}(T) \Vert_{H^{2}} \leq c(1+T^{-\beta}).\]
The above inequality, combined with the Sobolev embedding $H^2 \subset C^0$, implies that
\[\mathop{\sup}_{\theta \in B_{\mathcal{R}}(M)}\mathop{\sup}_{x \in \Omega}\vert\mathcal{G}(\theta)(x)\vert \leq c(1+T^{-\beta}).\]
Therefore, we have condition (\ref{boundmodel}) verified with $p = 0$.
Theorem \ref{FractionalReg}  and \cite[Lemma 29]{IntroNonLinear_nickl2020convergence} imply that
\[\begin{aligned}
    \Vert \mathcal{G}_{\beta_1}(\theta_1)-\mathcal{G}_{\beta_2}(\theta_2)\Vert _{L^2} \leq & c\Big((1+T^{-\beta})(1 + \norm{\theta_1}^2_{\mathcal{R}}\vee\norm{\theta_2}^2_{\mathcal{R}})\norm{\theta_1-\theta_2}_{(H^2_0(\Omega))^*} \\&+ (1+\abs{\log T})\abs{\beta_1 - \beta_2}\Big)
\end{aligned}
\]
for $\theta_1,\theta_2 \in \mathcal{R}$ and $\beta_1,\beta_2 \in [\beta^-,\beta^+]$, which verifies condition (\ref{lipmodel}) with ${\kappa}=2$, $l=2$.
Given our requirement on $\alpha$, we have
\[
\alpha + {\kappa} \geq 5 \geq \frac{3}{2} = \frac{d(l+1)}{2}.
\]
\par
Then, by Theorem \ref{mainthmModel} and the proof of Theorem \ref{mainthm}, we have
\[P_{\beta_0,\theta_0}^{(N)}\hat{Q} \Vert \mathcal{G}_{\hat{\beta}}(\theta)-\mathcal{G}_{\beta_0}(\theta_0)\Vert_{L^2}^{2}\leq C_1 \varepsilon_N^{2}\log N,\]
and we also have
\begin{equation}\label{Fractionalfinal1}
    P_{\beta_0,\theta_0}^{(N)}\hat{Q} (1+\norm{\theta}_{H^2}^q)^{\frac{2}{q+1}}\Vert \mathcal{G}_{\hat{\beta}}(\theta)-\mathcal{G}_{\beta_0}(\theta_0)\Vert_{L^2}^{\frac{2}{q+1}}\leq C_2 \varepsilon_N^{\frac{2}{q+1}}\log N
\end{equation}
with $q=2+4\alpha$.
Using Lemma 29 in \cite{IntroNonLinear_nickl2020convergence} and Lemma \ref{FractionalStab},
we deduce that
\begin{equation*}
    \begin{aligned}
    \norm{q_\theta - q_0}_{L^2(\Omega)}^{\frac{\alpha+2}{\alpha}}\leq &c\Big((1+\norm{\theta}_{H^{\alpha}(\Omega)}^{2+4\alpha})\norm{\mathcal{G}_{\hat{\beta}}(\theta)-\mathcal{G}_{\beta_0}(\theta_0)}_{L^2(\Omega)}\\
    &+T^{-2\beta_0}\abs{\beta_0-\hat{\beta}}^{\frac{\alpha+2}{\alpha}} +|T^{-\beta_0}-T^{-\hat{\beta}}|^{\frac{\alpha+2}{\alpha}}\Big)
    \end{aligned}
\end{equation*}
with a constant $c>0$ independent of $T$. 
Plugging (\ref{Fractionalfinal1}) into the above inequality, we obtain
\begin{equation*}
    \begin{aligned}
    P_{\beta_0,\theta_0}^{(N)}\hat{Q}\norm{q_\theta - q_0}_{L^2(\Omega)}^{\frac{\alpha+2}{\alpha}\cdot\frac{2}{q+1}}\leq &C_2\varepsilon_N^{\frac{2}{q+1}}\log N+c\Big(T^{-\beta_0}\abs{\beta_0-\hat{\beta}} +|T^{-\beta_0}-T^{-\hat{\beta}}|\Big)^{\frac{\alpha+2}{\alpha}\cdot\frac{2}{q+1}}.
    \end{aligned}
\end{equation*}
\end{proof}
\subsection{Inverse medium scattering problem}
In the followings, we assume that
\[n \in L^{\infty}(\mathbb{R}^3), \ \mbox{supp}(1-n) \subset D,\ 0\leq n \leq M_0 \ \mbox{with}\ M_0 \geq 1.\]
It is true that $u_n^s$ satisfies the following problem \cite[(1.51)-(1.52)]{cakoni2022inverse}:
\begin{equation}\label{Scat}
    \left\{\begin{aligned}
    &-\Delta w-{\omega}^2nw = {\omega}^2(n-1)u^i \quad \mbox{in}\ B_R,\\
    &\partial w/\partial r = S_R(w|_{\partial B_R}) \quad \mbox{on}\  \partial B_R,
    \end{aligned}\right.  
\end{equation}
where $B_R$ is an open ball of radius $R$ such that $\Bar{D} \subset B_R$. Here, the Dirichlet-to-Neumann map $S_R$ : $H^{1/2}(\partial B_R) \rightarrow H^{-1/2}(\partial B_R)$ is defined by
\[S_R : v \rightarrow \frac{\partial u_v}{\partial r}|_{\partial B_R}\quad \mbox{for}\ v \in H^{1/2}(\partial B_R),\]
where $u_v$ is the solution of (\ref{Helm}) satisfying the Sommerfeld radiation condition in $R^3\setminus B_R$ and the Dirichlet condition $u_v = v$ on $\partial B_R$.
It has been proved that $S_R$ is a bounded  bijective operator, and further that
\begin{equation}\label{SRbnd}
Re\langle S_R(v), v \rangle \leq 0 \quad \mbox{and} \quad Im\langle S_R(v),v \rangle \geq 0,\quad \forall  v \in H^{1/2}(\partial B_R),
\end{equation}
where $\langle \cdot,\cdot \rangle$ denotes the duality pairing between $H^{-1/2}(\partial B_R)$ and $H^{1/2}(\partial B_R)$, see \cite[(14)]{bourgeois2012remark}.
In addition, it is shown in \cite[1.22]{cakoni2022inverse} that the far field pattern have the following explicit expression:
\begin{equation}\label{ffpre}
    \begin{aligned}
        u_n^{\infty}(\hat{x},\vartheta) &= \frac{1}{4\pi}\int_{\partial B_R}\bbra{u^s(y,\vartheta)\frac{\partial e^{-\im {\omega} \hat{x}\cdot y}}{\partial r(y)}-\frac{\partial u^s}{\partial r}(y,\vartheta)e^{-\im {\omega} \hat{x}\cdot y}}ds(y)\\
        &= \frac{{\omega}^2}{4\pi}\int_{D}e^{-\im{\omega}\hat{x}\cdot y}(n-1)(y)u(y,\vartheta)dy,
    \end{aligned}
\end{equation}
where $u(u,\vartheta) = u^i(y,\vartheta) + u_q^s(y,\vartheta)$ with $u^i(y,\vartheta) = e^{\im{\omega} y\cdot\vartheta}$.

Now, we consider a general form of (\ref{Scat}) by replacing the right-hand side of the first equation with a source  $f \in L^2(B_R)$, i.e.,
\begin{equation}\label{GScat}
        \left\{\begin{aligned}
        &-\Delta w-{\omega}^2nw = f \quad \mbox{in}\ B_R,\\
        &\partial w/\partial r = S_R(w|_{\partial B_R}) \quad \mbox{on}\  \partial B_R.
        \end{aligned}\right.  
\end{equation}
\begin{lemma} \label{GScabnd}
For problem (\ref{GScat}), we have
\begin{equation} \label{ieqScat}
    \norm{w}_{H^1(B_R)} \leq C\norm{f}_{H^{-1}(B_R)},
\end{equation}
and
\begin{equation}
    \norm{w}_{H^2(B_R)} \leq C\norm{f}_{L^2(B_R)},
\end{equation}
where $C = C(D,{\omega},M_0)$.
\end{lemma}
\begin{proof}
The problem above is equivalent to the following variational formulation: for all $v \in H^1(B_R)$,
\begin{equation*}
    a_1(w,v) + a_2(w,v) = F(v),
\end{equation*}
where
\begin{equation*}
    a_1(w,v) = \int_{B_R} \nabla w\cdot \nabla \Bar{v} dx - {\omega}^2\int_{B_R}nw\Bar{v}dx + {\omega}^2M_0\int_{B_R}w\Bar{v}dx - \langle S_R(w),\Bar{v} \rangle,
\end{equation*}
\begin{equation*}
    a_2(w,v) = -{\omega}^2M_0\int_{B_R}w\Bar{v}dx,
\end{equation*}
and
\begin{equation*}
    F(v) = \int_{B_R} f\Bar{v}dx.
\end{equation*}
Since $Re\langle S_R(v), v \rangle \leq 0$ and $n \leq M_0$, we deduce that
\[Re\ a_1(w,w) \geq \int_{B_R}(\abs{\nabla w}^2 + {\omega}^2(M_0-n)\abs{w}^2)dx \geq C({\omega})\norm{w}_{H^1(B_R)}^2,\]
which implies $a_1(\cdot,\cdot)$ is strictly coercive. Furthermore, $a_1(\cdot,\cdot)$ is bounded because $S_R$ is a bounded operator.
Combining the Lax-Milgram theorem and the Riesz representation theorem, there exists a bounded invertible linear operator $\mathcal{A}: H^1(B_R) \rightarrow (H^1(B_R))^*$
such that\[a_1(w,v) = \mathcal{A}w(v).\]

Define an operator $I : H^1(B_R) \rightarrow (H^1(B_R))^*$ as
\[Iw(v) = \int_{B_R}w\Bar{v}dx.\] By the Sobolev compact embedding, $I$ is compact.
We see that $a_2(w,v) = -{\omega}^2 M_0 Iw(v)$. Then, (\ref{GScat}) is equivalent to finding $w \in H^1(B_R)$ such that
\[\mathcal{A}w-{\omega}^2M_0 Iw = F,\]
i.e.,
\[w-{\omega}^2 M_0\mathcal{A}^{-1}Iw = \mathcal{A}^{-1}F.\]
Since $\mathcal{A}^{-1}I$ is compact, by the Fredholm alternative, the unique solution to $\mathcal{A}w-{\omega}^2 {\color{black}M_0} Iw = F$ exists for every $f \in H^{-1}(B_R)$ provided the kernel of $\mathcal{A} -{\omega}^2 M_0 I$ is trivial, which follows from the uniqueness of the scattered solution (see page 10 in \cite{bourgeois2012remark}).
In other words, $\mathcal{A}-{\omega}^2 M_0I$ is bounded and invertible.
Then we have
\begin{equation*}
    \norm{w}_{H^1(B_R)} \leq C\norm{f}_{H^{-1}(B_R)}
\end{equation*}
where $C = C(D,{\omega},M_0)$. Since $\Delta w = -f+{\omega}^2nw$, we have
\begin{equation*}
    \norm{\Delta w}_{L^2(B_R)} \leq C \norm{f}_{L^2(B_R)} + \norm{w}_{L^2(B_R)} \leq C \norm{f}_{L^2(B_R)},
\end{equation*}
that is,
\begin{equation*}
    \norm{w}_{H^2(B_R)} \leq C \norm{f}_{L^2(B_R)} + \norm{w}_{L^2(B_R)} \leq C \norm{f}_{L^2(B_R)}
\end{equation*}
for $C = C(D,{\omega},M_0)$.
\end{proof}
\begin{lemma}\label{Scrabnd}
\begin{equation}\label{ffbnd}
     \norm{u_n^{\infty}}_{\infty} \leq  C(1 + \norm{n}_{L^2} + \norm{n}_{L^2}^2)
\end{equation}
for $C = C(D,{\omega},M_0)$.
\end{lemma}
\begin{proof}
Let $f = {\omega}^2(n-1)u^i$ and $w = u_n^s$, the above lemma \ref{GScabnd} implies
\begin{equation}\label{usbnd}
    \norm{u_n^s(\cdot,\vartheta)}_{H^1(B_R)} \leq C\norm{n-1}_{(H^1(D))^{*}}
\end{equation}
uniformly in $\vartheta \in \mathcal{S}^2$ for $C = C(D,{\omega},M_0)$.
Using (\ref{usbnd}) and (\ref{ffpre}), we have 
\begin{equation*}
    \begin{aligned}
       \abs{u_n^{\infty}(\hat{x},\vartheta)} &\leq \frac{{\omega}^2}{4\pi} \norm{n-1}_{L^2}(\abs{D}+\norm{u_n^s(\cdot,\vartheta)}_{L^2(D)}) \\
       &\leq C(\norm{n-1}_{L^2} + \norm{n-1}_{L^2}^2)\\
       &\leq C(1 + \norm{n}_{L^2} + \norm{n}_{L^2}^2)
    \end{aligned}
\end{equation*}
for $C = C(D,{\omega},M_0)$, that is,
\begin{equation*}
    \norm{u_n^{\infty}}_{\infty} \leq C(1 + \norm{n}_{L^2} + \norm{n}_{L^2}^2).
\end{equation*}
\end{proof}
\begin{lemma}\label{Scralip}
\begin{equation} \label{Scalip}
    \norm{u_{n_1}^{\infty}-u_{n_2}^{\infty}}_{\infty} \leq C(1 + \norm{n_1}_{L^2}^2\vee\norm{n_2}_{L^2}^2)\norm{n_2 - n_1}_{L^2}
\end{equation}
for $C = C(D,{\omega},M_0)$.
\end{lemma}
\begin{proof}
Using lemma \ref{GScabnd} with $w = u_{n_1}^s-u_{n_2}^s$ and $f = {\omega}^2(n_2- n_1)u_{n_1}^s + {\omega}^2(n_2- n_1)u^i$, we deduce that
\begin{equation*}
    \begin{aligned}
        \norm{u_{n_1}^s-u_{n_2}^s}_{H^1(B_R)}&\leq C\norm{{\omega}^2(n_2- n_1)u_{q_1}^s + {\omega}^2(n_2- n_1)u^i}_{(H^1(D))^{*}}\\
        &\leq C(\norm{n_2 - n_1}_{\infty}\norm{u_{n_1}^s}_{L^2(D)} + \norm{n_2 - n_1}_{\infty}).
    \end{aligned}
\end{equation*}
Then, applying (\ref{usbnd}) to the above inequality gives
\begin{equation}\label{usstab}
    \norm{u_{n_1}^s-u_{n_2}^s}_{H^1(B_R)}\leq C(1+\norm{n_1}_{L^2})\norm{n_2 - n_1}_{\infty}.
\end{equation}
for $C = C(D,{\omega},M_0)$.
From the expression of $u_q^{\infty}$, we have
\begin{equation*}
    \begin{aligned}
      \abs{u_{n_1}^{\infty}-u_{n_2}^{\infty}} &\leq
      \frac{{\omega}^2}{4\pi}\int_D\abs{(n_2-n_1)e^{\im{\omega}\vartheta\cdot y} +(n_2 - n_1)u_{n_1}^s - (n_2-1)(u_{n_1}^s-u_{n_2}^s)}dy\\
      &\leq \frac{{\omega}^2}{4\pi}\left[(\abs{D} + \norm{u_{n_1}^s}_{L^2(D)})\norm{n_2 - n_1}_{L^2} + \norm{n_2}_{L^2}\norm{u_{n_1}^s-u_{n_2}^{s}}_{L^2(D)}\right].
    \end{aligned}
\end{equation*}
Using (\ref{usbnd}) and (\ref{usstab}) to the above inequality, we finally have
\begin{equation*}
    \begin{aligned}
      \abs{u_{n_1}^{\infty}-u_{n_2}^{\infty}}
      &\leq C(1 + \norm{n_1}_{L^2}+\norm{n_2}_{L^2}+\norm{n_1}_{L^2}\norm{n_2}_{L^2})\norm{n_2 - n_1}_{L^2}
    \end{aligned}
\end{equation*}
for $C = C(D,{\omega},M_0)$, that is
\begin{equation*}
    \norm{u_{n_1}^{\infty}-u_{n_2}^{\infty}}_{\infty} \leq C(1 + \norm{n_1}_{L^2}^2\vee\norm{n_2}_{L^2}^2)\norm{n_2 - n_1}_{L^2}.
\end{equation*}
Next, we recall the following stability estimate
\begin{lemma}\label{Scrastab}
    Let $t>3/2, M>0,$ and $0<\delta<\frac{2t-3}{2t+3}$ be given constants. Assume that $1-n \in H^t(\mathbb{R}^3)$ satisfying $\norm{1-n_j}_{H^t(R^3)}\leq M$ and supp$(1-n)\subset D$, $j=1,2$. Then 
    \begin{equation*}
        \norm{n_1-n_2}_{L^{\infty}(D)} \leq C[-\ln^{-}(\norm{u_{n_1}^{\infty}-u_{n_2}^{\infty}}_{L^2(\mathcal{S}^2\times\mathcal{S}^2)})]^{-(\frac{2t-3}{2t+3}-\delta)},
    \end{equation*}
    where $C = C(D,t,{\omega},M,\varepsilon)$ and 
    \[\ln^{-}(z)=\left\{\begin{array}{ll}
       \ln(z)  &  \text{if\ } z\leq e^{-1},\\
       -1  &  \text{otherwise.\ }
    \end{array}\right.\]
\end{lemma}
\end{proof}
Then we prove Theorem \ref{mainScra} and Corollary \ref{CoroScra}.
\begin{proof}[Proof of Theorem \ref{mainScra}]
We verify Conditions \ref{loosecondreg} and \ref{loosecondstab} for $\mathcal{G}$ as defined in (\ref{forwardmapScra}) with $\mathcal{R} = C^{\alpha}(D)$.
\par
\textbf{Condition \ref{loosecondreg}}: by Lemmas \ref{Scrabnd} and \ref{Scralip}, combined with Lemma 29 in \cite{IntroNonLinear_nickl2020convergence}, we have
\begin{equation*}
    \norm{\mathcal{G}(\theta)}_{\infty} \leq  C, \quad \forall \theta \in B_{\mathcal{R}}(M)
\end{equation*}
\begin{equation*}
    \norm{\mathcal{G}(\theta_1)-\mathcal{G}(\theta_2)}_{L^2_{\lambda}(\mathcal{S}^2\times\mathcal{S}^2,\mathbb{R}^2)} \leq C\norm{\theta_2 - \theta_1}_{L^2(D)}, \quad \forall \theta_1,\theta_2 \in B_{\mathcal{R}}(M)
\end{equation*}
for $C = C(D,{\omega},M_0)$, which together verify Condition \ref{loosecondreg}.
\par
\textbf{Condition \ref{loosecondstab}}: Using Lemma \ref{Scrastab} with $t=\alpha$ and $0<\delta< a(\alpha): =\frac{2\alpha-3}{2\alpha+3}$, we deduce that
\begin{equation*}
    \mexp{-C{\norm{n_{\theta}-n_0}_{L^{\infty}(D)}^{\frac{1}{-a(\alpha)+\delta}}}} \leq \norm{u_{n_{\theta}}^{\infty}-u_{n_0}^{\infty}}_{L^2(\mathcal{S}^2\times\mathcal{S}^2)},\quad \forall \theta \in B_{\mathcal{R}}(M).
\end{equation*}
for $C = C(D,{\omega},\delta,M)$.
\par
Above all, we apply Theorem \ref{finalthmsvZ} and obtain
\begin{equation*}
    P_{\theta_0}^{(N)}\hat{Q}\mexp{-C{\norm{n_{\theta}-n_0}_{L^{\infty}(D)}^{\frac{1}{-a(\alpha)+\delta}}}}\lesssim \varepsilon_N^{2}\log N.
\end{equation*}
for $C = C(D,{\omega},\delta,B)$.
\end{proof}
\begin{proof}[Proof of Corollary \ref{CoroScra}]
We first denote that
\[L(\theta,\theta_0) = \mexp{-C{\norm{n_{\theta}-n_0}_{L^{\infty}(D)}^{\frac{1}{-a(\alpha)+\delta}}}},\]
where the right side of the equation is as in the result of Theorem \ref{mainScra}.
Apply Markov’s inequality to the result of Theorem \ref{mainScra}, we have
\begin{equation*}
P_{\theta_0}^{(N)}\hat{Q}\bbra{L(\theta,\theta_0)>\varepsilon_N^{2}\log^2 N}\leq \frac{P_{\theta_0}^{(N)}\hat{Q}L(\theta,\theta_0)}{\varepsilon_N^{2}\log^2 N}\lesssim \frac{1}{\log N} \rightarrow 0.
\end{equation*}
Plugging in the explicit forms of $\varepsilon_N$ and $L(\theta,\theta_0)$, we have
\begin{equation*}
P_{\theta_0}^{(N)}\hat{Q}\bbra{\norm{n_{\theta}-n_0}_{L^{\infty}(D)}>C(\log N)^{-\frac{2\alpha-3}{2\alpha+3}+\delta}} \rightarrow 0
\end{equation*}
for any $0<\delta<\frac{2\alpha-3}{2\alpha+3}$ and $C=C(\alpha,D,{\omega},\delta,B).$
\end{proof}
{\color{black}\section{Discussion about computation} \label{sec:computation} 
In this section, we will discuss how to compute the Gaussian mean-field family $\mathcal{Q}_G^J$ we proposed in \eqref{GMF} of the main text. We first note that Theorem \ref{finalthmsv} stated in the main text is not restricted to the specific Gaussian mean-field family $\mathcal{Q}_G^J$ proposed in this work. In fact, $\mathcal{Q}_G^J$ serves only as a sufficiently rich variational class for which the approximation error $\gamma_N^2$ can be controlled by the convergence rate $\varepsilon_N^2$. Theorem \ref{finalthmsv} is applicable to any variational class $\mathcal{Q}$ that contains $\mathcal{Q}_G^J$. In the context of inverse problems, many applications of variational inference methods fall within the scope of our theorem. More generally, in our view, there are two main research directions for variational inference methods in inverse problems:
\begin{enumerate}
\item \emph{Parameterizing measures directly:} This approach consists of directly parameterizing the approximating measure and optimizing the associated parameters to approximate the posterior distribution. For example, when restricting attention to Gaussian approximations, a novel Robbins--Monro algorithm was developed from a calculus-of-variations perspective \cite{pinski2015algorithms,pinski2015kullback}. Under the classical mean-field assumption, a general variational inference framework in separable Hilbert spaces was proposed in \cite{Jia2023JMLR,sui2024non}.
\item \emph{Parameterizing transformations:} By parameterizing a transformation, a simple reference measure can be mapped to a more complex measure in order to approximate the posterior distribution. The optimization of the transformation parameters then drives the approximation process. For example, the infinite-dimensional Stein variational gradient descent method introduced in \cite{jia2022stein} can be viewed as a function-space particle optimization approach with rigorous mathematical foundations in separable Hilbert spaces. Moreover, the normalizing flow based variational inference algorithm developed in \cite{zhao2025functionalnormalizingflowstatistical} maps Gaussian measures to complex distributions via functional normalizing flows.
\end{enumerate}

Variational inference methods introduced in \cite{pinski2015algorithms,pinski2015kullback,jia2022stein,zhao2025functionalnormalizingflowstatistical} fall within the scope of our theorem, since their variational classes contain $\mathcal{Q}_G^J$. Moreover, these methods are computationally tractable, as demonstrated by the numerical examples provided in the corresponding works. Specifically, in \cite{pinski2015algorithms} and \cite{pinski2015kullback}, the posterior is approximated by Gaussian distributions that are equivalent to the prior. With the wavelet prior $\Pi_N$ considered in our work, their variational family is broader than $\mathcal{Q}_G^J$. 
In \cite{jia2022stein,zhao2025functionalnormalizingflowstatistical}, the prior $\mu_0$ is transformed through a nonlinear mapping $T$ to approximate the posterior. This approach gives rise to the specific forms of their variational sets $\mathcal{Q}$:
\[
    \mathcal{Q} = \set{\nu | \nu = \mu_0 \circ T^{-1}, T\text{ is a restricted transform}}.
\]
The Functional Normalizing Flow introduced in \cite{zhao2025functionalnormalizingflowstatistical} allows the transformation $T$ to be decomposed into a sequence of mappings of the form $I + \varphi_n$, where each $\varphi_n$ takes values in a finite-dimensional subspace of the reproducing kernel Hilbert space (RKHS) induced by the prior. From the explicit expression of $\varphi_n$ given in \cite{zhao2025functionalnormalizingflowstatistical}, the composed transformation $T$ can represent linear mappings. Therefore, with the prior $\Pi_N$ considered in our work, the variational set in \cite{zhao2025functionalnormalizingflowstatistical} contains $\mathcal{Q}_G^J$.
Similarly for the iSVGD proposed in \cite{jia2022stein}, the transformation $T$ can also be decomposed into a series of transformation with form $I + \varphi_n$ with $\varphi_n$ lies in some RKHS. In principle, the composed mapping $T$ contains the linear mapping, which ensures that the variational set contains $\mathcal{Q}_G^J$ as a special case \cite{guella2022operator}. The mean-field assumptions considered in \cite{Jia2023JMLR,sui2024non} mainly concern hyperparameters in the prior and noise models, rather than the components of the unknown parameters themselves. Consequently, the approaches in \cite{Jia2023JMLR,sui2024non} can hardly be directly incorporated into the current theoretical framework. How to generalize the present framework to accommodate settings involving hyper-parameters remains an interesting direction for future research.
\par
As for the Gaussian mean-field family 
$\mathcal{Q}_G^J$ proposed in our work, although to our knowledge there may be no existing numerical results applying it to inverse problems of partial differential equations, this type of variational family is widely adopted in neural network contexts, and its computational methods have been extensively studied \cite{NIPS2011_7eb3c8be,pmlr-v119-swiatkowski2020a,NEURIPS2022_8ea50bf4,NEURIPS2020_310cc7ca}. Following the work \cite{NIPS2011_7eb3c8be}, we give some discussion about how we compute the variational set $\mathcal{Q}_G^J$. First, we introduce a linear operator
\begin{equation*}
    \Psi (\tilde{\theta}) =  \sum_{l=-1}^{\infty}\sum_{r\in R_l}\tilde{\theta}_{lr}\chi\psi_{lr}
\end{equation*}
defined on $$\tilde{\Theta}_{\alpha}=\set{\tilde{\theta}= (\tilde{\theta}_{lr}): \tilde{\theta}_{lr}\in \mathbb{R}, r\in R_l , l \in \mathbb{N}\cup\{-1,0\},  \sum_{l=-1}^{\infty}\sum_{r\in R_l}2^{2l\alpha}\theta_{lr}^2<\infty}.$$
Obviously, $\Psi$ is a linear operator from  $\tilde{\Theta}_{\alpha}$ to $H^{\alpha}$.
Let us define 
$$\tilde{\Pi}_N = \mathop{\bigotimes}_{l=-1}^J\mathop{\bigotimes}_{r\in R_l}N(0,(N\varepsilon_N^2)^{-1}).$$ 
We can see the original prior $\Pi_N = \tilde{\Pi}_N \circ \Psi^{-1}$ which is a pushforward measure of  $\tilde{\Pi}_N$ through the linear operator $\Psi$.
With $\Psi$, we consider independent and identically distributed (i.i.d.) random variables $(Y_i,X_i)_{i=1}^N$ of the following random design regression model
\begin{align*}
    Y_i = \mathcal{G}(\Psi(\tilde{\theta}))(X_i) + \varepsilon_i, \quad \varepsilon_i \mathop{\sim}^{iid} N(0,I_V), \quad i = 1, 2, \dots, N.
\end{align*}
For $\theta_0 \in H^{\alpha}$, assume data $D_N$ are drawn from the model
$$Y_i = \mathcal{G}(\Psi(\tilde{\theta}_0))(X_i) + \epsilon_i$$ where $\tilde{\theta}_0 = (\tilde{\theta}_{0,lr})=(\pdt{\theta_0}{\psi_{lr}}_{L^2(\mathcal{Z})})$.
Instead of considering the optimization problem on measures in function space:
$$ \hat{Q} = \mathop{\mathrm{argmin}}_{Q\in \mathcal{Q}_G^J} D(Q\Vert \Pi(\theta|D_{N})),$$
we turn to a more tractable optimization problem for the computation on a finite dimensional vector space:
$$ \hat{Q}_v = \mathop{\mathrm{argmin}}_{\tilde{Q}\in \tilde{\mathcal{Q}}_G^J} D(\tilde{Q}\Vert \tilde{\Pi}(\tilde{\theta}|D_{N})),$$
where 
$\tilde{\mathcal{Q}}_G^J=\Bigg\{\mathop{\bigotimes}_{l=-1}^J\mathop{\bigotimes}_{r\in R_l}N(\mu_{lr},\sigma_{lr}^2):\mu_{lr} \in \mathbb{R}, \sigma_{lr}^2 \geq 0\Bigg\}$ is a standard Gaussian mean-field family.
We see that the alternative optimization problem is an upper bound of the original: $$\mathop{\mathrm{min}}_{Q\in \mathcal{Q}_G^J} D(Q\Vert \Pi(\theta|D_{N})) \leq \mathop{\mathrm{min}}_{\tilde{Q}\in \tilde{\mathcal{Q}}_G^J} D(\tilde{Q}\Vert \tilde{\Pi}(\tilde{\theta}|D_{N}))$$ because $D(Q\Vert \Pi(\cdot|D_{N})) = D(\tilde{Q} \circ \Psi^{-1}\Vert \tilde{\Pi}(\cdot|D_{N})\circ \Psi^{-1}) \leq D(\tilde{Q}\Vert \tilde{\Pi}(\cdot|D_{N}))$ for any $Q\in \mathcal{Q}_G^J$.
In fact, for the solution $\hat{Q}_v$ of the alternative problem, we can still prove that its pushforward measure $\hat{Q}=\hat{Q}_v \circ \Psi^{-1} $ (with slight abuse of the symbol $\hat{Q}$) shares the same convergence rates as the original one, that are 
\begin{gather*}
      P_{\theta_0}^{(N)}\hat{Q}\Vert \mathcal{G}(\theta)-\mathcal{G}(\theta_0)\Vert _{L^{2}_{\lambda}}^{\frac{2}{p+1}}\lesssim \varepsilon_N^{\frac{2}{p+1}}\log N.\\
      P_{\theta_0}^{(N)}\hat{Q}[F(\Vert f_{\theta} - f_{\theta_0}\Vert )]^{\frac{2}{p+q+1}}\lesssim \varepsilon_N^{\frac{2}{p+q+1}}\log N.
  \end{gather*}
The proof follows the proof of Theorem 3.7 by transferring those results from $\Pi_N$ to $\tilde{\Pi}_N$ through the linear operator $\Psi$.
\par
Now we are going to discuss how to compute the alternative problem. Our regression model can be generalized as 
$$Y_i = h(\tilde{\theta})(X_i) + \epsilon_i, $$
where $h$ is a nonlinear operator from $\mathbb{R}^{d_J}$ to a function space. Our goal is to solve the optimization problem $\hat{Q}_v = \mathop{\mathrm{argmin}}_{\tilde{Q}\in \tilde{\mathcal{Q}}_G^J} D(\tilde{Q}\Vert \tilde{\Pi}(\tilde{\theta}|D_{N}))$. This type of problems widely appears in neural network settings where $h$ is a neural network with nonlinear activation functions and $\tilde{\theta}$ stands for weights of the network \cite{NIPS2011_7eb3c8be}. The computation approaches to this optimization problem have also been well studied. Following the procedure introduced in \cite{NIPS2011_7eb3c8be}, we reinterpreted the KL divergence $D(\tilde{Q}\Vert \tilde{\Pi}(\tilde{\theta}|D_{N}))$ as 
\begin{align*}
    &D(\tilde{Q}\Vert \tilde{\Pi}(\tilde{\theta}|D_{N}))\\ = &\int \log\frac{d\tilde{Q}}{d\tilde{\Pi}(\cdot|D_N)}d\tilde{Q} \\
    =& \int \log \frac{d\tilde{Q}}{d\tilde{\Pi}_N}d\tilde{Q}- \int \sum_{i=1}^{N}\log p_{\tilde{\theta}}(y_i,x_i)d\tilde{Q} +\log\int\prod_{i=1}^Np_{\tilde{\theta}}(y_i,x_i)d\tilde{\Pi}_N\\
    =&D(\tilde{Q}\Vert\tilde{\Pi}_N) -E_{\tilde{Q}}\sum_{i=1}^{N}\log p_{\tilde{\theta}}(y_i,x_i) + I(y_i,x_i).
\end{align*}
Since the integral term $I(y_i,x_i)$ is independent of $\tilde{Q}$, it is sufficient to consider
$$\mathop{\mathrm{min}}_{\tilde{Q}\in \tilde{\mathcal{Q}}_G^J} D(\tilde{Q}\Vert\tilde{\Pi}_N) -E_{\tilde{Q}}\sum_{i=1}^{N}\log p_{\tilde{\theta}}(y_i,x_i).$$
Because $\tilde{Q}_{\boldsymbol{\mu},\boldsymbol{\sigma}^2}$ is parametrized by $\boldsymbol{\mu}=(\mu_{lr})$ and $\boldsymbol{\sigma}^2=(\sigma_{lr}^2)$, we can transform the optimization problem into
$$\mathop{\mathrm{min}}_{\boldsymbol{\mu},\boldsymbol{\sigma}^2} L(\boldsymbol{\mu},\boldsymbol{\sigma^2},D_N) =\mathop{\mathrm{min}}_{\boldsymbol{\mu},\boldsymbol{\sigma}^2}L^C(\boldsymbol{\mu},\boldsymbol{\sigma}^2) + L^E(\boldsymbol{\mu},\boldsymbol{\sigma}^2,D_N),$$where we follow the symbol of \cite{NIPS2011_7eb3c8be} that $$L^C(\boldsymbol{\mu},\boldsymbol{\sigma}^2):= D(\tilde{Q}_{\boldsymbol{\mu},\boldsymbol{\sigma}^2}\Vert\tilde{\Pi}_N), \quad L^E(\boldsymbol{\mu},\boldsymbol{\sigma}^2,D_N):= -E_{\tilde{Q}}\sum_{i=1}^{N}\log p_{\tilde{\theta}}(y_i,x_i).$$ In order to minimize $L(\boldsymbol{\mu},\boldsymbol{\sigma}^2,D_N)$ with gradient descent, we need to derive the gradients of $L^E(\boldsymbol{\mu},\boldsymbol{\sigma}^2,D_N)$ and $L^C(\boldsymbol{\mu},\boldsymbol{\sigma}^2)$ with respect to $\boldsymbol{\mu},\boldsymbol{\sigma}^2$. Derivatives of multivariate Gaussian expectations have the following identities \cite{opper2009variational}:
\begin{gather*}
    \nabla_{\boldsymbol{\mu}}E_{\boldsymbol{a}\sim {N}(\boldsymbol{\mu},\boldsymbol{\Sigma})}[V(\boldsymbol{a})] = E_{\boldsymbol{a}\sim {N}(\boldsymbol{\mu},\boldsymbol{\Sigma})}[\nabla_{\boldsymbol{a}}V(\boldsymbol{a})],\\\nabla_{\boldsymbol{\Sigma}}E_{\boldsymbol{a}\sim {N}(\boldsymbol{\mu},\boldsymbol{\Sigma})}[V(\boldsymbol{a})] = \frac{1}{2}E_{\boldsymbol{a}\sim {N}(\boldsymbol{\mu},\boldsymbol{\Sigma})}[\nabla_{\boldsymbol{a}}\nabla_{\boldsymbol{a}}V(\boldsymbol{a})],
\end{gather*}
where $V$ is an arbitrary function of $\boldsymbol{a}$. Differentiating $L^E(\boldsymbol{\mu},\boldsymbol{\sigma}^2,D_N)$ and applying these identities yields
\begin{align*}
    \frac{\partial L^E(\boldsymbol{\mu},\boldsymbol{\sigma}^2,D_N)}{\partial \mu_{lr}} =& -E_{\tilde{\mathcal{Q}}}\sum_{i=1}^N  \frac{\partial \log p_{\tilde{\theta}}(y_i,x_i)}{\partial \tilde{\theta}_{lr}}  = E_{\tilde{\mathcal{Q}}}\sum_{i=1}^N \frac{1}{2} \frac{\partial \abs{y_i-\mathcal{G}(\Psi(\tilde{\theta}))(x_i)}_V^2}{\partial \tilde{\theta}_{lr}} \\ =& E_{\tilde{\mathcal{Q}}}\sum_{i=1}^N \pdt{\mathcal{G}(\Psi(\tilde{\theta}))(x_i)-y_i}{D\mathcal{G}_{\tilde{\theta}}(\Psi_{lr}(\tilde{\theta}_{lr}))(x_i)}_V,\\
    \frac{\partial L^E(\boldsymbol{\mu},\boldsymbol{\sigma}^2,D_N)}{\partial \sigma^2_{lr}} = & -\frac{1}{2} E_{\tilde{\mathcal{Q}}}\sum_{i=1}^N \frac{\partial^2 \log p_{\tilde{\theta}}(y_i,x_i)}{\partial \tilde{\theta}_{lr}^2} \approx \frac{1}{2}E_{\tilde{\mathcal{Q}}}\sum_{i=1}^N  \left(\frac{\partial \log p_{\tilde{\theta}}(y_i,x_i)}{\partial \tilde{\theta}_{lr}}\right)^2 \\
    \approx & \frac{1}{2}E_{\tilde{\mathcal{Q}}}\sum_{i=1}^N \pdt{\mathcal{G}(\Psi(\tilde{\theta}))(x_i)-y_i}{D\mathcal{G}_{\tilde{\theta}}(\Psi_{lr}(\tilde{\theta}_{lr}))(x_i)}_V^2,
\end{align*}
where $\Psi_{lr}(\tilde{\theta}_{lr}) = \chi \tilde{\theta}_{lr}\psi_{lr}$, $D\mathcal{G}_{\tilde{\theta}}$ denotes the Fréchet derivative of $\mathcal{G}$ at $\Psi{(\tilde{\theta})}$ and the approximation comes from substituting the diagonal of the empirical Fisher information matrix for the diagonal of the Hessian. The expectation with respect to $\tilde{Q}$ can be calculated by applying Monte-Carlo integration. Since the prior ${\tilde{\Pi}_N}=\mathop{\bigotimes}_{l=-1}^J\mathop{\bigotimes}_{r\in R_l}N(0,(N\varepsilon_N^2)^{-1})$, we have
\begin{align*}
    &L^C(\boldsymbol{\mu},\boldsymbol{\sigma}^2) = \sum_{l=-1}^J\sum_{r\in R_l} -\log {\sqrt{N\varepsilon_N^2}\sigma_{lr}} + \frac{N\varepsilon_N^2}{2}[\mu_{lr}^2 + \sigma_{lr}^2 - (N\varepsilon_N^2)^{-1}]\\
    &\frac{\partial L^C(\boldsymbol{\mu},\boldsymbol{\sigma}^2)}{\partial \mu_{lr}} = N\varepsilon_N^2\mu_{lr},\qquad \frac{\partial L^C(\boldsymbol{\mu},\boldsymbol{\sigma}^2)}{\partial {\sigma}_{lr}^2} = \frac{1}{2}\left(N\varepsilon_N^2 - \frac{1}{{\sigma}_{lr}^2}\right).
\end{align*}
In general, we can compute the gradient as
\begin{align*}
    &\frac{\partial L(\boldsymbol{\mu},\boldsymbol{\sigma^2},D_N)}{\partial \mu_{lr}} = N\varepsilon_N^2\mu_{lr} + \frac{1}{S}\sum_{k=1}^{S}\sum_{i=1}^N \pdt{\mathcal{G}(\Psi(\tilde{\theta}^{(k)}))(x_i)-y_i}{D\mathcal{G}_{\tilde{\theta}}(\Psi_{lr}(\tilde{\theta}_{lr}^{(k)}))(x_i)}_V\\
    &\frac{\partial L(\boldsymbol{\mu},\boldsymbol{\sigma^2},D_N)}{\partial {\sigma}_{lr}^2} = \frac{1}{2}\left(N\varepsilon_N^2 - \frac{1}{{\sigma}_{lr}^2}\right) +\frac{1}{2S}\sum_{k=1}^{S}\sum_{i=1}^N \pdt{\mathcal{G}(\Psi(\tilde{\theta}^{(k)}))(x_i)-y_i}{D\mathcal{G}_{\tilde{\theta}}(\Psi_{lr}(\tilde{\theta}_{lr}^{(k)}))(x_i)}_V^2,
\end{align*}
where $\tilde{\theta}^{(k)}$ are drawn independently from $\tilde{Q}$. Then, we can compute the optimization problem $\hat{Q}_v = \mathop{\mathrm{argmin}}_{\tilde{Q}\in \tilde{\mathcal{Q}}_G^J} D(\tilde{Q}\Vert \tilde{\Pi}(\tilde{\theta}|D_{N}))$ with gradient descent.}

{\color{black} We note that the number of parameters in our optimization problem grows with the truncation level $J$. Given that an increasing sample size $N$ necessitates a higher truncation level $J$ to maintain theoretical convergence,  we are here to explain the reason why our computation is scalable respect to the parameters dimension \cite{ghattas_learning_2021,bui-thanh_computational_2013}, which ensures a computational advantage in practice compared to standard function-space MCMC methods as $N$ grows. Since the computation cost of the gradients $\frac{\partial L(\boldsymbol{\mu},\boldsymbol{\sigma^2},D_N)}{\partial \mu_{lr}}, \frac{\partial L(\boldsymbol{\mu},\boldsymbol{\sigma^2},D_N)}{\partial {\sigma}_{lr}^2}$  is dominated by solving the PDEs, it is standard practice to assess the scalability of sampling algorithms in terms of the number of PDE solves \cite{amestoy_exploiting_2019,jia2022stein,ghattas_learning_2021,bui-thanh_computational_2013}. {Therefore,  we focus on the computation of the terms $\mathcal{G}(\Psi(\tilde{\theta}^{(k)}))$ and $D\mathcal{G}_{\tilde{\theta}}(\Psi_{lr}(\tilde{\theta}_{lr}^{(k)}))$} (the index $k$ is omitted in the following analysis for simplicity).
\begin{enumerate}[label=(\alph*)]
    \item The first term, $\mathcal{G}(\Psi(\tilde{\theta}))$, requires a single PDE solve, { so its computation is scalable with $J$}.
    \item The second term, $D\mathcal{G}_{\tilde{\theta}}(\Psi_{lr}(\tilde{\theta}_{lr}))$, is the Fréchet derivative applied to each $\Psi_{lr}(\tilde{\theta}_{lr})$, Since $D\mathcal{G}_{\tilde{\theta}}$ is a linear operator, we can compute $D\mathcal{G}_{\tilde{\theta}}(\Psi_{lr}(\tilde{\theta}_{lr}))$ simultaneously using a single factorization of the stiffness matrix obtained from the finite element method; {consequently, the overall computation is scalable with $J$.}
\end{enumerate}
We see that our computation is scalable respect to the parameters dimension, since the number of PDE solves required is independent with $J$. We illustrate (b) with a concrete example. Consider a following Darcy flow problem as in \cite{IntroNonLinear_nickl2023bayesian}:
\begin{align*}
    \left\{\begin{aligned}
    &\nabla\cdot (e^{\theta} \nabla u) = g \quad \mbox{on } \mathcal{X},\\
    &u = 0 \quad \mbox{on } \partial\mathcal{X},
    \end{aligned}\right.  
\end{align*}
with the PDE operator $\mathcal{G}(\theta) = u$.
Then, the derivative $v_{lr}:=D\mathcal{G}_{\tilde{\theta}}(\Psi_{lr}(\tilde{\theta}_{lr}))$ solves the equation:
\begin{align*}
    \left\{\begin{aligned}
    &\nabla\cdot (e^{\Psi(\tilde{\theta})} \nabla v_{lr}) = -\nabla\cdot (e^{\Psi(\tilde{\theta})}\Psi_{lr}(\tilde{\theta}_{lr})\nabla u_{\tilde{\theta}} ) \quad \mbox{on } \mathcal{X},\\
    & v_{lr} = 0 \quad \mbox{on } \partial\mathcal{X}.
    \end{aligned}\right.  
\end{align*}
where $u_{\tilde{\theta}}= \mathcal{G}(\Psi(\tilde{\theta}))$. Using a finite element basis $\set{\phi_1,\phi_2,\dots,\phi_n}$ on $\mathcal{X}$, Solving for $v_{lr}$ reduces to a system of linear equations:
\begin{align}\label{AX=B}
     A[\textbf{v}_{1,1},\textbf{v}_{1,2},\dots,\textbf{v}_{J,c2^{Jd}}]=[\textbf{F}_{1,1},\textbf{F}_{1,2},\dots,\textbf{F}_{J,c2^{Jd}}]
\end{align}
where $\textbf{v}_{l,r}$ is the coefficient vector of $v_{lr}$ corresponding to the finite element basis, and
\[
    \textbf{F}_{l,r} = \big(\int_{\mathcal{X}} \nabla\cdot (e^{\Psi(\tilde{\theta})}\Psi_{lr}(\tilde{\theta}_{lr})\nabla u_{\tilde{\theta}} )\phi_idx\big)_{i=1}^n,\quad A_{ij} = \int_{\mathcal{X}} \nabla \phi_{i}\cdot (e^{\Psi(\tilde{\theta})}\nabla \phi_{j})dx.
\]
The computation cost of obtaining all $v_{lr}$ is dominated by the factorization of matrix $A$ that is uniform for all $l,r$ \cite{duff_direct_2017}. Explicitly, $T_{\rm total}= T_{\rm factor} + cJ2^{Jd}\cdot T_{\rm solve}$ where $T_{\rm solve}$ is the solving time for a single right-hand side (RHS) vector, with $T_{\rm solve}\ll T_{\rm factor}$. Moreover, solving for all $v_{lr}$ simultaneously can be handled as a matrix $[\textbf{v}_{1,1},\textbf{v}_{1,2},\dots,\textbf{v}_{J,c2^{Jd}}]$, which can further reduces the computation cost \cite{amestoy_exploiting_2019}. Hence, the overall computation is scalable with $J$. Moreover, the computation of the matrix equation (\ref{AX=B}) can be readily parallelized, which may further improve computational efficiency.

In conclusion, the computation cost of the high-dimensional parameter optimization problem is scalable respect to the parameters dimension, just as in standard function-space MCMC methods \cite{ghattas_learning_2021}. Furthermore, the variational inference optimization problem benefits from parallel computation and subsampling, resulting in lower time costs compared to function-space MCMC algorithms that often suffer from poor mixing \cite{pinski2015algorithms,pinski2015kullback,jia2022stein,zhao2025functionalnormalizingflowstatistical} (as discussed above, the variational inference methods introduced in \cite{pinski2015algorithms,pinski2015kullback,jia2022stein,zhao2025functionalnormalizingflowstatistical} also fall within the scope of our theorem). Therefore, solving the high-dimensional parameter optimization problem remains computationally advantageous in practice compared to standard function-space MCMC methods as $N$ grows.}
\bibliographystyle{siamplain}
\bibliography{references}
